\documentclass{m2an}
\usepackage{amssymb,amsmath,color}
\usepackage{graphicx,amsthm}
\usepackage{amsmath}
\allowdisplaybreaks
\usepackage{subfigure,txfonts}
\usepackage[none]{hyphenat}
\usepackage{latexsym,bm}
\usepackage{booktabs}
\usepackage{cases}
\usepackage{lineno}

\usepackage{geometry}
\usepackage{mathtools}

\newtheorem{The}{Theorem}[section]
\newtheorem{lem}{Lemma}[section]
\newtheorem{Rem}{Remark}[section]
\newtheorem{Ass}{Assumption}[section]

\begin{document}

\title{ Error analysis of the highly efficient and energy stable schemes for the 2D/3D two-phase MHD }
\author{Ke Zhang}\address{College of Mathematics and System Sciences, Xinjiang University, Urumqi 830046, China. Email: zhangkemath@139.com}
\author{Haiyan Su}
\address{Corresponding author. College of Mathematics and System Sciences, Xinjiang University, Urumqi 830046, China. Email: shymath@126.com. }

\author{Xinlong Feng}\address{College of Mathematics and System Sciences, Xinjiang University, Urumqi 830046, China. Email: fxlmath@gmail.com/fxlmath@xju.edu.cn}
%
%
\begin{abstract}
In this paper, we mainly focus on the rigorous convergence analysis of two fully decoupled, unconditionally energy-stable methods for the diffuse interface two-phase magnetohydrodynamics (MHD) model.
The two methods consist of the semi-implicit stabilization method and the invariant energy quadratization (IEQ) method, which are both applied to the phase field system. In addition, the pressure correction method is used for the saddle point system, and appropriate implicit-explicit treatments are employed for the nonlinear coupled terms. We prove the unconditional energy stability of the two schemes. In addition, we mainly establish the error estimates based on the bounds of $\left\|\phi^{k}\right\|_{L^{\infty}}$ and $\left\|\textbf{b}^{k}\right\|_{L^{\infty}}$. Several numerical examples are presented to test the accuracy and stability of the proposed methods.
\end{abstract}
%
%
\subjclass{65N12, 65N30, 65N50, 35Q79}

\keywords{Two-phase MHD model; Fully decoupled methods; Unconditional energy stable; Error estimates.}

\maketitle
 
\section{Introduction}

The magneto-hydrodynamical (MHD) system models the behaviors of conducting fluids, such as plasmas, liquid metals, salt water, and electrolytes, in an external electromagnetic field \cite{2000Liquid, 2016Numerical}. The diffuse interface two-phase fluid dynamics is a new branch of physics that studies the flow behavior of two-phase systems \cite{2003ALC, 2010Level, 2012Effect}, which is famous for capturing the evolution of the interface. In this paper, we study the characteristics of fluid dynamics in a mixture of two incompressible, immiscible, conducting fluids under the external electromagnetic field, where the situation is referred to as two-phase MHD. It is primarily applicable to chemical reactors, liquid-metal applications, magnetic pumps, and some other related fields \cite{1964EFFECT, 2022Highly, 2021SharpZX}. To the best of our knowledge, the diffuse interface two-phase MHD model  was first proposed in \cite{2019AYang}.




The diffuse interface two-phase MHD model is governed  by the phase field equation for the free interface, the Navier-Stokes equations for hydrodynamics, and the Maxwell's equations for electromagnetism through convection, stresses, generalized Ohm's law and Lorentz forces \cite{2019AYang, 2022Highly}. And the  existence of weak solutions has been established in \cite{2019AYang} by using the fixed point theorem and the compactness method.  This model poses significant challenges in the development of highly efficient numerical methods, including: (1) the  nonlinear terms, e.g., $(\textbf{v}\cdot \nabla) \textbf{v}$ and $f(\phi)$; (2) the coupled terms, e.g., $\nabla\cdot(\phi \textbf{v})$ and $\nabla\times(\textbf{v}\times\textbf{b})$; (3) the incompressible constraint $\nabla\cdot \textbf{u}=0$, which leads to a saddle point system; and (4) the stiffness of the phase equations, which is  linked to the interfacial width.

It is remarkable that several  attempts have been made in this direction recently.  In \cite{2019AYang}, the authors developed a first-order, fully-coupled method that combines a semi-implicit scheme with a convex splitting technique to ensure energy stability. The first-order semi-implicit stabilization method and the IEQ method, both combining the projection method, were proposed  in \cite{2022Highly}. These methods achieve fully decoupled unconditional energy stability. Similarly, based on the Gauge-Uzawa scheme, both the first-order, fully-decoupled semi-implicit stabilization method and the IEQ method have been proposed in \cite{zhang2023gauge}, and they both satisfy the discrete energy law. The second-order IEQ weakly decoupled method \cite{SU2023107126} and the second-order coupled methods \cite{AAMM-13-761}, which include both time-discrete and fully-discrete schemes, were proposed for solving the diffuse interface two-phase MHD model.

On the one hand, several error estimates were derived mainly based on the coupling scheme for the diffuse interface two-phase MHD model.  For instance, in \cite{2022Unconditional}, convergence analyses were presented for a coupled first-order, semi-implicit stabilized method in a semi-discrete case and a fully-discrete case.  In \cite{2025Error}, error analyses were given for two weakly decoupled, first-order, time-discrete schemes. The presented scheme is weakly decoupled, meaning only  the  velocity field and pressure field are decoupled, while the phase field and the magnetic field are tightly coupled. The error analysis of a coupled, second-order Crank-Nicolson  time-discrete scheme was given in \cite{AAMM-13-761}. A fully-decoupled second-order scalar auxiliary variables (SAV) scheme, based on pressure correction  and the zero-energy-contribution (ZEC) method, was presented in \cite{2024A}. It is noteworthy that the error analysis was conducted only for the first-order scheme.

On the other hand, there exists a limited literature on unconditionally convergent error estimates for fully-decoupled schemes of the diffuse interface two-phase MHD model. For example, the study in \cite{2023ErrorQIU} derived conditional convergence estimates using a semi-implicit discretization framework with a convex-splitting scheme. The error analyses were carried out for the pressure field under the restricted condition $\Delta t\lesssim h$, and for the phase field, velocity field and magnetic field when $\Delta t\leq \Delta t_{0}$ and mesh size $h\leq h_{0}$.  In \cite{0Convergence}, under the restricted condition $\Delta t\leq C$, the convergence analysis was derived for the phase field in the $L^{\infty}(0,T; H^{1})$ norm and for the velocity and the magnetic fields in the $L^{\infty}(0,T; L^{2})$ norm, for a coupled fully-discrete, second-order modified Crank-Nicolson scheme.  Recently,  the unconditional convergence analyses of two fully-decoupled schemes were presented in \cite{2025Unconditionally}, based on the first-order time-discrete scheme and fully-discrete ZEC scheme.

We further explored the error analysis of the fully-decoupled, first-order  scheme, focusing on its unconditional convergence. Specifically, the time-discrete semi-implicit stabilization method and the IEQ scheme are applied to the phase field system, while the pressure correction scheme is used for the incompressible constraint, and the implicit-explicit treatments are employed for nonlinear coupled terms. For the diffuse interface two-phase MHD model, the fully-decoupled scheme will pose  significant difficulties for the analysis of the coupled terms. To address this problem, our main contributions in the paper include the following:
\vspace {5pt}
\begin{itemize}
\item We apply the mathematical induction method to prove the bounds of $\left\|\phi^{k}\right\|_{L^{\infty}}$ and $\left\|\textbf{b}^{k}\right\|_{L^{\infty}}$, which are crucial for our convergence analysis \cite{2019FullyZGD, 2020Convergence}.
\vspace {5pt}
\item The error estimates for both the fully-decoupled time-discrete semi-implicit stabilization method and the IEQ method are unconditionally valid, imposing no restrictions on the time step size or mesh size.
\vspace {5pt}
\item To verify the energy stability and convergence of our scheme, several numerical examples are provided.
\vspace {5pt}
\end{itemize}

This paper is organized as follows. In Section 2, some preliminary results and the two-phase MHD  model are reviewed. The fully-decoupled, unconditional energy stable semi-implicit stabilization algorithm is introduced in Section 3, and its convergence analysis is established in detail. Section 4 provides the error estimate for the unconditional energy stable IEQ scheme. Several numerical examples are presented in Section 5, including a smooth solution, spinodal decomposition, and Boussinesq approximation tests. In Section 6, we make some conclusions.

\section{The two-phase MHD model and notations}

We consider the following diffuse interface two-phase MHD model \cite{2019AYang, 2022Highly, SU2023107126}:
\begin{subequations}\label{2-1}
\begin{align}
&\phi_{t}+\nabla\cdot(\phi \textbf{v})=M \Delta w, \quad \text{in}\ \Omega\times (0,T], \label{mdel1}\\
&w=-\varepsilon \Delta\phi+f(\phi),\quad \text{in}\ \Omega\times (0,T], \label{model2}\\
&\textbf{v}_{t}+(\textbf{v}\cdot \nabla) \textbf{v}- \nu\Delta \textbf{v}+\frac{1}{\mu}\textbf{b}\times\nabla\times\textbf{b}+\nabla p+ \lambda\phi\nabla w=\textbf{f},\quad \text{in}\ \Omega\times (0,T], \label{model3}\\
&\nabla \cdot \textbf{v}=0,\quad \text{in}\ \Omega\times (0,T],\label{model4}\\
&\textbf{b}_{t}+ \frac{1}{\sigma\mu}\nabla\times(\nabla\times\textbf{b})- \nabla\times(\textbf{v}\times\textbf{b}) = \textbf{0},\quad \text{in}\ \Omega\times (0,T], \label{model5}\\
&\nabla \cdot\textbf{b}=0,\quad \text{in} \ \Omega\times (0,T], \label{model6}
\end{align}
\end{subequations}
where $\Omega$ is a bounded, convex polygon  or polyhedron domain in $R^{d}$ ($d=2, 3$) with a Lipschitz boundary, and $T>0$ denotes the termination time.  Let the symbols ($\phi$, $w$, $\textbf{v}$, $p$, $\textbf{b}$)  represent the phase field,  chemical potential, velocity field, pressure field, and magnetic field, respectively. In addition, the parameters $\nu$, $\mu$, $\lambda$ and $\sigma$ stand for the kinematic viscosity, magnetic permeability, capillary coefficient, and electric conductivity, respectively. The parameter $\varepsilon$ represents the interface thickness between the two fluids, and  $M$ is the mobility parameter.

The equations \eqref{2-1} are supplemented with the following initial conditions:
\begin{equation}\label{2-2b-initial}
\begin{aligned}
&\phi|_{t=0}=\phi_0,\quad \textbf{v}|_{t=0}=\textbf{v}_0,\quad \textbf{b}|_{t=0}=\textbf{b}_0,
\end{aligned}
\end{equation}
and the  corresponding boundary conditions:
\begin{equation}\label{2-2b-boundary}
\begin{aligned}
&\frac{\partial\phi}{\partial \textbf{n}}|_{\partial\Omega}=0, \quad \frac{\partial w}{\partial \textbf{n}}|_{\partial\Omega}=0, \quad \textbf{v}|_{\partial\Omega}=\textbf{0},\quad \textbf{b}\times \textbf{n}|_{\partial\Omega}=\textbf{0}.
\end{aligned}
\end{equation}
The phase field $\phi$ represents the  mixture of two immiscible, incompressible fluids, which can be presented as 
\begin{equation}\label{2-1c}
\phi ( \textbf{x}, t)=\left\{
\begin{aligned}
-1, \qquad \rm fluid\, 1,\\
1, \qquad \rm fluid\, 2.
\end{aligned}
\right.
\end{equation}

It should be noted that $f(\phi)$=$F^{\prime}(\phi)$, where the Ginzburg-Landau double-well type potential is defined as $F(\phi)=\frac{1}{4\varepsilon}(\phi^{2}-1)^{2}$. And we can extend the potential function $F(\phi)$ to the entire domain \cite{1958Free, 2010Energy} as
\begin{equation*}
F(\phi)=\left\{
\begin{aligned}
&\frac{1}{\varepsilon} (\phi+1)^{2}, \qquad \phi\in (-\infty, -1],\\
&\frac{1}{4\varepsilon} (\phi^2-1)^{2}, \quad \phi\in [-1, 1],\\
&\frac{1}{\varepsilon} (\phi-1)^{2}, \qquad \phi\in [1, +\infty),\\
\end{aligned}
\right.
\end{equation*}
assuming that $|F^{\prime\prime}(\phi)|\leq C_{1}=2/\varepsilon$, for all $\phi\in R$.
{\Rem
The Flory-Huggins logarithmic potential is another popular choice for the potential functional \cite{2019The} as
\begin{equation*}
F(\phi)=\frac{1+\phi}{2} \ln  \left(\frac{1+\phi}{2}\right)+\frac{1-\phi}{2} \ln \left(\frac{1-\phi}{2}\right)+\frac{\vartheta}{4}(\phi^{2}-1)^{2},
\end{equation*}
where $\vartheta> 2$ is the energy parameter.
}

Let $L^{m}(\Omega)$ denote the usual Lebesgue space on $\Omega$, equipped with the norm $\|\cdot\|_{L^{m}}$. We introduce the $L^{2}$ norm  $\|\textbf{v}\|=(\textbf{v},\textbf{v})^{1/2}$, and the inner product $(\textbf{u}, \textbf{v})=\int_{\Omega}\textbf{u}\cdot\textbf{v}\mbox{d}\textbf{x}$, where $\textbf{u}$ and $\textbf{v}$ are two vector functions. $W^{k,m}(\Omega)$ stands for the standard Sobolev spaces  defined on $\Omega$, equipped with the standard Sobolev norms $\|\cdot\|_{k,m}$. Let $H^{k}(\Omega)$ denote  $W^{k,2}(\Omega)$, with  the corresponding norm being $\|\cdot\|_{H^{k}}$. We introduce the following standard Sobolev spaces 
\begin{align*}
& H^{1} (\Omega)=\{\phi\in L^{2}(\Omega):\nabla\phi\in L^{2}(\Omega)^{d}\},\\
&\textbf{H}^{1}_{0} (\Omega)=\{\textbf{u}\in H^{1} (\Omega)^{d}: \textbf{u}|_{\partial\Omega}=\mathbf{0}\},\\
& L^{2}_{0}(\Omega)=\{p\in L^{2}(\Omega):\int_{\Omega}p\mbox{d}\textbf{x}=0\},\\
&\textbf{H}_{\tau}^{1}( \Omega)=\{\textbf{b}\in H^{1}(\Omega)^{d}:  \textbf{n}\times \textbf{b}|_{\partial \Omega}=\mathbf{0}\}.
\end{align*}

It is well known that the embedding inequalities \cite{1975Sobolev}, i.e.,
\begin{equation}\label{4-e}
\begin{aligned}
&\|\phi\|_{L^{m}}\leq C\|\phi\|_{H^{1}}, \, \phi\in H^{1}(\Omega), \, 2\leqslant m\leqslant 6,\\
& \|\textbf{q}\| \leqslant C\|\nabla \textbf{q}\|, \, \textbf{q}\in \textbf{H}^{1}_{0}(\Omega),\\
& \|\textbf{x}\|_{H^{1}} \leqslant C\|\nabla \cdot \textbf{x}\|+C\|\nabla\times \textbf{x}\|, \, \textbf{x}\in \textbf{H}^{1}_{\tau}(\Omega),\\
&\|\textbf{v}\|_{ L^{m}} \leqslant C\|\textbf{v}\|_{H^{1} }, \,  \textbf{v}\in \textbf{H}_{0}^{1}(\Omega), \, 2\leqslant m\leqslant 6,\\
&\|\textbf{v}\|_{L^{3}} \leqslant C\|\textbf{v}\|^{1/2}\|\textbf{v}\|_{H^{1}}^{1/2}, \, \textbf{v} \in H^{1}(\Omega)^{d},
\end{aligned}
\end{equation}
where $C$ is a generic coefficient that is independent of $\Delta t$ and  different at different occurrences.

We realize that $\|\nabla\times \textbf{s}\|^{2}+\|\nabla\cdot \textbf{s}\|^{2}=\|\nabla \textbf{s}\|^{2}$,  for all $\textbf{s}\in\textbf{H}^{1}_{0}(\Omega) $, as shown in  \cite{2019Stability}. The trilinear form $b(\cdot,\cdot,\cdot)$ is defined for all $\textbf{v}, \textbf{q}, \textbf{s} \in \textbf{H}^{1}_{0}(\Omega)$  as  
\begin{equation*}
b(\textbf{v},\textbf{q},\textbf{s})=( (\textbf{v}\cdot\nabla ) \textbf{q},\textbf{s}), \quad  b(\textbf{s},\textbf{q},\textbf{q})=\textbf{0},
\end{equation*}
and the following inequalities hold 
\begin{equation}
b(\textbf{v},\textbf{q},\textbf{s})\leqslant
\left\{
\begin{aligned}
&C\|\textbf{v}\|_{H^{1}}\|\textbf{q}\|_{H^{1}}\|\textbf{s}\|_{H^{1}},\\
&C\|\textbf{v}\|\|\textbf{q}\|_{H^{2}}\|\textbf{s}\|_{H^{1}},\\
&C\|\textbf{v}\|_{H^{1}}\|\textbf{q}\|_{H^{2}}\|\textbf{s}\|.\\
\end{aligned}
\right.
\end{equation}

\begin{The} Assuming that the source term  $\textbf{f}$=$\textbf{0}$, the two-phase MHD model (\ref{2-1})-(\ref{2-2b-boundary}) follows the energy dissipation law
$$
\frac{d}{dt}E(\phi, \textbf{v}, \textbf{b}) \leq 0,
$$
where the energy function $E(\phi, \textbf{v}, \textbf{b})$ is given by
\begin{equation}
E(\phi, \textbf{v}, \textbf{b})=\frac{\varepsilon\lambda}{2} \|\nabla  \phi\|^{2}+\frac{1}{2} \|\textbf{v}\|^{2}+\frac{1}{2\mu} \|\textbf{b}\|^{2}+ \lambda \int_{\Omega}F(\phi)\mbox{d}\textbf{x}.
\end{equation}
\end{The}

\begin{proof}
%
A similar proof can be found in \cite{2022Highly}.
\end{proof}

\section{A Semi-Implicit Stabilization Method }

In this section, we primarily present the error estimates of the fully decoupled and semi-implicit stabilization scheme for solving the two-phase MHD model (\ref{2-1})-(\ref{2-2b-boundary}). We first recall the following semi-implicit stabilization scheme \cite{2022Highly}.

\subsection{The semi-implicit stabilization algorithm (Scheme I)}

Let $\Delta t> 0$ be the  time step size, and $T$=$k\Delta t$, where  $0\leq k\leq [\frac{T}{\Delta t}]$. Given the initial values ($\phi^{0}$, $\textbf{v}^{0}$, $p^{0}$, $\textbf{b}^{0}$), we compute  ($\phi^{k}$, $\textbf{v}^{k}$, $p^{k}$, $\textbf{b}^{k}$) from the following scheme. Besides, we denote $\delta A^{k}$=$A^{k}-A^{k-1}$ to simplify our notation.

\textbf{Step 1}: Compute   $\phi^{k}$ and $w^{k}$ from 
\begin{subequations}\label{s-1}
\begin{align}
&\frac{\phi^{k}- \phi^{k-1}}{\Delta t}+\nabla\cdot \left(\phi^{k-1} \textbf{v}^{k-1}\right)-\Delta t\lambda\nabla\cdot \left(\left(\phi^{k-1}\right)^{2}\nabla w^{k}\right)=M\Delta w^{k}, \label{3-1a}\\
&w^{k}=-\varepsilon \Delta\phi^{k}+f\left(\phi^{k-1}\right)+S\left(\phi^{k}- \phi^{k-1}\right), \label{3-1b} \\
&\frac{\partial\phi^{k}}{\partial \textbf{n}}|_{\partial\Omega}=0, \quad \frac{\partial w^{k}}{\partial \textbf{n}}|_{\partial\Omega}=0.
\end{align}
\end{subequations}

\textbf{Step 2}: Compute  $\textbf{b}^{k}$ from 
\begin{subequations}\label{s-2}
\begin{align}
&\frac{\textbf{b}^{k}- \textbf{b}^{k-1}}{\Delta t} +\frac{1}{\sigma\mu}\nabla\times\left(\nabla\times\textbf{b}^{k}\right)-\nabla\times \left(\textbf{v}_{\star}^{k-1}\times\textbf{b}^{k-1}\right)=\textbf{0}, \label{3-2a} \\
&\frac{\textbf{v}_{\star}^{k-1}- \textbf{v}^{k-1}}{\Delta t}+\frac{1}{\mu}\textbf{b}^{k-1}\times \nabla\times\textbf{b}^{k}=\textbf{0}, \label{3-2b} \\
&\textbf{b}^{k}\times \textbf{n}|_{\partial\Omega}=\textbf{0}.
\end{align}
\end{subequations}

\textbf{Step 3}: Compute  $\tilde{\textbf{v}}^{k}$ from 
\begin{subequations}\label{s-3}
\begin{align}
&\frac{\tilde{\textbf{v}}^{k}- \textbf{v}_{\star}^{k-1}}{\Delta t}+\left(\textbf{v}^{k-1}\cdot \nabla\right) \tilde{\textbf{v}}^{k}-\nu \Delta \tilde{\textbf{v}}^{k} +\nabla p^{k-1}+ \lambda \phi^{k-1}\nabla w^{k}=\textbf{f}^{k}, \label{3-3a} \\
&\tilde{\textbf{v}}^{k}|_{\partial\Omega}=\textbf{0}.
\end{align}
\end{subequations}

\textbf{Step 4}: Compute  $\textbf{v}^{k}$ from 
\begin{subequations}\label{s-4}
\begin{align}
&\frac{ \textbf{v}^{k}-\tilde{\textbf{v}}^{k}}{\Delta t}+ \nabla p^{k}-\nabla p^{k-1}=\textbf{0},\label{3-4a}  \\
&\nabla\cdot \textbf{v}^{k}=0, \\
&\textbf{v}^{k}\cdot \textbf{n}|_{\partial\Omega}= 0.
\end{align}
\end{subequations}

{\Rem \label{3-1} We state that all variables are fully decoupled and provide the following explanations.

 (i) By linking with the   equations  (\ref{3-1a})-(\ref{3-1b}), we obtain the following equation 
\begin{equation*}\label{remark3-1}
\begin{aligned}
\frac{\delta\phi^{k}}{\Delta t}+\nabla\cdot \left(\phi^{k-1} \textbf{v}^{k-1}\right)-\Delta t\lambda\nabla\cdot\left(\left(\phi^{k-1}\right)^{2}\nabla \left(-\varepsilon \Delta\phi^{k}+f\left(\phi^{k-1}\right)+S\left(\delta\phi^{k}\right)\right)\right)=M\Delta w^{k}.\\
\end{aligned}
\end{equation*}

(ii) Then, we obtain
$ \textbf{v}_{\star}^{k-1}= \textbf{v}^{k-1}-  \frac{\Delta t}{\mu}\textbf{b}^{k-1}\times \nabla\times\textbf{b}^{k}$  from  (\ref{3-2b}).  Substituting this into  (\ref{3-2a}), we have 
\begin{equation*}
\frac{\delta\textbf{b}^{k} }{\Delta t} +\frac{1}{\sigma\mu}\nabla\times\left(\nabla\times\textbf{b}^{k}\right)-\nabla\times \left(\left(\textbf{v}^{k-1}-  \frac{\Delta t}{\mu}\textbf{b}^{k-1}\times \nabla\times\textbf{b}^{k}\right)\times\textbf{b}^{k-1}\right)=\textbf{0}.
\end{equation*}

(iii) Additionally, by taking the divergence of equation (\ref{3-4a}),  we obtain 
\begin{equation*}
\Delta p^{k}= \frac{\nabla\cdot \tilde{\textbf{v}}^{k}}{ \Delta t} + \Delta p^{k-1},
\end{equation*}
and we can update $\textbf{v}^{k}$ using
$\textbf{v}^{k}= \tilde{\textbf{v}}^{k}-\Delta t \left(\delta\nabla p^{k} \right)$. Obviously, the coupled, nonlinear, and saddle point type model has been decomposed into a series
of smaller elliptic type problems.
}

{\Rem \label{3-2}

The first-order  stabilized term $\Delta t\lambda\nabla\cdot \left(\left(\phi^{k-1}\right)^{2}\nabla w^{k}\right)$ is introduced to enhance the stability of the fully explicit term $\nabla\cdot \left(\phi^{k-1} \textbf{v}^{k-1}\right)$. Additionally,  the first-order term $S\left( \phi^{k}- \phi^{k-1} \right)$ is also presented to stabilize $f(\phi^{k-1})$ in the scheme, where $S>0$ is a stability parameter. This stabilizer is critical to maintain the accuracy and improve the energy stability while using large time steps \cite{2021EfficientEF}. 
}

%
%

\begin{The}\label{Theorem3-1}
For the source term $\textbf{f}$=$ \mathbf{0}$, if $S\geq\frac{C_{1}}{2}$, the \textbf{ Scheme I} is unconditionally energy stable in the  sense that 
$$E^{k}-E^{k-1} \leq 0,$$
where
\begin{equation}\label{3-0}
E^{k}=\frac{\lambda\varepsilon}{2}\left\|\nabla\phi^{k}\right\|^{2}+ \lambda \int_{\Omega}F\left(\phi^{k}\right)\mbox{d}\textbf{x}+\frac{1}{2\mu}\left\|\textbf{b}^{k}\right\|^{2}+ \frac{1}{2}\left\|\textbf{v}^{k}\right\|^{2} +\frac{\Delta t^{2}}{2}\left\|\nabla p^{k}\right\|^{2}.
\end{equation}
\end{The}

\begin{proof}
By taking the $L^{2}$ inner product of  equation (\ref{3-1a}) with $\lambda\Delta tw^{k}$, we get 
\begin{equation}\label{3-6}
\lambda \left(\delta\phi^{k}, w^{k}\right)-\lambda\Delta t\left(\phi^{k-1}\textbf{v}^{k-1}, \nabla w^{k}\right)+\Delta t^{2}\lambda^{2}\left\|\phi^{k-1}\nabla w^{k}\right\|^{2}=-\lambda\Delta tM\left\|\nabla w^{k}\right\|^{2}.
\end{equation}

Taking the $L^{2}$ inner product of equation (\ref{3-1b}) with $\lambda\left(\delta\phi^{k} \right)$,  we have 
\begin{equation}\label{3-7}
\lambda\varepsilon \left(\delta(\nabla\phi^{k}),  \nabla\phi^{k}\right)+\lambda S\left\|\delta\phi^{k}\right\|^{2}+\lambda\left(f\left(\phi^{k-1}\right), \delta\phi^{k}\right)=\lambda \left(\delta\phi^{k}, w^{k}\right).
\end{equation}

Taking the $L^{2}$ inner product of equation  (\ref{3-2a})  with $\frac{\Delta t}{\mu}\textbf{b}^{k}$, we obtain 
\begin{equation}\label{3-8}
\frac{1}{2\mu}\left(\delta\left\|\textbf{b}^{k}\right\|^{2}+\left\|\delta\textbf{b}^{k}\right\|^{2}\right)+\frac{\Delta t}{\sigma\mu^{2}}\left\|\nabla\times\textbf{b}^{k}\right\|^{2}-\frac{\Delta t}{ \mu}\left(\textbf{v}_{\star}^{k-1}\times\textbf{b}^{k-1}, \nabla\times\textbf{b}^{k}\right)=0.
\end{equation}

Taking the $L^{2}$ inner product of equation (\ref{3-2b})  with $\Delta t\textbf{v}_{\star}^{k-1}$, we derive 
\begin{equation}\label{3-9}
\frac{1}{2}\left(\left\|\textbf{v}_{\star}^{k-1}\right\|^{2}-\left\|\textbf{v}^{k-1}\right\|^{2}+\left\|\textbf{v}_{\star}^{k-1}-\textbf{v}^{k-1}\right\|^{2}\right)+\frac{ \Delta t}{\mu}\left(\textbf{v}_{\star}^{k-1}\times\textbf{b}^{k-1},  \nabla\times\textbf{b}^{k} \right)=0.
\end{equation}

Taking the $L^{2}$ inner product of equation (\ref{s-3}) with  $\Delta t \tilde{\textbf{v}}^{k}$, we have 
\begin{equation}\label{3-10}
\frac{1}{2}\left(\left\|\tilde{\textbf{v}}^{k}\right\|^{2}-\left\|\textbf{v}_{\star}^{k-1}\right\|^{2}+\left\|\tilde{\textbf{v}}^{k}-\textbf{v}_{\star}^{k-1}\right\|^{2}\right)+\Delta t\nu \left\|\nabla\tilde{\textbf{v}}^{k}\right\|^{2}+\Delta t \left(\nabla p^{k-1}, \tilde{\textbf{v}}^{k}\right)+\Delta t \lambda \left(\phi^{k-1}\nabla w^{k}, \tilde{\textbf{v}}^{k}\right)=0.
\end{equation}

In fact, the formula (\ref{3-4a}) can be reconstructed as
\begin{equation*}
 \textbf{v}^{k}+\Delta t\nabla p^{k}=\tilde{\textbf{v}}^{k} +  \Delta t\nabla p^{k-1},
\end{equation*}
then taking the $L^{2}$ inner product of the above equation with itself, we get 
\begin{equation}\label{3-11}
\frac{1}{2}\left\|\textbf{v}^{k}\right\|^{2}+\frac{\Delta t^{2}}{2}\left\|\nabla p^{k}\right\|^{2}= \frac{1}{2}\left\|\tilde{\textbf{v}}^{k}\right\|^{2}+\frac{\Delta t^{2}}{2}\left\|\nabla p^{k-1}\right\|^{2}+ \Delta t \left( \tilde{\textbf{v}}^{k}, \nabla p^{k-1}\right).
\end{equation}

We consider the following estimate 
\begin{equation}\label{3-12}
\begin{aligned}
\Delta t \lambda \left(\phi^{k-1}\nabla w^{k}, \tilde{\textbf{v}}^{k}\right)-\Delta t\lambda\left(\phi^{k-1}\textbf{v}^{k-1}, \nabla w^{k}\right)&\leq \frac{1}{4}\left\|\tilde{\textbf{v}}^{k}-\textbf{v}^{k-1}\right\|^{2}+ \Delta t^{2} \lambda^{2} \left\|\phi^{k-1}\nabla w^{k}\right\|^{2},\\
\frac{1}{2}\left\|\tilde{\textbf{v}}^{k}-\textbf{v}^{k-1}\right\|^{2}&\leq\frac{1}{2}\left\|\textbf{v}_{\star}^{k-1}-\textbf{v}^{k-1}\right\|^{2}+\frac{1}{2}\left\|\tilde{\textbf{v}}^{k}-\textbf{v}_{\star}^{k-1}\right\|^{2}.\\
\end{aligned}
\end{equation}

According to Taylor expansion, we have 
\begin{equation*}
\begin{aligned}
\delta F(\phi^{k})& =f\left(\phi^{k-1}\right)\delta\phi^{k}+\frac{1}{2}f'\left(\xi\right)\left(\delta\phi^{k}\right)^{2},\\
\left(\frac{\delta F\left(\phi^{k}\right) }{\Delta t},1\right)&=\left(f\left(\phi^{k-1}\right),\frac{\delta\phi^{k}}{\Delta t}\right) +\frac{1}{2\Delta t}\left(f'\left(\xi\right),  \left(\delta\phi^{k}\right)^{2}\right)\\
&\leq\left(f\left(\phi^{k-1}\right),\frac{\delta\phi^{k}}{\Delta t}\right)+\frac{C_{1}}{2\Delta t}\left\|\delta\phi^{k}\right\|^{2}, \quad \text{where } \phi^{k-1} < \xi < \phi^{k}.
\end{aligned}
\end{equation*}
Thus, we obtain 
\begin{equation}\label{3-13}
\lambda \left(f\left(\phi^{k-1}\right), \delta\phi^{k}\right)\geq \lambda \left( \delta F\left(\phi^{k}\right),1 \right) - \frac{  C_{1}}{2}\lambda \left\|\delta\phi^{k}\right\|^{2}.
\end{equation}

By combining equations (\ref{3-6})-(\ref{3-13}), we derive 
\begin{equation}\label{3-14}
\begin{aligned}
&\frac{\lambda\varepsilon}{2}\delta\left\|\nabla\phi^{k}\right\|^{2}+\frac{1}{2\mu}\delta\left\|\textbf{b}^{k}\right\|^{2}+ \frac{1}{2}\delta\left\|\textbf{v}^{k}\right\|^{2}+ \lambda\left(\delta F\left(\phi^{k}\right),1\right)+\frac{\Delta t^{2}}{2}\delta\left\|\nabla p^{k}\right\|^{2} +\lambda\Delta tM\left\|\nabla w^{k}\right\|^{2}+\frac{\lambda\varepsilon}{2}\left\|\delta(\nabla\phi^{k})\right\|^{2}\\
&\qquad \qquad +\left( \lambda S-\frac{\lambda C_{1}}{2}\right)\left\|\delta\phi^{k}\right\|^{2}+\frac{1}{2\mu}\left\|\delta\textbf{b}^{k}\right\|^{2}+\frac{ \Delta t}{\sigma\mu^{2}}\left\|\nabla\times\textbf{b}^{k}\right\|^{2}+\frac{1}{4}\left\|\tilde{\textbf{v}}^{k}-\textbf{v}^{k-1}\right\|^{2}+ \nu\Delta t\left\|\nabla\tilde{\textbf{v}}^{k}\right\|^{2} \leq 0.
\end{aligned}
\end{equation}

Thus, we can demonstrate that the \textbf{Scheme I} is  unconditionally energy stable 
\begin{equation*}
E^{k}-E^{k-1} \leq 0,
\end{equation*}
where $E^{k}$ is defined by equation (\ref{3-0}).
\end{proof}

{\Rem \label{3-3} By applying \textbf{Theorem \ref{Theorem3-1}} and summing up  equation (\ref{3-14}) from \(k=0\) to \(m\) (\(m \leq \frac{T}{\Delta t}\)), we get the stable bound as
\begin{equation*}
\begin{aligned}
&\frac{\lambda\varepsilon}{2}\left\|\nabla\phi^{m}\right\|^{2}+\frac{1}{2\mu}\left\|\textbf{b}^{m}\right\|^{2}+ \frac{1}{2}\left\|\textbf{v}^{m}\right\|^{2} + \lambda\left(F\left(\phi^{m}\right),1\right)+\frac{\Delta t^{2}}{2}\left\|\nabla p^{m}\right\|^{2}\\
&+\mathop{\sum}\limits_{k=0}^{m}  \left[  \lambda\Delta tM\left\|\nabla w^{k}\right\|^{2}+\frac{\lambda\varepsilon}{2}\left\|\delta(\nabla\phi^{k})\right\|^{2}+\left( \lambda S-\frac{\lambda C_{1}}{2}\right)\left\|\delta\phi^{k}\right\|^{2} +\frac{1}{2\mu}\left\|\delta\textbf{b}^{k}\right\|^{2}+\frac{ \Delta t}{\sigma\mu^{2}}\left\|\nabla\times\textbf{b}^{k}\right\|^{2}+ \nu\Delta t\left\|\nabla\tilde{\textbf{v}}^{k}\right\|^{2} +\frac{1}{4}\left\|\tilde{\textbf{v}}^{k}-\textbf{v}^{k-1}\right\|^{2} \right] \\
&\leq\frac{\lambda\varepsilon}{2}\left\|\nabla\phi^{0}\right\|^{2}+\frac{1}{2\mu}\left\|\textbf{b}^{0}\right\|^{2}+ \frac{1}{2}\left\|\textbf{v}^{0}\right\|^{2}+ \lambda\left(F\left(\phi^{0}\right),1\right)+\frac{\Delta t^{2}}{2}\left\|\nabla p^{0}\right\|^{2}\\
&\leq C_{2},
\end{aligned}
\end{equation*}
where $C_{2}$ is a general positive constant. From above inequality, we derive 
\begin{equation}\label{ieq8}
\sum_{k=0}^{m} \lambda\Delta tM \left\|\nabla w^{k}\right\|^{2}\leq C_{2},\quad  \sum_{k=0}^{m}\frac{ \Delta t}{\sigma\mu^{2}} \left\|\nabla\times\textbf{b}^{k}\right\|^{2}\leq C_{2}.
\end{equation}
}
\subsection{Convergence analysis}
In this subsection, we derive the convergence results of the proposed semi-implicit stabilization algorithm.  We define $f\lesssim g$ to mean that there exists  a generic positive constant C such that $f\leq Cg$. We shall use repeatedly the following discrete Gronwall inequality \cite{1990Finite}.

{\lem \label{L4-1}
Let $d_{0}, \alpha_{k}, \beta_{k}, \eta_{k}$ and $\kappa_{k}$ be a sequence of nonnegative numbers for integers $k\geq 0$ such that
\begin{equation*}\alpha_{k}+\Delta t\sum\limits_{n=0}^{k}\beta_{n} \leqslant\Delta t\sum\limits_{n=0}^{k}\kappa_{n}\alpha_{n}+\Delta t\sum\limits_{n=0}^{k}\eta_{n}+d_{0},
\end{equation*}
assume that $\kappa_{n}\Delta t \leq 1$ for all $n$, and set $\zeta_{n}=(1-\kappa_{n}\Delta t)^{-1}$. Then,  for all $k\geqslant 0$,
\begin{equation*}
\alpha_{k}+\Delta t\sum\limits_{n=0}^{k}\beta_{n} \leqslant \exp\left(\Delta t\sum\limits_{n=0}^{k}\zeta_{n}\kappa_{n}\right)\left(\Delta t\sum\limits_{n=0}^{k}\eta_{n}+d_{0}\right).
\end{equation*}
}
Based on \cite{2019FullyZGD}, we employ the following lemma.
{\lem \label{lem2}
Let $c_{\alpha}, c_{\beta}, c_{\gamma}$ be nonnegative numbers, $a_{k}$  be a sequence of nonnegative numbers for $k\geq0$, such that
\begin{equation*}
a_{k}\leq c_{\alpha} +c_{\beta}\Delta t a_{k-1}+c_{\gamma}\Delta t^{2}a_{k-1}^{2}.
\end{equation*}
If $\max\left\{c_{\beta},\sqrt{c_{\gamma}}\right\}D\Delta t\leq 1$, then, for $k\geq 0$,
\begin{equation*}
a_{k-1}\leq D,
\end{equation*}
where D=$\max\left\{a_{0}, c_{\alpha}\right\}+2$.
}

First, we rewrite the equations (\ref{2-1}) as  
\begin{subequations}\label{4-2}
\begin{align}
&\frac{\delta\phi(t_{k})}{\Delta t}+\nabla\cdot\left(\phi(t_{k-1}) \textbf{v}(t_{k-1})\right)-\Delta t\lambda\nabla\cdot \left(\phi(t_{k-1})^{2}\nabla w(t_{k})\right)
=M \Delta w(t_{k})+R_{a}^{k} , \label{3-15a}\\
&-w(t_{k})-\varepsilon \Delta\phi(t_{k})+f\left(\phi(t_{k-1})\right)+S\left(\delta\phi(t_{k}) \right) =R_{b}^{k} ,\label{3-15b}\\
&\frac{\delta\textbf{b}(t_{k})}{\Delta t}+\frac{1}{\sigma\mu}\nabla\times(\nabla\times\textbf{b}(t_{k} ))-\nabla\times(\textbf{v}(t_{k-1})\times\textbf{b}(t_{k-1}))=R_{c}^{k},\label{3-15c}\\
&\frac{\delta\textbf{v}(t_{k})}{\Delta t}-\nu\Delta\textbf{v}(t_{k})+ \left(\textbf{v}(t_{k-1})\cdot\nabla\right) \textbf{v}(t_{k})+\nabla p(t_{k-1}) +\frac{1}{\mu}\textbf{b}(t_{k-1})\times\nabla\times\textbf{b}(t_{k})+ \lambda\phi(t_{k-1} ) \nabla w(t_{k})=R_{d}^{k}, \label{3-15d}\\
&\frac{\textbf{v}(t_{k})-\textbf{v}(t_{k})}{\Delta t}+\delta(\nabla p(t_{k}))=R_{e}^{k},\label{3-15e}
\end{align}
\end{subequations}
where  $R_{a}^{k}, R_{b}^{k}, R_{c}^{k}, R_{d}^{k}, R_{e}^{k}$ are truncation errors defined by 
\begin{equation}
\left\{
\begin{aligned}
&R_{a}^{k} =\frac{\delta\phi(t_{k}) }{\Delta t}-\phi_{t}(t_{k})-\delta(\nabla\cdot\left(\phi(t_{k}) \textbf{v}(t_{k})\right))-\Delta t\lambda\nabla\cdot \left(\phi(t_{k-1})^{2}\nabla w(t_{k})\right),\\
&R_{b}^{k} =-\delta f\left(\phi(t_{k})\right)+S\left(\delta\phi(t_{k}) \right),\\
&R_{c}^{k}=\frac{\delta\textbf{b}(t_{k})}{\Delta t}-\textbf{b}_{t}(t_{k})+\delta(\nabla\times\left(\textbf{v}(t_{k})\times\textbf{b}(t_{k})\right)),\\
&R_{d}^{k}=\frac{\delta\textbf{v}(t_{k})}{\Delta t}-\textbf{v}_{t}(t_{k})- (\delta\textbf{v}(t_{k})\cdot\nabla)\textbf{v}(t_{k}) -\delta(\nabla  p(t_{k}))-\frac{1}{\mu}\left(\delta\textbf{b}(t_{k}))\times\nabla\times\textbf{b}(t_{k})- \lambda(\delta\phi(t_{k} )\right) \nabla w(t_{k}),\\
&R_{e}^{k}=\delta(\nabla p(t_{k})).
\end{aligned}
\right.
\end{equation}
{\Ass \label{ass1}
We assume that the solution ($\phi, \textbf{v}, p, \textbf{b} $) of the continuous problem (\ref{2-1})-(\ref{2-2b-boundary}) satisfies the following regularity assumption 
\begin{equation}
\left\{
\begin{aligned}
&\phi, \phi_{t}, \phi_{tt}\in L^{\infty}(0,T,H^{2}(\Omega)),\quad  w\in L^{\infty}(0,T,H^{2}(\Omega)), \quad w_{t } \in L^{\infty}(0,T, H^{1}(\Omega)),\\
&\textbf{v}\in L^{\infty}(0,T,H^{2}(\Omega)), \quad \textbf{b}\in L^{\infty}(0,T,H^{2}(\Omega)),\quad p\in L^{\infty}(0,T,H^{1}(\Omega)).\\
\end{aligned}
\right.
\end{equation}}

One can easily establish the following estimates for the truncation errors, assuming that the exact solutions are sufficiently smooth or satisfy the above assumptions.
{\lem \label{lemma3-3}
Under the \textbf{Assumption \ref{ass1}}, the truncation errors satisfy
\begin{equation}
\begin{aligned}
\left\|R_{a}^{k}\right\|+\left\|R_{b}^{k}\right\|+\left\|R_{c}^{k}\right\|+\left\|R_{d}^{k}\right\|+\left\|R_{e}^{k}\right\|\lesssim \Delta t, \quad  0  \leq k  \leq \frac{T}{\Delta t}.
\end{aligned}
\end{equation}
}
To derive the error estimates, we denote the error functions as
\begin{equation}
\left\{
\begin{aligned}
&e_{\phi}^{k}=\phi(t_{k})-\phi^{k},\quad e_{w}^{k}=w(t_{k})-w^{k},\quad e_{u}^{k}=\textbf{v}(t_{k})-\textbf{v}^{k},\quad e_{b}^{k}=\textbf{b}(t_{k})-\textbf{b}^{k},\\
&e_{p}^{k}=p(t_{k})-p^{k},\quad \tilde{e}_{u}^{k}=\textbf{v}(t_{k})-\tilde{\textbf{v}}^{k},\quad G_{\phi}^{k}=f(\phi(t_{k}))-f(\phi^{k}).
\end{aligned}
\right.
\end{equation}

By subtracting (\ref{3-1a}) from (\ref{3-15a}), (\ref{3-1b}) from (\ref{3-15b}), (\ref{3-2a}) from (\ref{3-15c}),  (\ref{3-3a}) and (\ref{3-2b})  jointly from (\ref{3-15d}),  (\ref{3-4a}) from (\ref{3-15e}),  we obtain the error equations as follows 
\begin{subequations}\label{4-7}
\begin{align}
&\frac{\delta e_{\phi}^{k} }{\Delta t} +\nabla\cdot \left(\phi(t_{k-1})\textbf{v}(t_{k-1})-\phi^{k-1}\textbf{v}^{k-1}\right)-\Delta t\lambda\nabla\cdot   \left(\phi(t_{k-1})^{2}\nabla w(t_{k})-(\phi^{k-1})^{2}\nabla w^{k}\right)=M\Delta e_{w}^{k}+R_{a}^{k}, \label{3-20a} \\
&-e_{w}^{k}-\varepsilon\Delta e_{\phi}^{k}+G_{\phi}^{k-1}+S\left(\delta e_{\phi}^{k} \right)=R_{b}^{k}, \label{3-20b}\\
&\frac{\delta e_{b}^{k} }{\Delta t}+\frac{1}{\sigma\mu}\nabla\times\left(\nabla\times e_{b}^{k}\right)-\nabla\times \left(\textbf{v}(t_{k-1})\times\textbf{b}(t_{k-1})\right)+\nabla\times\left(\textbf{v}_{\star}^{k-1}\times\textbf{b}^{k-1}\right)=R_{c}^{k},\label{3-20c}\\
&\frac{\tilde{e}_{u}^{k}-e_{u}^{k-1}}{\Delta t}-\nu\Delta\tilde{e}_{u}^{k}+\left(\textbf{v}(t_{k-1})\cdot\nabla\right)\textbf{v}(t_{k})-\left(\textbf{v}^{k-1}\cdot\nabla\right)\tilde{\textbf{v}}^{k}+\nabla e_{p}^{k-1}+\frac{1}{\mu}\textbf{b}(t_{k-1})\times\nabla\times\textbf{b}(t_{k})\nonumber\\
&\qquad -\frac{1}{\mu}\textbf{b}^{k-1}\times\nabla\times\textbf{b}^{k}+\lambda\phi(t_{k-1})\nabla w(t_{k})-\lambda\phi^{k-1}\nabla w^{k}=R_{d}^{k}, \label{3-20d}\\
&\frac{e_{u}^{k}-\tilde{e}_{u}^{k}}{\Delta t}+\delta(\nabla e_{p}^{k})=R_{e}^{k}. \label{3-20e}
\end{align}
\end{subequations}

We consider the following $L^{\infty}$ stabilities of $\phi^{k}$ and $\textbf{b}^{k}$, which play a key role in the error estimates.
{\lem \label{lem3-4}
Under  the \textbf{Assumption \ref{ass1}}, there exists a constant C such that  if $\Delta t\leq C$,
the solution $\phi^{k}$ and $\textbf{b}^{k}$ of the semi-implicit stabilization scheme   satisfy
\begin{equation}\label{lemma3-4}
\left\|\phi^{k}\right\|_{L^{\infty}}\leq \kappa_{\phi}, \quad \left\|\textbf{b}^{k}\right\|_{L^{\infty}}\leq \kappa_{b}, \quad k=0, 1, \cdots, \frac{T}{\Delta t}.
\end{equation}
}
\begin{proof}
Using the mathematical induction method, we prove this lemma in the following steps.

\textbf{Step i}. When $k$=0,  we have
$$\left\|\phi^{0}\right\|_{L^{\infty}}=\left\|\phi (t_{0})\right\|_{L^{\infty}}\leq \kappa_{\phi_{1}},\quad   \left\|\textbf{b}^{0}\right\|_{L^{\infty}}=\left\|\textbf{b}(t_{0})\right\|_{L^{\infty}}\leq \kappa_{b_{1}}.$$

Then, we assume that  $\left\|\phi^{k-1}\right\|_{L^{\infty}}\leq \kappa_{\phi_{2}}$ and   $\left\|\textbf{b}^{k-1}\right\|_{L^{\infty}}\leq \kappa_{b_{2}}$ are established. Next, we provide the proof that  $\left\|\phi^{k}\right\|_{L^{\infty}}\leq \kappa_{\phi_{3}}$ and $\left\|\textbf{b}^{k}\right\|_{L^{\infty}}\leq \kappa_{b_{3}}$, where   $\kappa_{\phi_{1}}, \kappa_{b_{1}}, \kappa_{\phi_{2}}, \kappa_{b_{2}}, \kappa_{\phi_{3}}, \kappa_{b_{3}}$ are general positive constants.

Taking the $L^{2}$ inner product of equation (\ref{3-20a}) with $ \lambda\Delta te_{w}^{k}$ and $ \Delta t\varepsilon e_{\phi}^{k}$, we have 
\begin{equation}\label{4-12}
\left\{
\begin{aligned}
& \lambda\left(\delta e_{\phi}^{k} ,  e_{w}^{k}\right)- \lambda\Delta t\left(\phi(t_{k-1})\textbf{v}(t_{k-1})-\phi^{k-1}\textbf{v}^{k-1},  \nabla e_{w}^{k}\right)\\
& \quad + \Delta t^{2}\lambda^{2}\left(\phi(t_{k-1})^{2}\nabla w(t_{k})- (\phi^{k-1})^{2}\nabla w^{k}, \nabla e_{w}^{k}\right)+ \lambda M\Delta t\left\|\nabla e_{w}^{k}\right\|^{2}= \Delta t\lambda \left(R_{a}^{k}, e_{w}^{k}\right),\\
& \varepsilon\left(\delta e_{\phi}^{k} , e_{\phi}^{k} \right)- \varepsilon\Delta t\left( \phi(t_{k-1})\textbf{v}(t_{k-1})-\phi^{k-1}\textbf{v}^{k-1},  \nabla e_{\phi}^{k} \right)\\
&\quad + \Delta t^{2}\lambda \varepsilon\left(\phi(t_{k-1})^{2}\nabla w(t_{k})- (\phi^{k-1})^{2}\nabla w^{k}, \nabla e_{\phi}^{k}\right)+  M\varepsilon\Delta t\left(\nabla e_{w}^{k}, \nabla e_{\phi}^{k}  \right)= \varepsilon\Delta t\left(R_{a}^{k}, e_{\phi}^{k}  \right).
\end{aligned}
\right.
\end{equation}

Taking the $L^{2}$ inner product of equation (\ref{3-20b}) with $ \lambda\left(\delta e_{\phi}^{k} \right )$ and $ \Delta tMe_{w}^{k}$,  we obtain 
\begin{equation}\label{4-13}
\left\{
\begin{aligned}
&- \lambda\left(\delta e_{\phi}^{k} ,  e_{w}^{k}\right)+\frac{\lambda\varepsilon}{2}\left(\delta\left\|\nabla e_{\phi}^{k}\right\|^{2} +\left\|\delta(\nabla e_{\phi}^{k})\right\|^{2}\right)  + \lambda \left(G_{\phi}^{k-1}, \delta e_{\phi}^{k}  \right)+ \lambda S\left\|\delta e_{\phi}^{k} \right\|^{2}= \lambda \left(R_{b}^{k}, \delta e_{\phi}^{k} \right),\\
& \Delta tM\left\|e_{w}^{k}\right\|^{2}- \Delta t\varepsilon M\left( \nabla e_{\phi}^{k}, \nabla e_{w}^{k} \right)  - \Delta t M\left(e_{w}^{k},  G_{\phi}^{k-1}\right)- \Delta t SM\left(e_{w}^{k}, \delta e_{\phi}^{k} \right)=- \Delta t M \left(R_{b}^{k}, e_{w}^{k}\right).
\end{aligned}
\right.
\end{equation}

Taking the $L^{2}$ inner product of equation (\ref{3-20c}) with $ \Delta te_{b}^{k}$, we have 
\begin{equation}\label{4-14}
\begin{aligned}
&\frac{1}{2}\left(\delta \left\|e_{b}^{k}\right\|^{2} +\left\|\delta e_{b}^{k} \right\|^{2}\right)+\frac{ \Delta t}{\sigma\mu}\left\|\nabla\times e_{b}^{k}\right\|^{2} - \Delta t \left(\textbf{v}(t_{k-1})\times\textbf{b}(t_{k-1})-\textbf{v}_{\star}^{k-1}\times\textbf{b}^{k-1},  \nabla\times e_{b}^{k}\right )= \Delta t \left(R_{c}^{k}, e_{b}^{k} \right).
\end{aligned}
\end{equation}

Taking the $L^{2}$ inner product of  equation (\ref{3-20d})  with $ \Delta t \tilde{e}_{u}^{k}$,  we get 
\begin{equation}\label{4-15}
\begin{aligned}
&\frac{1}{2}\left(\left\|\tilde{e}_{u}^{k}\right\|^{2}-\left\| e_{u}^{k-1}\right\|^{2}+ \left\|\tilde{e}_{u}^{k}- e_{u}^{k-1}\right\|^{2}\right)+\Delta t\nu\left\|\nabla\tilde{e}_{u}^{k}\right\|^{2}+ \Delta t\left((\textbf{v}(t_{k-1})\cdot\nabla)\textbf{v}(t_{k})-(\textbf{v}^{k-1}\cdot\nabla) \tilde{\textbf{v}}^{k} ,\tilde{e}_{u}^{k}\right)  + \Delta t \left(\tilde{e}_{u}^{k}, \nabla e_{p}^{k-1}\right)\\
&\quad + \Delta t\lambda \left(\phi(t_{k-1})\nabla w(t_{k})-\phi^{k-1}\nabla w^{k}, \tilde{e}_{u}^{k} \right) +\frac{ \Delta t}{\mu}\left(\textbf{b}(t_{k-1})\times\nabla\times\textbf{b}(t_{k}) -\textbf{b}^{k-1}\times\nabla\times\textbf{b}^{k} ,\tilde{e}_{u}^{k} \right)= \Delta t\left(R_{d}^{k}, \tilde{e}_{u}^{k}\right).
\end{aligned}
\end{equation}

Taking the $L^{2}$ inner product of  equation (\ref{3-20e}) with itself, we have 
\begin{equation}\label{4-16}
\begin{aligned}
&\frac{1}{2}\left\|e_{u}^{k}\right\|^{2}+\frac{\Delta t^{2}}{2}\left\|\nabla e_{p}^{k}\right\|^{2}+ \Delta t\left(e_{u}^{k}, \nabla e_{p}^{k}\right )\\
&=\frac{1}{2}\left\|\tilde{e}_{u}^{k}\right\|^{2}+\frac{\Delta t^{2}}{2}\left\|\nabla e_{p}^{k-1}\right\|^{2}+\frac{\Delta t^{2}}{2}\left\|R_{e}^{k}\right\|^{2}+ \Delta t\left(\tilde{e}_{u}^{k},  \nabla e_{p}^{k-1}\right)+ \Delta t \left(\tilde{e}_{u}^{k},  R_{e}^{k}\right)+ \Delta t^{2}\left(\nabla e_{p}^{k-1}, R_{e}^{k}\right).
\end{aligned}
\end{equation}

Combining equations (\ref{4-12})-(\ref{4-16}), we obtain 
\begin{equation}\label{33}
\begin{aligned}
& \lambda M\Delta t\left\|\nabla e_{w}^{k}\right\|^{2} +\frac{\varepsilon }{2} \left(\delta \left\|e_{\phi}^{k}\right\|^{2} +\left\|\delta e_{\phi}^{k} \right\|^{2}\right)+\frac{\lambda\varepsilon }{2} \left(\delta \left\|\nabla e_{\phi}^{k}\right\|^{2} +\left\|\delta(\nabla e_{\phi}^{k}) \right\|^{2}\right)+ \lambda S\left\|\delta e_{\phi}^{k} \right\|^{2}+ \Delta tM\left\|e_{w}^{k}\right\|^{2}\\
&+\frac{1}{2}\left(\delta\left\|e_{b}^{k}\right\|^{2} +\left\|\delta e_{b}^{k} \right\|^{2}\right)+\frac{ \Delta t}{\sigma\mu}\left\|\nabla\times e_{b}^{k}\right\|^{2} +\frac{1}{2}\left(\delta \left\|e_{u}^{k}\right\|^{2} + \left\|\tilde{e}_{u}^{k}- e_{u}^{k-1}\right\|^{2}\right)+\Delta t\nu\left\|\nabla\tilde{e}_{u}^{k}\right\|^{2}+\frac{\Delta t^{2}}{2}\left(\delta\left\|\nabla e_{p}^{k}\right\|^{2} \right)\\
&= \lambda\Delta t\left(\phi(t_{k-1})\textbf{v}(t_{k-1})-\phi^{k-1}\textbf{v}^{k-1},  \nabla e_{w}^{k}\right)- \Delta t^{2}\lambda^{2}\left(\phi(t_{k-1})^{2}\nabla w(t_{k})- (\phi^{k-1})^{2}\nabla w^{k}, \nabla e_{w}^{k}\right) & (:\text{term } A)\\
&\quad + \varepsilon\Delta t\left( \phi(t_{k-1})\textbf{v}(t_{k-1})-\phi^{k-1}\textbf{v}^{k-1},  \nabla e_{\phi}^{k} \right)- \Delta t^{2}\lambda \varepsilon\left(\phi(t_{k-1})^{2}\nabla w(t_{k})- (\phi^{k-1})^{2}\nabla w^{k}, \nabla e_{\phi}^{k}\right) & (:\text{term } B)\\
&\quad - \lambda \left(G_{\phi}^{k-1}, \delta e_{\phi}^{k}  \right)+ \Delta t M\left(e_{w}^{k},  G_{\phi}^{k-1}\right)+ \Delta t SM\left(e_{w}^{k}, \delta e_{\phi}^{k}  \right) & (:\text{term } C)\\
&\quad + \Delta t \left(\textbf{v}(t_{k-1})\times\textbf{b}(t_{k-1})-\textbf{v}_{\star}^{k-1}\times\textbf{b}^{k-1},  \nabla\times e_{b}^{k} \right)- \Delta t\left((\textbf{v}(t_{k-1})\cdot\nabla)\textbf{v}(t_{k})-(\textbf{v}^{k-1}\cdot\nabla ) \tilde{\textbf{v}}^{k} ,\tilde{e}_{u}^{k}\right) & (:\text{term } D)\\
&\quad - \Delta t\lambda \left(\phi(t_{k-1})\nabla w(t_{k})-\phi^{k-1}\nabla w^{k}, \tilde{e}_{u}^{k} \right) -\frac{ \Delta t}{\mu}\left(\textbf{b}(t_{k-1})\times\nabla\times\textbf{b}(t_{k}) -\textbf{b}^{k-1}\times\nabla\times\textbf{b}^{k} ,\tilde{e}_{u}^{k} \right) & (:\text{term } E)\\
&\quad + \Delta t \left(\tilde{e}_{u}^{k},  R_{e}^{k}\right)+ \Delta t^{2}\left(\nabla e_{p}^{k-1}, R_{e}^{k}\right)+\frac{\Delta t^{2}}{2}\left\|R_{e}^{k}\right\|^{2}+ \Delta t\lambda \left(R_{a}^{k}, e_{w}^{k}\right)+ \varepsilon\Delta t\left(R_{a}^{k}, e_{\phi}^{k}  \right) & (:\text{term } F)\\
&\quad + \lambda \left(R_{b}^{k}, \delta e_{\phi}^{k} \right)- \Delta t M \left(R_{b}^{k}, e_{w}^{k}\right)+ \Delta t \left(R_{c}^{k}, e_{b}^{k} \right)+ \Delta t\left(R_{d}^{k}, \tilde{e}_{u}^{k}\right). & (:\text{term } G)\\
\end{aligned}
\end{equation}

We derive the estimates of the right-hand side  by using the Young inequality $ab\leq \xi a^{2}+ \frac{1}{4\xi}b^{2}$, \textbf{Assumption \ref{ass1}}, \textbf{Lemma \ref{lemma3-3}}.  
\begin{equation*}\label{4-18}
\begin{aligned}
(\text{term } A) &\leq \left|\lambda\Delta t\left(e_{\phi}^{k-1}\textbf{v}(t_{k-1})+\phi^{k-1}e_{u}^{k-1},  \nabla e_{w}^{k}\right)\right| +\Delta t^{2}\lambda^{2}\left| \left(\phi(t_{k-1})^{2}\nabla e_{w}^{k}, \nabla e_{w}^{k} \right)+\left(\phi(t_{k-1})^{2}\nabla w^{k}, \nabla e_{w}^{k}\right)-\left(\left(\phi^{k-1}\right)^{2}\nabla w^{k}, \nabla e_{w}^{k} \right) \right|\\
&\leq\lambda\Delta t\left(\left\|e_{\phi}^{k-1}\right\| \left\|\textbf{v}(t_{k-1})\right\|_{ L^{\infty}} +\left\|\phi^{k-1}\right\|_{L^{\infty}}\left\|e_{u}^{k-1}\right\|    \right)\left\|\nabla e_{w}^{k}\right\| \\
&\quad+ \Delta t^{2}\lambda^{2} \left(\left\| \phi(t_{k-1})^{2}\right\|_{L^{\infty}}\left\|\nabla e_{w}^{k}\right\|^{2}
+\left\|\phi(t_{k-1})^{2}\right\|_{L^{\infty}}\left\|\nabla w^{k}\right\|\left\|\nabla e_{w}^{k}\right\| + \left\|\left(\phi^{k-1}\right)^{2}\right\|_{L^{\infty}}\left\|\nabla w^{k}\right\|\left\|\nabla e_{w}^{k}\right\| \right)\\
& \leq \frac{\lambda M\Delta t}{8}\left\|\nabla e_{w}^{k}\right\|^{2}+C\left(\Delta t \left\|e_{\phi}^{k-1}\right\|^{2} + \Delta t \left\|e_{u}^{k-1}\right\|^{2}+ \Delta t ^{2}\left\|\nabla e_{w}^{k}\right\|^{2} +\Delta t^{3}\left\|\nabla w^{k}\right\| ^{2} \right).\\
(\text{term } B)  &\leq  \left|\varepsilon \Delta t\left(e_{\phi}^{k-1}\textbf{v}(t_{k-1})+\phi^{k-1}e_{u}^{k-1}, \nabla e_{\phi}^{k}\right )\right| +\Delta t^{2}\lambda \varepsilon \left| \left(\phi(t_{k-1})^{2}\nabla e_{w}^{k}, \nabla e_{\phi}^{k} \right )
+\left( \phi(t_{k-1})^{2}\nabla w^{k},  \nabla e_{\phi}^{k}  \right) +\left(\left(\phi^{k-1}\right)^{2}\nabla w^{k},      \nabla e_{\phi}^{k}\right) \right|\\
&\leq \varepsilon \Delta t\left(\left\|e_{\phi}^{k-1}\right\|  \left\|\textbf{v}(t_{k-1})\right\|_{ L^{\infty}}+ \left\|\phi^{k-1}\right\|_{L^{\infty}}\left\|e_{u}^{k-1}\right\| \right)\left\|\nabla e_{\phi}^{k}\right \|\\
&\quad +\Delta t^{2}\lambda \varepsilon \left(\left\|\phi(t_{k-1})^{2}\right\|_{L^{\infty}}\left\|\nabla e_{w}^{k}\right\|
+\left\|\phi(t_{k-1})^{2}\right\|_{L^{\infty}}\left\|\nabla w^{k}\right\|+\left\|\left(\phi^{k-1}\right)^{2}\right\|_{L^{\infty}}\left\|\nabla w^{k}\right\| \right)\left\|\nabla e_{\phi}^{k} \right\|\\
& \leq \frac{\lambda M \Delta t }{8}\left\|\nabla e_{w}^{k}\right\|^{2}+ C\left( \Delta t\left\| e_{\phi}^{k-1}\right\|^{2} +\Delta t\left\|e_{u}^{k-1}\right\|^{2}+ \Delta t\left\|\nabla e_{\phi}^{k}\right\|^{2}+ \Delta t^{3}\left\|\nabla w^{k}\right\|^{2}\right) .
\end{aligned}
\end{equation*}

For $G_{\phi}^{k-1}$ (see \cite{2020ErrorXU}), the estimate holds as 
\begin{equation*}
\begin{aligned}
&\left\|G_{\phi}^{k-1}\right\|\lesssim \left\|e_{\phi}^{k-1}\right\|, \quad \left\|\nabla G_{\phi}^{k-1}\right\|\lesssim  \left\|e_{\phi}^{k-1}\right\|+\left \|\nabla e_{\phi}^{k-1}\right\|.
\end{aligned}
\end{equation*}
For the \text{term } C,  the estimate is given by 
\begin{align*}
&- \lambda \left(G_{\phi}^{k-1}, \delta e_{\phi}^{k} \right )\\
&= -\lambda\Delta t\left(G_{\phi}^{k-1}, \frac{\delta e_{\phi}^{k} }{\Delta t}\right ) \\
&=  -\lambda\Delta t\left(G_{\phi}^{k-1}, -\nabla\cdot\left(\phi(t_{k-1})\textbf{v}(t_{k-1})-\phi^{k-1}\textbf{v}^{k-1}\right)+ \Delta t \lambda \nabla\cdot\left(\phi(t_{k-1})^{2}\nabla w(t_{k})-\left(\phi^{k-1}\right)^{2}\nabla w^{k}\right )+M\Delta e_{w}^{k}+R_{a}^{k}   \right)  \\
&\leq  \left|\lambda\Delta t\left(\nabla G_{\phi}^{k-1},  \phi(t_{k-1})\textbf{v}(t_{k-1})-\phi^{k-1}\textbf{v}^{k-1}\right)\right|+\left|\lambda^{2}\Delta t^{2}\left(\nabla G_{\phi}^{k-1},  \phi(t_{k-1})^{2}\nabla w(t_{k})-\left(\phi^{k-1}\right)^{2}\nabla w^{k}  \right)\right| \\
&\quad +\left|\lambda\Delta t \left(\nabla G_{\phi}^{k-1}, M\nabla e_{w}^{k} \right)\right|+\left|\lambda\Delta t \left(  G_{\phi}^{k-1}, R_{a}^{k}\right) \right| \\
&\leq \lambda \Delta t\left\|\nabla G_{\phi}^{k-1}\right\|\left(\left\|e_{\phi}^{k-1}\right\| \left\|\textbf{v}(t_{k-1})\right\|_{ L^{\infty}} +\left\|\phi^{k-1}\right\|_{L^{\infty}} \left\|e_{u}^{k-1}\right\|    \right)\\
&\quad+\lambda^{2}\Delta t^{2}\left\|\nabla G_{\phi}^{k-1}\right\|\left(\left\|\phi(t_{k-1})^{2}\right\|_{L^{\infty}}\left\|\nabla e_{w}^{k} \right\| +\left\|\phi(t_{k-1})^{2}\right\|_{L^{\infty}}\left\|\nabla w^{k}\right\|+  \left\|\phi^{k-1}\right\|^{2}_{L^{\infty}}\left\|\nabla w^{k}\right\|  \right )\\
&\quad +\lambda \Delta tM\left\|\nabla G_{\phi}^{k-1}\right\|\left\|\nabla e_{w}^{k}\right\| +\lambda \Delta t\left\|  G_{\phi}^{k-1}\right\|^{2}+\lambda \Delta t \left\|R_{a}^{k}\right\|^{2}\\
& \leq  \frac{\lambda M\Delta t}{8}\left\|\nabla e_{w}^{k}\right\|^{2}+C\left(  \Delta t\left\|e_{\phi}^{k-1}\right\| ^{2}+\Delta t\left\|\nabla e_{\phi}^{k-1}\right\| ^{2} +\Delta t\left\|e_{u}^{k-1}\right\|^{2} +  \Delta t^{2} \left\|\nabla e_{w}^{k}\right\|^{2}+\Delta t^{2}+\Delta t^{3}\left\|\nabla w^{k}\right\|^{2}\right).\\
\end{align*}
Furthermore,  we derive 
\begin{align*}
\Delta t M\left(e_{w}^{k},  G_{\phi}^{k-1}\right)& \leq  \frac{\Delta t M}{8} \left\|e_{w}^{k}\right\|^{2} +C \Delta t \left\|e_{\phi}^{k-1}\right\| ^{2},\\
\Delta t SM\left(e_{w}^{k}, \delta e_{\phi}^{k}  \right) & \leq\frac{\Delta tM}{8}\left\|e_{w}^{k}\right\|^{2} +C\Delta t\left\|\delta e_{\phi}^{k}  \right\|^{2}.
\end{align*}

For the \text{term } $ D=\Delta t \left(\textbf{v}(t_{k-1})\times\textbf{b}(t_{k-1})-\textbf{v}_{\star}^{k-1}\times\textbf{b}^{k-1},  \nabla\times e_{b}^{k} \right)- \Delta t\left(\left(\textbf{v}(t_{k-1})\cdot\nabla\right)\textbf{v}(t_{k})-\left(\textbf{v}^{k-1}\cdot\nabla \right)\tilde{\textbf{v}}^{k} ,\tilde{e}_{u}^{k}\right)$, we have 
\begin{equation*}\label{4-22}
\begin{aligned}
&\Delta t\left(e_{b}^{k-1}\times \textbf{v}(t_{k-1})+\textbf{b}^{k-1}\times (\textbf{v}(t_{k-1})- \textbf{v}_{\star}^{k-1} ), \nabla\times e_{b}^{k}  \right)\\
&\quad\leq \Delta t\left|\left(e_{b}^{k-1}\times \textbf{v}(t_{k-1}), \nabla\times e_{b}^{k} \right)+\left(\textbf{b}^{k-1}\times e_{u}^{k-1}, \nabla\times e_{b}^{k}  \right)  +\left( \textbf{b}^{k-1}\times\frac{\Delta t}{\mu}\textbf{b}^{k-1}\times \nabla\times \textbf{b}^{k},   \nabla\times e_{b}^{k}\right )  \right|\\
& \quad  \leq \Delta t \left(\left\|e_{b}^{k-1}\right\|\left\|\textbf{v}(t_{k-1})\right\|_{L^{\infty}}+\left\|\textbf{b}^{k-1}\right\|_{L^{\infty}}\left\|e_{u}^{k-1}\right\|+ \frac{\Delta t}{\mu}\left\|\textbf{b}^{k-1}\right\|_{L^{\infty}}^{2} \left\|\nabla\times \textbf{b}^{k}\right\| \right)\left\|\nabla\times e_{b}^{k} \right\|\\
&\quad \leq  \frac{\Delta t }{4\sigma\mu} \left\|\nabla\times e_{b}^{k} \right\|^{2} +C\left(  \Delta t \left\|e_{b}^{k-1}\right\|^{2}+ \Delta t \left\|e_{u}^{k-1}\right\|^{2}+\Delta t^{3}\left\|\nabla\times \textbf{b}^{k}\right\|^{2}\right),\\
\end{aligned}
\end{equation*}
and 
\begin{equation*} 
\begin{aligned}
\Delta t\left((\textbf{v}(t_{k-1})\cdot\nabla)\textbf{v}(t_{k})-\left(\textbf{v}^{k-1}\cdot\nabla \right) \tilde{\textbf{v}}^{k} ,\tilde{e}_{u}^{k}\right)  &\leq \left| \Delta t\left((e_{u}^{k-1}\cdot\nabla)\textbf{v}(t_{k})-(\textbf{v}^{k-1}\cdot\nabla)\tilde{e}_{u}^{k}, \tilde{e}_{u}^{k}  \right )\right|\\
& \, \leq  \frac{\Delta t \nu }{8}\left\|\nabla\tilde{e}_{u}^{k}\right\|^{2}+ C\Delta t\left \|e_{u}^{k-1}\right\|^{2}.
\end{aligned}
\end{equation*}

From equation (\ref{3-20e}), we obtain the following expression for \(\tilde{e}_{u}^{k}\) 
\[
\tilde{e}_{u}^{k} = e_{u}^{k} + \Delta t \left(\delta( \nabla e_{p}^{k} ) \right) - \Delta t R_{e}^{k}.
\]

Then, we present the estimates for the \text{terms } E-G as 
\begin{align*}
(\text{term } E)& \leq \Delta t\lambda\left|\left(e_{\phi}^{k-1}\nabla w(t_{k})+\phi^{k-1}\nabla e_{w}^{k } , \tilde{e}_{u}^{k}  \right)\right|+\frac{ \Delta t}{\mu}\left|\left(e_{b}^{k-1}\times\nabla\times  \textbf{b}(t_{k})-\textbf{b}^{k-1}\times\nabla\times e_{b}^{k }, \tilde{e}_{u}^{k} \right) \right|\\
&\leq \Delta t\lambda \left\|e_{\phi}^{k-1}\right\|\left\|\nabla w(t_{k})\right\|_{L^{\infty}} \left\| \tilde{e}_{u}^{k}\right\| +\Delta t\lambda \left\|\phi^{k-1}\right\|_{L^{\infty}}\left\|\nabla e_{w}^{k }\right\|  \left\|\tilde{e}_{u}^{k}\right\| \\
&\quad +\frac{ \Delta t}{\mu} \left(\left\|e_{b}^{k-1 }\right\|\left\|\nabla\times  \textbf{b}(t_{k})\right\|_{L^{4}}\left\|\tilde{e}_{u}^{k} \right\|_{L^{4}}+\left\|\textbf{b}^{k-1}\right\|_{L^{\infty}}\left\|\nabla\times e_{b}^{k }\right\|\left\|\tilde{e}_{u}^{k} \right\| \right )\\
&\leq \Delta t \lambda \left\|e_{\phi}^{k-1}\right\| \left\|\nabla w(t_{k})\right\|_{L^{\infty}} \left\| \tilde{e}_{u}^{k}\right\| + \Delta t\lambda\left\|\phi^{k-1}\right\|_{L^{\infty}} \left\|\nabla e_{w}^{k }\right\|  \left( \left\|e_{u}^{k}\right\|+\Delta t\left\|\delta(\nabla e_{p}^{k} ) \right\|+\Delta t \left\|R_{e}^{k}\right\| \right) \\
&\quad +\frac{ \Delta t}{\mu}   \left\|e_{b}^{k-1 }\right\|\left\|\nabla\times  \textbf{b}(t_{k})\right\|_{L^{4}} \left\|\tilde{e}_{u}^{k} \right\|_{L^{4}}+\frac{ \Delta t}{\mu} \left\|\textbf{b}^{k-1}\right\|_{L^{\infty}}  \left\|\nabla\times e_{b}^{k }\right\|\left( \left\|e_{u}^{k}\right\|+\Delta t\left\|\delta(\nabla e_{p}^{k} ) \right\|+\Delta t \left\|R_{e}^{k}\right\| \right)  \\
& \leq  \frac{\lambda M\Delta t}{8}\left\|\nabla e_{w}^{k }\right\|^{2} +\frac{\Delta t\nu}{8}\left\|\nabla\tilde{e}_{u}^{k}\right\|^{2}+\frac{\Delta t}{4\sigma\mu}\left\|\nabla\times e_{b}^{k }\right\|^{2} + C\left(\Delta t\left\|e_{\phi}^{k-1}\right\|^{2} +\Delta t\left\|e_{u}^{k}\right\|^{2}+ \Delta t^{3}\left\|\delta(\nabla e_{p}^{k} ) \right\|^{2}+\Delta t^{2}+\Delta t\left\|e_{b}^{k-1 }\right\|^{2} \right).\\
(\text{term } F)& \leq \left|\Delta t \left(\tilde{e}_{u}^{k},  R_{e}^{k}\right)+ \Delta t^{2}\left(\nabla e_{p}^{k-1}, R_{e}^{k}\right)+\frac{\Delta t^{2}}{2}\left\|R_{e}^{k}\right\|^{2}+ \Delta t\lambda \left(R_{a}^{k}, e_{w}^{k}\right)+ \varepsilon\Delta t\left(R_{a}^{k}, e_{\phi}^{k} \right )\right|\\
& \leq \frac{\Delta t\nu}{8}\left\|\nabla\tilde{e}_{u}^{k}\right\|^{2} +\frac{\Delta tM}{8}\left\|e_{w}^{k}\right\|^{2}+ C\left(\Delta t^{3}\left\|\nabla e_{p}^{k-1}\right\|^{2}+ \Delta t \left\|  e_{\phi}^{k}\right\|^{2} +\Delta t^{2} \right).\\
(\rm term  \, G)& \leq \left|\lambda \left(R_{b}^{k}, \delta e_{\phi}^{k} \right)- \Delta t M \left(R_{b}^{k}, e_{w}^{k}\right)+ \Delta t \left(R_{c}^{k}, e_{b}^{k} \right)+ \Delta t\left(R_{d}^{k}, \tilde{e}_{u}^{k}\right)\right|\\
& \leq \frac{\varepsilon}{4}\left\|\delta e_{\phi}^{k} \right\|^{2}+ \frac{\Delta t M}{8}\left\|e_{w}^{k}\right\|^{2}+\frac{\Delta t\nu}{8}\left\|\nabla\tilde{e}_{u}^{k}\right\|^{2}+C\left( \Delta t^{2}+ \Delta t \left\|e_{b}^{k}\right\|^{2} \right).
\end{align*}

By combining the above estimates with   (\ref{33}), we get 
\begin{equation*}
\begin{aligned}
& \frac{\lambda M\Delta t}{2}\left(\delta \left\|\nabla e_{w}^{k}\right\|^{2}  \right) +\frac{\varepsilon }{2} \left(\delta\left\|e_{\phi}^{k}\right\|^{2} \right) +\frac{\lambda\varepsilon}{2}\left(\delta \left\|\nabla e_{\phi}^{k}\right\|^{2} +\left\|\delta( \nabla e_{\phi}^{k} )\right\|^{2}\right) + \frac{\varepsilon+ 4\lambda S}{4}\left\|\delta e_{\phi}^{k} \right\|^{2}+ \frac{\Delta tM}{2}\left\|e_{w}^{k}\right\|^{2}\\
& +\frac{1}{2}\left(\delta\left\|e_{b}^{k}\right\|^{2} +\left\|\delta e_{b}^{k} \right\|^{2}\right)+\frac{ \Delta t}{2\sigma\mu}\left\|\nabla\times e_{b}^{k}\right\|^{2}+ \frac{\Delta t\nu}{2}\left\|\nabla\tilde{e}_{u}^{k}\right\|^{2}  +\frac{1}{2}\left(\delta\left\|e_{u}^{k}\right\|^{2} + \left\|\tilde{e}_{u}^{k}- e_{u}^{k-1}\right\|^{2}\right) +\frac{\Delta t^{2}}{2}\left(\delta\left\|\nabla e_{p}^{k}\right\|^{2} \right)\\
& \leq C\Delta t  \left(\Delta t\left\|\nabla e_{w}^{k }\right\|^{2}+\left\| e_{\phi}^{k}\right\|^{2}+\left\|  e_{\phi}^{k-1}\right\|^{2} +\left\|\nabla e_{\phi}^{k-1}\right\|^{2}  +\left\|\nabla e_{\phi}^{k}\right\|^{2}  +\left\|e_{b}^{k-1}\right\|^{2}+\left\|e_{b}^{k }\right\|^{2} + \left\| e_{u}^{k}\right\|^{2}+\left\| e_{u}^{k-1}\right\|^{2}+\Delta t^{2}\left\|\nabla e_{p}^{k}\right\|^{2}+ \Delta t^{2}\left\|\nabla e_{p}^{k-1}\right\|^{2}    \right) \\
&\quad +C\left(\Delta t^{3}\left\|\nabla w^{k}\right\|^{2} + \Delta t^{3}\left\|\nabla\times \textbf{b}^{k}\right\|^{2} \right) +C\Delta t^{2},\\
\end{aligned}
\end{equation*}
where we have supplemented the term $\frac{\lambda M\Delta t}{2}\left\|\nabla e_{w}^{k-1}\right\|^{2}$ to ensure the inequality holds in the form required by \textbf{Lemma 3.1}.

Summing up the above inequality from $k=0,\cdots, m$ and using the fact that $\left\|\nabla e_{w}^{0}\right\|^{2}=\left\|e_{\phi}^{0}\right\|^{2}=\left\|\nabla e_{\phi}^{0}\right\|^{2}=\left\|e_{b}^{0}\right\|^{2}=\left\|e_{u}^{0}\right\|^{2}=\left\|\nabla e_{p}^{0}\right\|^{2}$=0, along with the results of \textbf{Remark \ref{3-3}} and \textbf{Lemma 3.1},
we obtain 
\begin{equation}\label{s-26}
\begin{aligned}
&\frac{\varepsilon}{2} \left\|e_{\phi}^{m}\right\|^{2}+\frac{\lambda\varepsilon}{2} \left\|\nabla e_{\phi}^{m}\right\|^{2}+\frac{1}{2}\left\|e_{b}^{m}\right\|^{2}+\frac{1}{2}\left\| e_{u}^{m}\right\|^{2}+\frac{\Delta t^{2}}{2} \left\|\nabla e_{p}^{m}\right\|^{2}+  \frac{\lambda M\Delta t}{2}\left\|\nabla e_{w}^{m}\right\|^{2}\\
& + \Delta t\sum_{k=0}^{m}\left(\frac{\varepsilon+4\lambda S}{4\Delta t}  \left\|\delta e_{\phi}^{k}\right\|^{2} +\frac{\lambda\varepsilon}{2\Delta t} \left\|\delta(\nabla e_{\phi}^{k})\right\|^{2}+ \frac{1}{2\sigma\mu}\left\|\nabla\times e_{b}^{k}\right\|^{2}+ \frac{1}{2\Delta t}\left\|\delta e_{b}^{k}\right\|^{2}  + \frac{1}{2\Delta t}\left\|\tilde{e}_{u}^{k}- e_{u}^{k-1}\right\|^{2} +\frac{ M }{2}\left\|e_{w}^{k}\right\|^{2}+ \frac{\nu}{2} \left\|\nabla\tilde{e}_{u}^{k}\right\|^{2} \right)\\
&\leq C_{3}\Delta t^{2},
\end{aligned}
\end{equation}
where  $C_{0}$,  $C_{3}$ are two positive general constants, with $\Delta t\leq C_{0}$ and $m\leq K$.

\textbf{Step ii}. Then, we   give the estimates of $\left\|\textbf{b}^{k }\right\|_{H^{2}}$ and $\left\|\phi^{k }\right\|_{H^{2}}$.
\vspace{3pt}

(i)
Applying the divergence operator to equation (\ref{3-4a}), we have 
\begin{equation*}
-\Delta t\Delta\left(\delta p^{k}\right)=-\nabla\cdot \tilde{\textbf{v}}^{k}.
\end{equation*}
Combining  the inequality \eqref{s-26}, we obtain 
\begin{equation}\label{p}
\Delta t\left\|\delta p^{k}\right\|_{H^{2}}\leq \left\|\nabla\cdot\tilde{\textbf{v}}^{k}\right\|=\left\|\nabla\cdot \tilde{e}^{k}_{u}\right \|\lesssim\Delta t^{1/2}.
\end{equation}

With the help of  the identity $\nabla\times\nabla\times \textbf{b}^{k}=-\Delta \textbf{b}^{k}+ \nabla\left( \nabla\cdot\textbf{b}^{k}\right)=-\Delta \textbf{b}^{k}$, and   the equations (\ref{3-2a})-(\ref{3-2b}), we get 
\begin{equation}\label{3-29}
 -\frac{1}{\sigma\mu}\Delta\textbf{b}^{k}=\nabla\times\left(\textbf{v}^{k-1}_{\star}\times\textbf{b}^{k-1}  \right) -\frac{\delta\textbf{b}^{k}}{\Delta t}.
\end{equation}

By utilizing  the equations  (\ref{3-2b}) and  (\ref{3-4a}), we derive 
\begin{equation*}
\textbf{v}^{k-1}_{\star}=\tilde{\textbf{v}}^{k-1}-\Delta t\nabla\left(\delta p^{k-1}\right)-\frac{\Delta t}{\mu}\textbf{b}^{k-1}\times\nabla\times\textbf{b}^{k}.
\end{equation*}

Due to the identity $\nabla\times(\textbf{a}\times \textbf{b})=\textbf{b}\cdot\nabla\textbf{a}-\textbf{a}\cdot\nabla\textbf{b}+\textbf{a}\nabla\cdot\textbf{b}-\textbf{b}\nabla\cdot\textbf{a}   $, $\nabla\cdot\textbf{b}^{k}=0$, and  the equation (\ref{3-29}), we have 
\begin{equation}\label{4-27}
\begin{aligned}
\|\textbf{b}^{k}\|_{H^{2}}&\lesssim \left\| \frac{\delta\textbf{b}^{k}}{\Delta t}\right\|+\left\|\nabla\times\left(\textbf{v}^{k-1}_{\star}\times\textbf{b}^{k-1}  \right) \right\| \\
&\lesssim \left\|\frac{\delta\textbf{b}^{k}}{\Delta t}\right\|+\left\|\textbf{b}^{k-1}\nabla\textbf{v}^{k-1}_{\star}\right\|+\left\|\textbf{v}^{k-1}_{\star}\nabla\textbf{b}^{k-1}\right\|+\left\|\textbf{b}^{k-1}\nabla\cdot\textbf{v}^{k-1}_{\star} \right\|\\
&\lesssim\left\|\frac{-\delta e_{b}^{k} }{\Delta t}+\frac{\delta\textbf{b}(t_{k})}{\Delta t}\right\|  & (\text{term } M_{1}) \\
&\quad +\left\|\textbf{b}^{k-1}\nabla\tilde{\textbf{v}}^{k-1}\right\|  +\Delta t\left\|\textbf{b}^{k-1}\nabla\nabla\left(\delta p^{k-1}\right)\right\|  +\frac{\Delta t}{\mu}\left\|\textbf{b}^{k-1}\nabla\left(\nabla\times\textbf{b}^{k}\times\textbf{b}^{k-1} \right )\right\|    & (\text{term } M_{2}) \\
&\quad+\left\|\tilde{\textbf{v}}^{k-1}\nabla\textbf{b}^{k-1} \right\|  +\Delta t\left\| \nabla \left(\delta p^{k-1} \right)\nabla\textbf{b}^{k-1}\right \|  +\frac{\Delta t}{\mu}\left\|\left(\nabla\times\textbf{b}^{k}\times\textbf{b}^{k-1} \right)\nabla\textbf{b}^{k-1}\right\|  & (\text{term } M_{3}) \\
&\quad +\left\|\textbf{b}^{k-1}\nabla\cdot\tilde{\textbf{v}}^{k-1}\right\| +\Delta t\left\|\textbf{b}^{k-1}\Delta \left(\delta p^{k-1}\right)\right\| +\frac{\Delta t}{\mu}\left\|\textbf{b}^{k-1}\nabla\cdot \left(\nabla\times\textbf{b}^{k}\times\textbf{b}^{k-1} \right)\right\|. & (\text{term } M_{4}) \\
\end{aligned}
\end{equation}

Obviously, we can obtain the \text{term } $M_{1}\lesssim C_{4}$ from equation (\ref{s-26}). With the help of the inequalities
$$\left\|\nabla\tilde{\textbf{v}}^{k-1}\right\|\leq \left\|\nabla\tilde{e}_{u}^{k-1}\right\|+\left\|\nabla\tilde{\textbf{v}} (t_{k-1})\right\|\lesssim \Delta t^{1/2} +C_{5}\lesssim C_{6}, \quad \left\|\nabla\textbf{b}^{k}\right\|\leq  \left\|\nabla e _{b}^{k }\right\|+\left\|\nabla\textbf{b}  (t_{k }) \right\|\lesssim \left\|\nabla\times e _{b}^{k }\right\|+\left\|\nabla \textbf{b}  (t_{k })\right\|\lesssim \Delta t^{1/2} +C_{7}\lesssim C_{8}, $$
and equation (\ref{p}), we have 
\begin{equation}\label{4-28}
\begin{aligned}
(\text{term } M_{2}) & \leq  \left\|\textbf{b}^{k-1}\right\|_{L^{\infty}}\left\|\nabla\tilde{\textbf{v}}^{k-1}\right\|  +\Delta t  \left\|\textbf{b}^{k-1} \right\|_{L^{\infty}}\left\|\delta p^{k-1}\right\|_{H^{2}}  +\frac{\Delta t}{\mu}\left(\left\|\textbf{b}^{k-1} \right\|^{2}_{L^{\infty}}\left\|\textbf{b}^{k}\right\|_{H^{2}}+ \left\|\textbf{b}^{k-1}\right\|_{L^{\infty}}\left\|\nabla\textbf{b}^{k}\right\|_{L^{3}}\left\|\nabla\textbf{b}^{k-1}\right\|_{L^{6}}\right)\\
&\leq  C\left(\left\|\nabla\tilde{\textbf{v}}^{k-1}\right\| +\Delta t \left\|\delta p^{k-1}\right\|_{H^{2}}\right)+C\Delta t \left(\left\|\textbf{b}^{k}\right\|_{H^{2}}+ \left\|\nabla\textbf{b}^{k}\right\| ^{1/2}\left\| \textbf{b}^{k }\right\|_{H^{2}}^{1/2}\left\|\textbf{b}^{k-1}\right\|_{H^{2}}\right)\\
&\leq C_{9}+C\Delta t^{1/2}+C\Delta t \left\|\textbf{b}^{k}\right\|_{H^{2}}+\frac{1}{6}\left\|\textbf{b}^{k}\right\|_{H^{2}}+C\Delta t^{2}\left\|\textbf{b}^{k-1}\right\|_{H^{2}}^{2}\\
&\leq C_{10} +C\Delta t \left\|\textbf{b}^{k}\right\|_{H^{2}}+\frac{1}{6}\left\|\textbf{b}^{k}\right\|_{H^{2}}+C\Delta t^{2}\left\|\textbf{b}^{k-1}\right\|_{H^{2}}^{2},
\end{aligned}
\end{equation}
where we apply 
\begin{equation*}
\begin{aligned}
&\frac{\Delta t}{\mu}\left\|\textbf{b}^{k-1}\nabla\left(\nabla\times\textbf{b}^{k}\times\textbf{b}^{k-1}  \right)\right\|\\
&\leq \frac{\Delta t}{\mu}\left(\left\|\textbf{b}^{k-1}\right\|_{L^{\infty}}\left\|\nabla \left(\nabla\times\textbf{b}^{k}\right)\right\|_{L^{2}}\left\|\textbf{b}^{k-1} \right\|_{L^{\infty}} +   \left\|\textbf{b}^{k-1}\right\|_{L^{\infty}}\left\|\nabla\times\textbf{b}^{k}\right\|_{L^{3}}\left\|\nabla \textbf{b}^{k-1} \right\|_{L^{6}}\right)\\
&\leq \frac{\Delta t}{\mu}\left(\left\|\textbf{b}^{k-1}\right \|^{2}_{L^{\infty}}\left\|\textbf{b}^{k}\right\|_{H^{2}}+ \left\|\textbf{b}^{k-1}\right\|_{L^{\infty}}\left\|\nabla\textbf{b}^{k}\right\|_{L^{3}}\left\|\nabla\textbf{b}^{k-1}\right\|_{L^{6}} \right).
\end{aligned}
\end{equation*}

From equation (\ref{s-26}), we have $\left\|\nabla\tilde{e}^{k-1}_{u}\right\|\lesssim \Delta t^{1/2}$ and $\left\|\nabla\times\textbf{b}^{k} \right\|\lesssim\left\|\nabla \times e_{b}^{k}\right\|+\left\|\nabla\times \textbf{b}(t_{k})\right\|\lesssim \Delta t^{1/2}+C_{11}\lesssim C_{12}$. Thus, we obtain 
\begin{equation}\label{4-29}
\begin{aligned}
(\text{term } M_{3}) & \leq \left\|\tilde{e}^{k-1}_{u}\nabla\textbf{b}^{k-1} \right\| +\left\|\textbf{v}(t_{k-1})\nabla\textbf{b}^{k-1} \right\| +\Delta t\left\| \nabla \left(\delta p^{k-1}\right)\nabla\textbf{b}^{k-1} \right\|   + \frac{\Delta t}{\mu} \left\|\left(\nabla\times\textbf{b}^{k}\times\textbf{b}^{k-1}\right)\nabla\textbf{b}^{k-1}\right\|\\
 &\leq \left\|\tilde{e}^{k-1}_{u}\right\|_{L^{6}}\left\|\nabla\textbf{b}^{k-1} \right\|_{L^{3}} +\left\|\textbf{v}(t_{k-1})\right\|_{L^{\infty}}\left\|\nabla\textbf{b}^{k-1} \right\| +\Delta t\left\| \nabla \left(\delta p^{k-1}\right )\right\|_{L^{6}}\left\| \nabla\textbf{b}^{k-1} \right\|_{L^{3}}  + \frac{\Delta t}{\mu} \left\|\nabla\times\textbf{b}^{k}\right\|_{L^{3}}\left\|\textbf{b}^{k-1}\right\|_{L^{\infty}}\left\|\nabla\textbf{b}^{k-1}\right\|_{L^{6}}\\
 &\leq \left\|\nabla\tilde{e}^{k-1}_{u}\right\|\left\|\nabla\textbf{b}^{k-1} \right\|^{1/2}\left\|\textbf{b}^{k-1} \right\|^{1/2}_{H^{2}}+C\left\|\nabla\textbf{b}^{k-1} \right\|  \\
&\quad +C\Delta t\left\|\delta p^{k-1}\right\|_{H^{2}} \left\|\nabla\textbf{b}^{k-1} \right\|^{1/2}\left\|\textbf{b}^{k-1} \right\|^{1/2}_{H^{2}}+C\Delta t \left\|\nabla\times\textbf{b}^{k}\right\|^{1/2}\left\|\textbf{b}^{k}\right\|^{1/2}_{H^{2}}\left\| \textbf{b}^{k-1}\right\|_{H^{2}}\\
 &\leq  C_{8} +C\Delta t\left\|\textbf{b}^{k-1} \right\| _{H^{2}}+\frac{1}{6}\left\|\textbf{b}^{k}\right\| _{H^{2}}+C\Delta t^{2}\left\| \textbf{b}^{k-1}\right\|_{H^{2}}^{2}.
\end{aligned}
\end{equation}

Similarly, we derive 
\begin{equation}\label{4-30}
\begin{aligned}
(\text{term } M_{4}) & \leq C_{9} +C\Delta t\left\|\textbf{b}^{k-1} \right\| _{H^{2}}+\frac{1}{6}\left\|\textbf{b}^{k}\right\| _{H^{2}}+C\Delta t^{2}\left\| \textbf{b}^{k-1}\right\|_{H^{2}}^{2}.
\end{aligned}
\end{equation}

By combining equation (\ref{4-27}) with  equations (\ref{4-28})-(\ref{4-30}), we find the following:  if $\Delta t\leq \tilde{C}$, then there exist positive constants $\hat{C}_{4}, C_{13}$, and $C_{14}$ such that, for $k\leq K$, 
\begin{equation}\label{4-29}
\begin{aligned}
\left\|\textbf{b}^{k}\right\|_{H^{2}}\leq \hat{C}_{4}+C_{13}\Delta t\left\|\textbf{b}^{k-1} \right\| _{H^{2}} +C_{14}\Delta t^{2}\left\| \textbf{b}^{k-1}\right\|_{H^{2}}^{2}.
\end{aligned}
\end{equation}

Thus, by  \textbf{Lemma \ref{lem2}}, if $ \max \left\{C_{13}, C_{14}^{1/2}\right\}\hat{D}\Delta t \leq 1$, that is, if $\Delta t \leq 1/ \max \left\{C_{13}, C_{14}^{1/2}\right\}\hat{D}$,   we have 
\begin{equation}\label{bH2}
\left\|\textbf{b}^{K}\right\|_{H^{2}}\leq \hat{D},
\end{equation}
where   $\hat{D}= \max \left\{\left\|\textbf{b}^{0}\right\|_{H^{2}}, \hat{C}_{4}\right\}+2$.

(ii)  From equation (\ref{3-20b}), we obtain 
\begin{equation*}
\varepsilon\left\|\Delta e_{\phi}^{k}\right\|\leq\left\|R_{b}^{k}\right\|+ \left\|e_{w}^{k}\right\|+\left\|G_{\phi}^{k-1}\right\| +S\left\|\delta e_{\phi}^{k} \right\|.
\end{equation*}

From inequality \eqref{s-26}, we get 
\[
\left\|e_\phi^{k}\right\| \lesssim \Delta t, \quad
\left\|e_w^{k}\right\| \lesssim \Delta t^{1/2}, \quad
\left\|\delta e_{\phi}^{k} \right\| \lesssim \Delta t.
\]

Thus, we derive 
\begin{equation}\label{4-33}
\begin{aligned}
\left\|e_\phi^{k }\right\|_{H^{2}}&\leq C\left( \left\|e_\phi^{k }\right\| +\left\|\Delta e_\phi^{k }\right\| \right)\\
&\leq C\left(\left\|e_\phi^{k }\right\|+\left\|R_{b}^{k}\right\|+ \left\|e_{w}^{k}\right\|+\left\|e_\phi^{k-1}\right\| + \left\|\delta e_{\phi}^{k} \right\|\right)\\
&\leq C_{15}\Delta t^{1/2}.
\end{aligned}
\end{equation}

\textbf{Step iii}. Lastly, we derive estimates for $\left\|\phi^{k }\right\|_{L^{\infty}}$ and $ \left\|\textbf{b}^{K }\right\|_{L^{\infty}}$. With the help of \eqref{bH2}, we have
$$\left\|e_{b}^{K}\right\|_{H^{2}}\leq \left\|\textbf{b}^{K}\right\|_{H^{2}}+\left\|\textbf{b}(t_{K})\right\|_{H^{2}}\leq C_{16},$$ 
and we obtain 
\begin{equation*}
\begin{aligned}
\left\|\textbf{b}^{K}\right\|_{L^{\infty}}&\leq \left\|e_{b}^{K}\right\|_{L^{\infty}}+\left\|\textbf{b}(t_{K})\right\|_{L^{\infty}}\\
&\leq  C_{17}\left\|e_{b}^{K}\right\|_{H^{2}}^{3/4}\left\|e_{b}^{K}\right\| ^{1/4} +\left\|\textbf{b}(t_{K})\right\|_{L^{\infty}}\\
&\leq C_{17}C_{16}^{3/4}C_{3}^{1/8}\Delta t^{1/4} +\left\|\textbf{b}(t_{K})\right\|_{L^{\infty}}\\
&\leq \kappa_{b_{3}},
\end{aligned}
\end{equation*}
for $C_{17}C_{16}^{3/4}C_{3}^{1/8}\Delta t^{1/4}\leq 1$, i.e., $\Delta t\leq 1/(C_{17}^{4}C_{16}^{3}C_{3}^{1/2})$. Hence, the bound \(\left\|\mathbf{b}^{K}\right\|_{L^{\infty}} \lesssim \kappa_{b_{3}}\) is established.

Then, we derive the estimate for \(\left\|\phi^{k}\right\|_{L^{\infty}}\) as 
\begin{equation*}
\begin{aligned}
\left\|\phi^{k}\right\|_{L^{\infty}}&\leq \left\|e_{\phi}^{k}\right\|_{L^{\infty}}+\left\|\phi(t_{k})\right\|_{L^{\infty}}\\
&\leq C_{17}\left\|e_{\phi}^{k}\right\|_{H^{2}} ^{3/4}\left\|e_{\phi}^{k}\right\| ^{1/4}+\left\|\phi(t_{k})\right\|_{L^{\infty}}\\
&\leq C_{17}C_{15}^{3/4}\Delta t^{3/8}C_{3}^{1/8}\Delta t^{1/4} +\left\|\phi(t_{k})\right\|_{L^{\infty}}\\
&\leq \kappa_{\phi_{3}},
\end{aligned}
\end{equation*}
provided that $C_{17}C_{15}^{3/4}\Delta t^{3/8}C_{3}^{1/8}\Delta t^{1/4} \leq 1$, which is equivalent to $\Delta t\leq 1/(C_{17} ^{8/5}C_{15}^{6/5} C_{3}^{1/5})$.

Based on the above process, we can derive inequality (\ref{lemma3-4}) for $\Delta t\leq C$. The constants $\kappa_{\phi}$, $\kappa_{b}$, and $C$ are defined by 
\begin{equation*}
\begin{aligned}
&\kappa_{\phi}=\max\left\{\kappa_{\phi_{1}}, \kappa_{\phi_{2}}, \kappa_{\phi_{3}}\right\}, \quad \kappa_{b}=\max\left\{\kappa_{b_{1}}, \kappa_{b_{2}}, \kappa_{b_{3}}\right\},\\
&C= \min \left \{ C_{0}, (C_{6}-C_{5})^2,  (C_{8}-C_{7})^2, (C_{12}-C_{11})^2, \tilde{C},    \frac{1}{ \max \left\{C_{13}, C_{14}^{1/2}\right\}\hat{D}},  \frac{1}{C_{17}^{4}C_{16}^{3}C_{3}^{1/2}}, \frac{1}{C_{17} ^{8/5}C_{15}^{6/5} C_{3}^{1/5}} \right\}.
\end{aligned}
 \end{equation*}

\end{proof}

\begin{The}\label{the3-1}
Suppose the solution of the equations (\ref{2-1}) satisfies the \textbf{Assumption \ref{ass1}}. Then, the \textbf{Scheme I} is unconditionally convergent, specifically,
\begin{equation}\label{e1}
\begin{aligned}
& \frac{\lambda\varepsilon}{2} \left\|\nabla e_{\phi}^{m}\right\|^{2}+\frac{1}{2}\left\|e_{b}^{m}\right\|^{2}+\frac{1}{2}\left\| e_{u}^{m}\right\|^{2}+\frac{\Delta t^{2}}{2} \left\|\nabla e_{p}^{m}\right\|^{2}+  \frac{\lambda M\Delta t}{2}\left\|\nabla e_{w}^{m}\right\|^{2}\\
& + \Delta t\sum_{k=0}^{m}\left(\frac{\varepsilon+4\lambda S}{4\Delta t}  \left\|\delta e_{\phi}^{k}\right\|^{2} +\frac{\lambda\varepsilon}{2\Delta t} \left\|\delta(\nabla e_{\phi}^{k})\right\|^{2}+ \frac{1}{2\sigma\mu}\left\|\nabla\times e_{b}^{k}\right\|^{2}+ \frac{1}{2\Delta t}\left\|\delta e_{b}^{k}\right\|^{2}  + \frac{1}{2\Delta t}\left\|\tilde{e}_{u}^{k}- e_{u}^{k-1}\right\|^{2} + \frac{\nu}{2} \left\|\nabla\tilde{e}_{u}^{k}\right\|^{2}\right)\\
&\leq C\Delta t^{2},
\end{aligned}
\end{equation}
where \(m\) satisfies $0\leq m\leq T/\Delta t$.
\end{The}

\begin{proof}
\textbf{Case 1}. Based on the correctness of \textbf{Lemma \ref{lem3-4}}, we conclude that  $\left\|\phi^{k }\right\|_{L^{\infty}}\leq \kappa_{\phi}$ and $\left\|\textbf{b}^{k}\right\|_{L^{\infty}}\leq \kappa_{b}$. Hence, equation (\ref{e1}) is valid under the condition that \(\Delta t \leq C\).

\textbf{Case 2}. If $\Delta t \geq C$, by means of \textbf{Remark \ref{3-3}} and \textbf{Assumption \ref{ass1}}, we derive 
\begin{equation*}\label{4-26}
\begin{aligned}
& \frac{\lambda\varepsilon}{2} \left\|\nabla e_{\phi}^{m}\right\|^{2}+\frac{1}{2}\left\|e_{b}^{m}\right\|^{2}+\frac{1}{2}\left\| e_{u}^{m}\right\|^{2}+\frac{\Delta t^{2}}{2} \left\|\nabla e_{p}^{m}\right\|^{2}+  \frac{\lambda M\Delta t}{2}\left\|\nabla e_{w}^{m}\right\|^{2}\\
& + \Delta t\sum_{k=0}^{m}\left(\frac{\varepsilon+4\lambda S}{4\Delta t}  \left\|\delta e_{\phi}^{k}\right\|^{2} +\frac{\lambda\varepsilon}{2\Delta t} \left\|\delta(\nabla e_{\phi}^{k})\right\|^{2}+ \frac{1}{2\sigma\mu}\left\|\nabla\times e_{b}^{k}\right\|^{2}+ \frac{1}{2\Delta t}\left\|\delta e_{b}^{k}\right\|^{2}  + \frac{1}{2\Delta t}\left\|\tilde{e}_{u}^{k}- e_{u}^{k-1}\right\|^{2}  + \frac{\nu}{2} \left\|\nabla\tilde{e}_{u}^{k}\right\|^{2}\right)\\
& \leq C_{18}=\frac{C_{18}}{C^{2}}C^{2} \leq \frac{C_{18}}{C^{2}}\Delta t^{2}  \leq C \Delta t^{2}.
\end{aligned}
\end{equation*}

The   unconditionally convergent is valid, as demonstrated by \textbf{Case 1} and \textbf{Case 2}.
\end{proof}

\section{A Scheme Based on IEQ Method }

In this section, we focus on the convergence analysis of the unconditionally energy stable IEQ scheme \cite{2019Fast, 2019EfficientLIU}.  The method only requires the nonlinear potential to be bounded from below, thus bypassing the need for artificial extension of the nonlinear potential. Initially, we need to transform the chemical potential \(f(\phi)\) as  
\begin{equation}
\begin{aligned}
& f(\phi) = M(\phi) N, \
&\text{where } M(\phi) = \frac{f(\phi)}{\sqrt{F(\phi) + C}}, \quad N = \sqrt{F(\phi) + C},
\end{aligned}
\end{equation}
where $C$ is chosen such that  $F(\phi) +C > 0$. We treat $N$ as a new variable and take the time derivative of $N$ to obtain 
$$N_{t}=\frac{1}{2}M(\phi)\phi_{t}.$$

Hence, the equations \eqref{mdel1}-\eqref{model2} can be equivalently rewritten as: 
\begin{equation}\label{IEQmodel}
\begin{aligned}
&\phi_{t}+\nabla\cdot(\phi \textbf{v})=M \Delta w, &\text{in} \ \Omega\times (0,T],\\
&w=-\varepsilon \Delta\phi+M(\phi)N, &\text{in} \ \Omega\times (0,T],\\
&N_{t} = \frac{1}{2}M(\phi)\phi_{t}, &\text{in} \ \Omega\times (0,T].
\end{aligned}
\end{equation}


The semi-discrete IEQ scheme (\textbf{Scheme II}): With the initial values $\phi^{0}, \textbf{v}^{0}, \textbf{b}^{0}$, $p^{0}$=0, and $N^{0}=\sqrt{F(\phi^{0}) +C}$, we solve for $\phi^{k}, w^{k}, \textbf{v}^{k}, \textbf{b}^{k}$, and $p^{k}$ through the following steps.

\textbf{Step IEQ-1}. Compute  $\phi^{k}$ and $w^{k}$ from 
\begin{subequations}\label{4-2ieq}
\begin{align}
&\frac{\phi^{k}- \phi^{k-1}}{\Delta t}+\nabla\cdot \left(\phi^{k-1} \textbf{v}^{k-1}\right)-\Delta t\lambda\nabla\cdot \left(\left(\phi^{k-1}\right)^{2}\nabla w^{k}\right)=M\Delta w^{k}, \label{4-2a}\\
&w^{k}=-\varepsilon \Delta\phi^{k}+M\left(\phi^{k-1}\right)N^{k},\label{4-2b}\\
&\frac{\partial\phi^{k}}{\partial \textbf{n}}|_{\partial\Omega}=0, \quad \frac{\partial w^{k}}{\partial \textbf{n}}|_{\partial\Omega}=0,
\end{align}
\end{subequations}
where
\begin{equation}\label{4-3}
N^{k}=N^{k-1}+\frac{1}{2}M\left(\phi^{k-1}\right)\left(\phi^{k}-\phi^{k-1}\right).
\end{equation}
The IEQ scheme is composed of \textbf{Step IEQ-1} and the steps from \textbf{Step 2} to \textbf{Step 4} of Scheme I.
\begin{The}\label{theorem4-1}
Without loss of generality, we set the source term $\textbf{f}$=$\textbf{0}$.  The \textbf{scheme II}, consisting of  \textbf{Step IEQ-1} and \textbf{Step 2}-\textbf{Step 4}, is unconditionally energy stable in the sense that 
\begin{equation}\label{4-44ieq}
\begin{aligned}
&\tilde{E}^{k}- \tilde{E}^{k-1} \leq 0,\
&\text{where } \tilde{E}^{k}=\frac{\lambda\varepsilon}{2}\left\|\nabla\phi^{k}\right\|^{2}+\frac{1}{2\mu}\left\|\textbf{b}^{k}\right\|^{2}+ \frac{1}{2}\left\|\textbf{v}^{k}\right\|^{2}+ \lambda \left\|N^{k} \right\|^{2} +\frac{\Delta t^{2}}{2}\left\|\nabla p^{k}\right\|^{2}.
\end{aligned}
\end{equation}

\end{The}
\begin{proof}
Taking the $L^{2}$ inner product of  equation  (\ref{4-2a}) with $\Delta t\lambda w^{k}$,  equation  (\ref{4-2b}) with $\lambda \left(\delta \phi^{k} \right)$, and  equation (\ref{4-3}) with $2\lambda N^{k}$, we derive 
\begin{equation}\label{4-5ieq}
\begin{aligned}
&\lambda \left(\delta\phi^{k},  w^{k}\right)-\Delta t \lambda\left(\phi^{k-1}\textbf{v}^{k-1}, \nabla w^{k}\right)+\Delta t^{2} \lambda^{2}\left\|\phi^{k-1}\nabla w^{k}\right\|^{2}=-\Delta t \lambda M\left\|\nabla w^{k}\right\|^{2},\\
& \frac{\lambda\varepsilon}{2}\left(\delta\left\|\nabla \phi^{k}\right\|^{2}+ \left\|\delta(\nabla \phi^{k})\right\|^{2}  \right) +\lambda \left(M\left(\phi^{k-1}\right)N^{k}, \delta\phi^{k}\right) =\lambda \left(\delta\phi^{k}, w^{k}\right),\\
&\lambda \left(\delta\left\|N^{k}\right\|^{2} +\left\|\delta N^{k}\right\|^{2} \right)=\lambda \left(M\left(\phi^{k-1}\right)\delta\phi^{k}, N^{k} \right).
\end{aligned}
\end{equation}

Combining equations (\ref{3-8})-(\ref{3-11}) and  (\ref{4-5ieq}), with the help of equations (\ref{3-12})-(\ref{3-13}), we conclude 
\begin{equation}\label{4-6}
\begin{aligned}
&\frac{\lambda\varepsilon}{2}\delta\left\|\nabla\phi^{k}\right\|^{2}+\frac{1}{2\mu}\delta\left\|\textbf{b}^{k}\right\|^{2}+ \frac{1}{2}\delta\left\|\textbf{v}^{k}\right\|^{2}+ \lambda \delta\left\|N^{k}\right\|^{2}+\frac{\Delta t^{2}}{2}\delta\left\|\nabla p^{k}\right\|^{2} +\lambda\Delta tM\left\|\nabla w^{k}\right\|^{2}+\frac{\lambda\varepsilon}{2}\left\|\delta(\nabla\phi^{k})\right\|^{2}\\
&+\lambda\left\|\delta N^{k}\right\|^{2}+\frac{1}{2\mu}\left\|\delta\textbf{b}^{k}\right\|^{2}+\frac{ \Delta t}{\sigma\mu^{2}}\left\|\nabla\times\textbf{b}^{k}\right\|^{2}+\frac{1}{4}\left\|\tilde{\textbf{v}}^{k}-\textbf{v}^{k-1}\right\|^{2}+ \nu\Delta t\left\|\nabla\tilde{\textbf{v}}^{k}\right\|^{2} \leq 0.
\end{aligned}
\end{equation}
\end{proof}
{\Rem  \label{Remark4-1}
Summing up the above inequality (\ref{4-6}) from \(k=0\) to \(m\) (\(m \leq \frac{T}{\Delta t}\)), we obtain 
\begin{equation}\label{4sm-6}
\begin{aligned}
&\frac{\lambda\varepsilon}{2} \left\|\nabla\phi^{m}\right\|^{2} +\frac{1}{2\mu} \left\|\textbf{b}^{m}\right\|^{2} + \frac{1}{2} \left\|\textbf{v}^{m}\right\|^{2} + \lambda \left\|N^{m}\right\|^{2} +\frac{\Delta t^{2}}{2} \left\|\nabla p^{m}\right\|^{2} +\sum_{k=0}^{m}\Big( \lambda\Delta tM\left\|\nabla w^{k}\right\|^{2}+\frac{\lambda\varepsilon}{2}\left\|\delta(\nabla\phi^{k}) \right\|^{2}\\
&+\lambda\left\|\delta N^{k} \right\|^{2}+\frac{1}{2\mu}\left\|\delta\textbf{b}^{k} \right\|^{2}+\frac{ \Delta t}{\sigma\mu^{2}}\left\|\nabla\times\textbf{b}^{k}\right\|^{2} +\frac{1}{4}\left\|\tilde{\textbf{v}}^{k}-\textbf{v}^{k-1}\right\|^{2}+ \nu\Delta t\left\|\nabla\tilde{\textbf{v}}^{k}\right\|^{2}\Big) \leq C_{ieq},
\end{aligned}
\end{equation}
where $C_{ieq}$ is a general positive constant. From inequality (\ref{4sm-6}), we derive the following 
\begin{equation}\label{ieq8}
\sum_{k=0}^{m}\lambda\Delta tM  \left\|\nabla w^{k}\right\|^{2}\leq C_{ieq},\quad  \sum_{k=0}^{m}\frac{ \Delta t}{\sigma\mu^{2}} \left\|\nabla\times\textbf{b}^{k}\right\|^{2}\leq C_{ieq}.
\end{equation}
}

\subsection{Convergence analysis}
We rewrite the Cahn-Hilliard equations as 
\begin{subequations}\label{4-7ieq}
\begin{align}
&\frac{\delta\phi(t_{k})}{\Delta t}+\nabla\cdot\left(\phi(t_{k-1}) \textbf{v}(t_{k-1})\right)-\Delta t\lambda\nabla\cdot \left(\phi(t_{k-1})^{2}\nabla w(t_{k})\right)
=M \Delta w(t_{k})+R_{a}^{k}, \label{4-9a} \\
&-w(t_{k})-\varepsilon \Delta\phi(t_{k})+M(\phi (t_{k-1}))N(t_{k})=R_{f}^{k} , \label{4-9b} \\
&\frac{\delta N(t_{k})}{\Delta t}=\frac{1}{2}M(\phi (t_{k-1}))\frac{\delta\phi (t_{k})}{\Delta t} +R_{g}^{k}, \label{4-9c}
\end{align}
\end{subequations}
where the truncation errors are given by 
\begin{equation}\label{4-8ieq}
\left\{
\begin{aligned}
&R_{a}^{k} =\frac{\delta\phi(t_{k}) }{\Delta t}-\phi_{t}(t_{k})-\delta(\nabla\cdot(\phi(t_{k}) \textbf{v}(t_{k})))-\Delta t\lambda\nabla\cdot \left(\phi(t_{k-1})^{2}\nabla w(t_{k})\right),\\
&R_{f}^{k} =-\left(\delta M(\phi (t_{k}))\right)N (t_{k}),\\
&R_{g}^{k} = \frac{\delta N (t_{k})}{\Delta t}-N_{t} (t_{k})-\frac{1}{2}M(\phi (t_{k-1}))\frac{\delta\phi (t_{k})}{\Delta t} +\frac{1}{2}M(\phi (t_{k }))\phi_{t} (t_{k}).
\end{aligned}
\right.
\end{equation}
We obtain the error equations by subtracting equation (\ref{4-9b}) from equation (\ref{4-2b}), and equation (\ref{4-9c}) from equation \(N^{k}\) in \eqref{4-3} as 
\begin{subequations}\label{4-9ieq}
\begin{align}
&-e_{w}^{k}-\varepsilon\Delta e_{\phi}^{k}+e_{M}^{k-1}N (t_{k })+M\left(\phi ^{k-1}\right)e_{N}^{k}=R_{f}^{k}, \label{4-11a} \\
&\delta e_{N}^{k} =\frac{1}{2}\left(e_{M}^{k-1} \delta\phi (t_{k}) +M\left(\phi ^{k-1}\right) \delta e_{\phi}^{k} \right)+\Delta tR_{g}^{k}, \label{4-11b}
\end{align}
\end{subequations}
where the error terms are defined as
$$e_{M}^{k} = M(\phi(t_{k}))-M\left(\phi ^{k}\right),\quad e_{N}^{k} = N(t_{k})-N ^{k}.$$

{\Ass \label{ass4} Based on the \textbf{Assumption \ref{ass1}}, we  further impose  the following regularity conditions 
\begin{equation}\label{ieq-re}
\begin{aligned}
N\in L^{\infty}(0,T,W^{1, \infty}(\Omega)), \quad  N_{tt} \in L^{2}(0,T, L^{2}(\Omega)).
\end{aligned}
\end{equation}
}
{\lem \label{lemma4-1}
Under the above assumption, the truncation errors are bounded as follows 
\begin{equation}
\begin{aligned}
 \left\|R_{f}^{k}\right\|_{H^{1}}+\left\|R_{g}^{k}\right\| \lesssim \Delta t, \quad 0  \leq k  \leq \frac{T}{\Delta t}.
\end{aligned}
\end{equation}
}
{\lem \label{lemma4-2}
Suppose the following conditions hold:
\begin{enumerate}
    \item \(F(\phi)\) is uniformly bounded from below: \(F(\phi) > -A\) for any \(\phi \in (-\infty, \infty)\);
    \item \(F(\phi) \in C^{2}(-\infty, \infty)\);
    \item There exists a positive constant \(C_{c}\) such that
    \begin{equation}
    \begin{aligned}
    \max_{k \leq K} \left( \left\|\phi(t_{k})\right\|_{L^{\infty}}, \left\|\phi^{k}\right\|_{L^{\infty}} \right) \leq C_{c}.
    \end{aligned}
    \end{equation}
\end{enumerate}
Then, it follows that
\begin{equation}
\begin{aligned}
\left\| M(\phi(t_{k})) - M\left(\phi^{k}\right) \right\| \leq \hat{C}_{c} \left\|\phi(t_{k}) - \phi^{k}\right\|,
\end{aligned}
\end{equation}
for \(k \leq T\), where \(\hat{C}_{c}\) is a positive number that depends only on \(C_{c}\), \(A\), and \(C\).

}

{\lem \label{lemma4-3}
We assume the following conditions:
\begin{enumerate}
    \item \(F(\phi)\) is uniformly bounded from below: \(F(\phi) > -A\) for any \(\phi \in (-\infty, \infty)\);
    \item \(F(\phi) \in C^{3}(-\infty, \infty)\);
    \item There exists a positive constant \(D_{0}\) such that
    \begin{equation}
    \begin{aligned}
    \max_{k \leq K} \left( \left\|\phi(t_{k})\right\|_{L^{\infty}}, \left\|\phi^{k}\right\|_{L^{\infty}}, \left\|\nabla\phi(t_{k})\right\|_{L^{3}} \right) \leq D_{0}.
    \end{aligned}
    \end{equation}
\end{enumerate}
Therefore,
\begin{equation}
\begin{aligned}
\left\|\nabla M(\phi(t_{k})) - \nabla M\left(\phi^{k}\right)\right\| \leq \hat{D}_{0} \left( \left\|\phi(t_{k}) - \phi^{k}\right\| + \left\|\nabla\phi(t_{k}) - \nabla\phi^{k}\right\| \right),
\end{aligned}
\end{equation}
for \(k \leq T\), where \(\hat{D}_{0}\) is a positive number that is only dependent on \(\Omega\), \(D_{0}\), \(A\), and \(C\).

}

The proofs of the above three lemmas are detailed in \cite{2020Convergence}.

{\lem
Given that the solution to the considered model  satisfies \textbf{Assumptions \ref{ass1}} and \textbf{\ref{ass4}}, there exists a general constant \(C\) such that if \(\Delta t \leq C\), the solution \(\mathbf{b}^{k}\) of \textbf{Scheme II }satisfies 
\begin{equation}\label{ieq-lemma}
\left\|\mathbf{b}^{k}\right\|_{L^{\infty}} \leq \Pi_{b}, \quad k = 0, \cdots, \frac{T}{\Delta t}.
\end{equation}
}
\begin{proof}
The process is similar to \textbf{Lemma 3.4}, and we will not go into the details further.
\end{proof}

{\lem \label{lemma4-5}
Under the following assumptions:
\begin{enumerate}
    \item \(F(\phi)\) is uniformly bounded from below: \(F(\phi) > -A\) for any \(\phi \in (-\infty, \infty)\);
    \item \(F(\phi) \in C^{3}(-\infty, \infty)\);
    \item The exact solutions of the reconstructed model by the IEQ mthod  satisfy \textbf{Assumptions \ref{ass1}} and \textbf{\ref{ass4}},
\end{enumerate}
there exists a positive constant \(C_{e}\) such that for \(\Delta t \leq C_{e}\), the solution \(\phi^{k}\) of \textbf{Scheme II} is uniformly bounded as 
\begin{equation}\label{ieq418}
\left\|\phi^{k}\right\|_{L^{\infty}} \leq \Pi_{\phi}, \quad k = 0, 1, \cdots, \frac{T}{\Delta t}.
\end{equation}
}

\begin{proof}
Similarly, we employ mathematical induction to prove this lemma.

\textbf{Step $\star$}. When $k$=0, we  obtain  that $\left\|\phi^{0}\right\|_{L^{\infty}}=\left\|\phi(t_{0})\right\|_{L^{\infty}} \leq \Pi_{\phi_{1}}$. Additionally, we assume that $\left\|\phi^{k-1}\right\|_{L^{\infty}} \leq \Pi_{\phi_{2}}$, for all  $k\leq \frac{T}{\Delta t}$. Subsequently, we will prove that \(\left\|\phi^{k}\right\|_{L^{\infty}} \leq \Pi_{\phi_{3}}\) also holds through the following process.

By taking the $L^{2}$ inner product of  equation (\ref{4-11a}) with  $\lambda \left(\delta e_{\phi}^{k} \right)$ and $ \Delta t M e_{w}^{k}$ respectively, we derive 
\begin{subequations}\label{ieq-1}
\begin{align}
&-\lambda \left(\delta e_{\phi}^{k} ,  e_{w}^{k}\right)+\lambda\varepsilon \left(\nabla e_{\phi}^{k}, \nabla \left(\delta e_{\phi}^{k} \right)\right ) =-\lambda\left(e_{M}^{k-1}N (t_{k })+M\left(\phi ^{k-1}\right)e_{N}^{k}, \delta e_{\phi}^{k} \right) +\lambda\left(R_{f}^{k}, \delta e_{\phi}^{k} \right),\\
&\Delta t M\left\|e_{w}^{k}\right\|^{2}-\varepsilon \Delta t M \left( \nabla e_{\phi}^{k}, \nabla e_{w}^{k}\right) =  \Delta t M\left( e_{M}^{k-1}N (t_{k })+M\left(\phi ^{k-1}\right)e_{N}^{k}, e_{w}^{k}\right) -\Delta t M \left(e_{w}^{k}, R_{f}^{k}\right).
\end{align}
\end{subequations}

By taking the $L^{2}$ inner product of equation  (\ref{4-11b}) with $2\lambda e_{N}^{k}$, we obtain 
\begin{equation}\label{ieq-2}
\begin{aligned}
& \lambda\left(\delta\left\| e_{N}^{k}\right\|^{2}+ \left\|\delta e_{N}^{k}\right\|^{2}\right)=\lambda\left(e_{M}^{k-1}\delta\phi (t_{k})+M\left(\phi ^{k-1}\right)\delta e_{\phi}^{k}, e_{N}^{k}\right)+2\lambda\Delta t\left(R_{g}^{k}, e_{N}^{k}\right).
\end{aligned}
\end{equation}

By combining equations (\ref{4-12}), (\ref{ieq-1}), (\ref{ieq-2}), (\ref{4-14}), (\ref{4-15}), and (\ref{4-16}), we establish 
\begin{equation}\label{ieq-03}
\begin{aligned}
& \lambda M\Delta t\left\|\nabla e_{w}^{k}\right\|^{2} +\frac{\varepsilon }{2} \left(\delta \left\|e_{\phi}^{k}\right\|^{2} +\left\|\delta e_{\phi}^{k} \right\|^{2}\right)+\frac{\lambda\varepsilon }{2} \left(\delta \left\|\nabla e_{\phi}^{k}\right\|^{2} +\left\|\delta(\nabla e_{\phi}^{k}) \right\|^{2}\right)+ \Delta tM\left\|e_{w}^{k}\right\|^{2}+\frac{1}{2}\left(\delta\left\|e_{b}^{k}\right\|^{2} +\left\|\delta e_{b}^{k} \right\|^{2}\right)\\
&+\frac{ \Delta t}{\sigma\mu}\left\|\nabla\times e_{b}^{k}\right\|^{2}+\frac{1}{2}\left(\delta \left\|e_{u}^{k}\right\|^{2} + \left\|\tilde{e}_{u}^{k}- e_{u}^{k-1}\right\|^{2}\right) +\Delta t\nu\left\|\nabla\tilde{e}_{u}^{k}\right\|^{2}+\frac{\Delta t^{2}}{2}\left(\delta \left\|\nabla e_{p}^{k}\right\|^{2} \right)+\lambda\left(\delta\left\|e_{N}^{k}\right\|^{2} + \left\|\delta e_{N}^{k} \right\|^{2}\right)\\
&= (\rm term \, A)+ (\rm term \, B)+ (\rm term \, D) +(\rm term \, E)+  (\rm term \, F)\\
&\quad  + \lambda \left(R_{f}^{k}, \delta e_{\phi}^{k} \right)- \Delta t M \left(R_{f}^{k}, e_{w}^{k}\right)+ \Delta t \left(R_{c}^{k}, e_{b}^{k} \right)+ \Delta t\left(R_{d}^{k}, \tilde{e}_{u}^{k}\right)   &(: \text{term } H)\\
&\quad -\lambda\left(e_{M}^{k-1}N(t_{k }), \delta e_{\phi}^{k} \right)+\Delta t M\left( e_{M}^{k-1}N (t_{k })+M(\phi ^{k-1})e_{N}^{k}, e_{w}^{k}\right)  &(: \text{term } I)\\
&\quad  +\lambda\left(e_{M}^{k-1}\delta\phi (t_{k}) , e_{N}^{k}\right)+2 \lambda\Delta t\left(R_{g}^{k}, e_{N}^{k}\right).  &(: \text{term } J)
\end{aligned}
\end{equation}

We derive the following estimates for the right-hand sides  using \textbf{Assumption \ref{ass4}} and \textbf{Lemma \ref{lemma4-1}}. For the \text{term} H, we first give the following estimation 
\begin{equation*}
\begin{aligned}
\left(R_{f}^{k}, \delta e_{\phi}^{k} \right)&= \Delta t \left(R_{f}^{k}, \frac{\delta e_{\phi}^{k} }{\Delta t} \right)  \\
&\leq\Delta t\left| \left(R_{f}^{k}, M\Delta e_{w}^{k}+R_{a}^{k}- \nabla\cdot \left(\phi(t_{k-1})\textbf{v}(t_{k-1})-\phi^{k-1}\textbf{v}^{k-1}\right)+\Delta t\lambda\nabla\cdot   \left(\phi(t_{k-1})^{2}\nabla w(t_{k})-\left(\phi^{k-1}\right)^{2}\nabla w^{k}\right)  \right) \right| \\
& \leq \Delta tM \left|\left(\nabla R_{f}^{k}, \nabla e_{w}^{k}\right) \right| +\Delta t\left|\left( R_{f}^{k},  R_{a}^{k}\right)  \right|+ \Delta t\left|\left(\nabla R_{f}^{k},  e_{\phi}^{k-1}\textbf{v}(t_{k-1}) +\phi^{k-1}e_{u}^{k-1}\right)  \right|\\
&\quad + \Delta t^{2} \lambda \left| \left(\nabla R_{f}^{k},  \phi(t_{k-1})^{2} \nabla e_{w}^{k}   + \phi(t_{k-1})^{2} \nabla w^{k} -\left(  \phi^{k-1} \right)^{2}\nabla w^{k}  \right) \right| \\
&\leq \Delta tM \left\|\nabla R_{f}^{k}\right\|\left\|\nabla e_{w}^{k}\right\| +\Delta t\left\|R_{f}^{k}\right\|\left\|R_{a}^{k}\right\|+\Delta t\left\|\nabla R_{f}^{k}\right\| \left( \left\|e_{\phi}^{k-1}\right\|\left \|\textbf{v}(t_{k-1})\right\|_{ L^{\infty}} +\left\|\phi^{k-1}\right\|_{L^{\infty}} \left\|e_{u}^{k-1}\right\|  \right)     \\
&\quad  + \Delta t^{2}\lambda \left\|\nabla R_{f}^{k}\right\|\left(\left\|\phi(t_{k-1})^{2}\right\|_{L^{\infty}}\left\|\nabla e_{w}^{k} \right\| +\left\|\phi(t_{k-1})^{2}\right\|_{L^{\infty}}\left\|\nabla w^{k}\right\|+  \left\|\left(\phi^{k-1}\right)^{2}\right\|_{L^{\infty}}\left\|\nabla w^{k}\right\|   \right)\\
&\leq \frac{\lambda M \Delta t }{16}\left\|\nabla e_{w}^{k}\right\|^{2}+ C\left(\Delta t\left\|e_{\phi}^{k-1}\right\|^{2}+ \Delta t\left\|e_{u}^{k-1}\right\|^{2} +\Delta t^{2}+\Delta t^{3}\left\|\nabla w^{k}\right\| ^{2}  \right).
\end{aligned}
\end{equation*}
Therefore, we obtain the estimate of the \text{term } H as 
\begin{equation}\label{ieq-12}
\begin{aligned}
(\text{term } H) &\leq \frac{\lambda M \Delta t }{16}\left\|\nabla e_{w}^{k}\right\|^{2}+ \frac{\Delta tM}{4}\left\| e_{w}^{k}\right\|^{2} +\frac{\Delta t\nu}{8}\left\|\nabla\tilde{e}_{u}^{k}\right\|^{2}+ C\left( \Delta t\left\|e_{\phi}^{k-1}\right\|^{2}+ \Delta t\left\|e_{u}^{k-1}\right\|^{2}+  \Delta t\left\| e_{b}^{k}\right\|^{2} +\Delta t^{2}+\Delta t^{3}\left\|\nabla w^{k}\right\| ^{2} \right).
\end{aligned}
\end{equation}

For the \text{term } I=$-\lambda\left(e_{M}^{k-1}N(t_{k }), \delta e_{\phi}^{k} \right)+\Delta t M\left( e_{M}^{k-1}N (t_{k })+M\left(\phi ^{k-1}\right)e_{N}^{k}, e_{w}^{k}\right)=I_{1}+I_{2}$, we derive the following estimates  using equation (\ref{3-20a}) as
\begin{equation*}
\begin{aligned}
I_{1}&=-\lambda\left(e_{M}^{k-1}N(t_{k }), \delta e_{\phi}^{k} \right)\\
&=-\lambda\Delta t\left(e_{M}^{k-1}N(t_{k }), \frac{\delta e_{\phi}^{k} }{\Delta t}\right)\\
&\leq \left|\lambda\Delta t\left(e_{M}^{k-1}N(t_{k }), M\Delta e_{w}^{k}+R_{a}^{k}- \nabla\cdot \left(\phi(t_{k-1})\textbf{v}(t_{k-1})-\phi^{k-1}\textbf{v}^{k-1}\right)+\Delta t\lambda\nabla\cdot   \left(\phi(t_{k-1})^{2}\nabla w(t_{k})-\left(\phi^{k-1}\right)^{2}\nabla w^{k}\right) \right)\right|.
\end{aligned}
\end{equation*}
By \textbf{Lemma \ref{lemma4-2}} and \textbf{Lemma \ref{lemma4-3}}, we have 
\begin{equation*}
\begin{aligned}
\left|\lambda\Delta t\left(e_{M}^{k-1}N(t_{k }), M\Delta e_{w}^{k}\right)\right|&\leq\left|M\lambda\Delta t\left(\nabla (e_{M}^{k-1}N(t_{k })),  \nabla e_{w}^{k}\right)\right|\\
&\leq \left|M\lambda\Delta t\left(N(t_{k })\nabla e_{M}^{k-1}+ \nabla N(t_{k })e_{M}^{k-1}, \nabla e_{w}^{k}\right)\right|\\
&\leq M\lambda \Delta t\left\|N(t_{k })\right\|_{L^{\infty}}\left\|\nabla e_{M}^{k-1}\right\| \left\|\nabla e_{w}^{k}\right\|+M\lambda \Delta t  \left\|\nabla N(t_{k })\right\|_{L^{\infty}}\left\| e_{M}^{k-1}\right\| \left\|\nabla e_{w}^{k}\right\|\\
&\leq \frac{\lambda M\Delta t}{16}\left\|\nabla e_{w}^{k }\right \|^{2}+C\left(\Delta t\left\|\nabla e_{\phi}^{k-1} \right\|^{2}+ \Delta t\left\| e_{\phi}^{k-1} \right\|^{2} \right).
\end{aligned}
\end{equation*}
Then, we derive the estimates for  $\rm I_{1}$ and $\rm I_{2}$ as 
\begin{equation*}
\begin{aligned}
&I_{1}\leq \frac{\lambda M\Delta t}{16}\left\|\nabla e_{w}^{k } \right\|^{2}+ C\left(\Delta t\left\|\nabla e_{\phi}^{k-1} \right\|^{2}+\Delta t^{2}\left\|\nabla e_{w}^{k }\right \|^{2}+ \Delta t\left\| e_{\phi}^{k-1} \right\|^{2}+\Delta t \left\| e_{u}^{k-1}\right \|^{2} +\Delta t^{2}+ \Delta t^{3}\left\|\nabla w^{k}\right\| ^{2} \right),\\
&I_{2}\leq  \frac{\Delta tM}{8}\left \| e_{w}^{k } \right\|^{2}+ C\left( \Delta t\left\| e_{\phi}^{k-1}\right \|^{2}+\Delta t\left\|e_{N}^{k}\right\|^{2}  \right),
\end{aligned}
\end{equation*}
where $\left\|M\left(\phi ^{k-1}\right)\right\|_{L^{\infty}}$ is bounded due to  the fact that $\left\|\phi^{k-1}\right\|_{L^{\infty}}\leq\Pi_{\phi_{2}}$. Hence, we obtain 
\begin{equation}\label{ieq-13}
\begin{aligned}
(\text{term } I) &\leq \frac{\lambda M\Delta t}{16}\left\|\nabla e_{w}^{k } \right\|^{2}+ \frac{\Delta tM}{8}\left \| e_{w}^{k } \right\|^{2}\\
&\quad + C\left(\Delta t\left\|\nabla e_{\phi}^{k-1} \right\|^{2}+\Delta t^{2}\left\|\nabla e_{w}^{k }\right \|^{2}+ \Delta t\left\| e_{\phi}^{k-1} \right\|^{2}+\Delta t \left\| e_{u}^{k-1}\right \|^{2}+\Delta t\left\|e_{N}^{k}\right\|^{2}  +\Delta t^{2}+\Delta t^{3}\left\|\nabla w^{k}\right\| ^{2}  \right).
\end{aligned}
\end{equation}

The estimate of the term J is derived as 
\begin{equation}\label{ieq-14}
\begin{aligned}
(\text{term } J) &\leq \lambda \left\|e_{M}^{k-1}\right\|_{L^{4}}\left\|\delta\phi (t_{k}) \right\|_{L^{4}}\left\|e_{N}^{k}\right\|+2\lambda \Delta t\left\|R_{g}^{k}\right\|\left\|e_{N}^{k}\right\|\\
&\leq C\left(\Delta t\left\|e_{\phi}^{k-1}\right\|^{2}+\Delta t\left\|\nabla e_{\phi}^{k-1}\right\|^{2}+\Delta t\left\|e_{N}^{k}\right\|^{2} +\Delta t^{2} \right).\\
\end{aligned}
\end{equation}

By combining equations (\ref{ieq-12})-(\ref{ieq-14}) with equation (\ref{ieq-03}), we have 
\begin{equation}\label{ieq-023}
\begin{aligned}
&\frac{\lambda M\Delta t}{2}\left(\delta\left\|\nabla e_{w}^{k}\right\|^{2}   \right) +\frac{\varepsilon }{2} \left(\delta\left\|e_{\phi}^{k}\right\|^{2} +\left\|\delta e_{\phi}^{k} \right\|^{2}\right) +\frac{\lambda\varepsilon}{2}\left(\delta \left\|\nabla e_{\phi}^{k}\right\|^{2}  +\left\|\delta( \nabla e_{\phi}^{k} )\right\|^{2}\right)  + \frac{\Delta tM}{2}\left\|e_{w}^{k}\right\|^{2}+\frac{1}{2}\left(\delta \left\|e_{b}^{k}\right\|^{2} +\left\|\delta e_{b}^{k} \right\|^{2}\right)\\
& +\frac{ \Delta t}{2\sigma\mu}\left\|\nabla\times e_{b}^{k}\right\|^{2}+ \frac{\Delta t\nu}{2}\left\|\nabla\tilde{e}_{u}^{k}\right\|^{2}+\frac{1}{2}\left(\delta\left\|e_{u}^{k}\right\|^{2} + \left\|\tilde{e}_{u}^{k}- e_{u}^{k-1}\right\|^{2}\right)  +\frac{\Delta t^{2}}{2}\left(\delta\left\|\nabla e_{p}^{k}\right\|^{2} \right)+\lambda\left(\delta\left\|e_{N}^{k}\right\|^{2} + \left\|\delta e_{N}^{k} \right\|^{2}\right)\\
& \leq C \Delta t \Big(\Delta t\left\|\nabla e_{w}^{k }\right\|^{2}+\left\| e_{\phi}^{k}\right\|^{2}+\left\|  e_{\phi}^{k-1}\right\|^{2} +\left\|\nabla e_{\phi}^{k-1}\right\|^{2} +\left\|\nabla e_{\phi}^{k}\right\|^{2}  +\left\|e_{b}^{k-1}\right\|^{2}+\left\|e_{b}^{k }\right\|^{2} + \left\| e_{u}^{k}\right\|^{2}+\left\| e_{u}^{k-1}\right\|^{2}\\
&\quad+\Delta t^{2}\left\|\nabla e_{p}^{k}\right\|^{2}+ \Delta t^{2}\left\|\nabla e_{p}^{k-1}\right\|^{2}  +\left\|e_{N}^{k}\right\|  ^{2}   \Big)+ C\left(\Delta t^{3}\left\|\nabla w^{k}\right\| ^{2}+\Delta t^{3}\left\|\nabla\times \textbf{b}^{k}\right\| ^{2}  \right)+ C\Delta t^{2},\\
\end{aligned}
\end{equation}
where we add the term $\frac{\lambda M\Delta t}{2}\left\|\nabla e_{w}^{k-1}\right\|^{2}$ such that the inequality  still holds, consistent with \textbf{Lemma 3.1}.

Summing the above inequality from \(k=0\) to \(m\),  and using the initial conditions
$$\left\|\nabla e_{w}^{0}\right\|^{2}=\left\|e_{\phi}^{0}\right\|^{2}=\left\|\nabla e_{\phi}^{0}\right\|^{2}=\left\|e_{b}^{0}\right\|^{2}=\left\|e_{u}^{0}\right\|^{2}=\left\|\nabla e_{p}^{0}\right\|^{2}=\left\|e_{N}^{0}\right\|^{2}=0,$$
combining \textbf{Lemma 3.1} and \textbf{Remark \ref{Remark4-1}},we derive the following inequality 
\begin{equation}\label{ieqs-1}
\begin{aligned}
&\frac{\varepsilon}{2} \left\|e_{\phi}^{m}\right\|^{2}+\frac{\lambda\varepsilon}{2} \left\|\nabla e_{\phi}^{m}\right\|^{2}+\frac{1}{2}\left\|e_{b}^{m}\right\|^{2}+\frac{1}{2}\left\| e_{u}^{m}\right\|^{2}+\frac{\Delta t^{2}}{2} \left\|\nabla e_{p}^{m}\right\|^{2}+  \frac{\lambda M\Delta t}{2}\left\|\nabla e_{w}^{m}\right\|^{2}+\lambda \left\|e_{N}^{m}\right\|^{2} + \Delta t\sum_{k=0}^{m}\Bigg(\frac{\varepsilon }{2\Delta t}  \left\|\delta e_{\phi}^{k}\right\|^{2} \\
&+\frac{\lambda\varepsilon}{2\Delta t} \left\|\delta(\nabla e_{\phi}^{k})\right\|^{2}+ \frac{1}{2\sigma\mu}\left\|\nabla\times e_{b}^{k}\right\|^{2}+ \frac{1}{2\Delta t}\left\|\delta e_{b}^{k}\right\|^{2}  + \frac{1}{2\Delta t}\left\|\tilde{e}_{u}^{k}- e_{u}^{k-1}\right\|^{2} +\frac{ M }{2}\left\|e_{w}^{k}\right\|^{2}+ \frac{\nu}{2} \left\|\nabla\tilde{e}_{u}^{k}\right\|^{2}+\frac{\lambda}{\Delta t}\left\|\delta e_{N}^{k} \right\|^{2} \Bigg)\\
&\leq C_{f}\Delta t^{2},
\end{aligned}
\end{equation}
where the parameters  $\hat{C_{e}}$ and $C_{f}$ are two positive general constants, for  $\Delta t\leq \hat{C_{e}}$.

\textbf{Step $\star\star$}.  Then, we evaluate the estimate of $\left\|\phi^{k}\right\|_{H^{2}}$.

Combining the above inequality~(\ref{ieqs-1}) with the following equation~(\ref{4-2b}):
\[
w^{k} = -\varepsilon \Delta \phi^{k} + M(\phi^{k-1})N^{k},
\]
we find that there exists a positive constant $C_{g}$ such that 
\begin{equation}\label{ieqs-2111}
\begin{aligned}
 \left\| \phi^{k} \right\|_{H^{2}} &\leq \frac{1}{\varepsilon}\left\| w^{k} \right\| + \frac{1}{\varepsilon} \left\| M(\phi^{k-1})N^{k} \right\| \\
&\leq \frac{1}{\varepsilon}\left\| e_{w}^{k} \right\| + \frac{1}{\varepsilon}\left\| w(t_{k}) \right\| + \frac{1}{\varepsilon}\left\| M(\phi^{k-1}) \right\|_{L^{\infty}} \left( \left\| e_{N}^{k} \right\| + \left\| N(t_{k}) \right\| \right) \\
&\leq C_{g}.
\end{aligned}
\end{equation}

Thus, we can obtain
\begin{equation}\label{ieqs-431}
\begin{aligned}
\left\|e_{\phi}^{k}\right\|_{H^{2}}\leq \left\|\phi^{k}\right\|_{H^{2}} +\left\|\phi(t_{k})\right\|_{H^{2}}\leq C_{h}.
\end{aligned}
\end{equation}

From inequalities (\ref{ieqs-1}) and (\ref{ieqs-431}), we have 
\begin{equation*}
\begin{aligned}
\left\|\phi^{k}\right\|_{L^{\infty}}&\leq \left\|e_{\phi}^{k}\right\|_{L^{\infty}}+\left\|\phi (t_{k})\right\|_{L^{\infty}} \\
&\leq C_{i} \left\|e_{\phi}^{k}\right\|_{H^{2}}^{3/4} \left\|e_{\phi}^{k}\right\|^{1/4} + \left\|\phi (t_{k})\right\|_{L^{\infty}}\\
&\leq C_{i}  C_{h}^{3/4}\left( \frac{2}{\varepsilon} C_{f}\right)^{1/8}\Delta t^{1/4}+\left\|\phi (t_{k})\right\|_{L^{\infty}}\\
&\leq \Pi_{\phi_{3}},
\end{aligned}
\end{equation*}
where  $C_{i}  C_{h}^{3/4} \left( \frac{2}{\varepsilon} C_{f}\right)^{1/8}\Delta t^{1/4}\leq 1 $, namely, $\Delta t\leq C_{e}=\frac{1}{C_{i} ^{4}  C_{h}^{3 } \left( \frac{2}{\varepsilon} C_{f}\right)^{1/2}}$. Thus, it is established that $\left\|\phi^{k}\right\|_{L^{\infty}}\leq  \Pi_{\phi_{3}}$.

From the above mentioned process, we derive the inequality (\ref{ieq418}) that $\Delta t\leq C_{e}$, for
\begin{equation}
\begin{aligned}
&\Pi_{\phi}=\max\left\{\Pi_{\phi_{1}}, \Pi_{\phi_{2}}, \Pi_{\phi_{3}}\right\},\quad C_{e} =\min  \left\{\hat{C_{e}}, \frac{1}{C_{17} ^{4}  C_{16}^{3 } C_{f}^{1/2}},  \frac{1}{C_{i} ^{4}  C_{h}^{3 } \left( \frac{2}{\varepsilon} C_{f}\right)^{1/2} }\right\}.
\end{aligned}
\end{equation}
\end{proof}
 
\begin{The}
Suppose the solution of the \textbf{Scheme II} satisfies the \textbf{Assumption \ref{ass4}}. The  numerical scheme is unconditionally convergent and has the following error estimates: for  $0\leq m\leq \frac{T}{\Delta t}$,
\begin{equation}\label{e2}
\begin{aligned}
&\frac{\lambda\varepsilon}{2} \left\|\nabla e_{\phi}^{m}\right\|^{2}+\frac{1}{2}\left\|e_{b}^{m}\right\|^{2}+\frac{1}{2}\left\| e_{u}^{m}\right\|^{2}+\frac{\Delta t^{2}}{2} \left\|\nabla e_{p}^{m}\right\|^{2}+  \frac{\lambda M\Delta t}{2}\left\|\nabla e_{w}^{m}\right\|^{2}+\lambda \left\|e_{N}^{m}\right\|^{2}\\
& + \Delta t\sum_{k=0}^{m}\left( \frac{\lambda\varepsilon}{2\Delta t} \left\|\delta(\nabla e_{\phi}^{k})\right\|^{2}+ \frac{1}{2\sigma\mu}\left\|\nabla\times e_{b}^{k}\right\|^{2}+ \frac{1}{2\Delta t}\left\|\delta e_{b}^{k}\right\|^{2}  + \frac{1}{2\Delta t}\left\|\tilde{e}_{u}^{k}- e_{u}^{k-1}\right\|^{2}  + \frac{\nu}{2} \left\|\nabla\tilde{e}_{u}^{k}\right\|^{2}+\frac{\lambda}{\Delta t}\left\|\delta e_{N}^{k} \right\|^{2} \right)\\
&\leq C \Delta t^{2}.
\end{aligned}
\end{equation}

\end{The}

\begin{proof}
\textbf{Case 1}.
Since $\left\|\textbf{b}^{k}\right\|_{L^{\infty}}\leq\Pi_{b}  $ and $\left\|\phi^{k}\right\|_{L^{\infty}}\leq\Pi_{\phi} $ hold for any $0\leq k\leq \frac{T}{\Delta t}$, by the proof of \textbf{ Step $\star$} of \textbf{Lemma \ref{lemma4-5}}, we obtain that equation \eqref{ieqs-1} is valid for $0\leq m\leq \frac{T}{\Delta t}$, provided that $\Delta t\leq C_{e}$.

\textbf{Case 2}. If  $ \Delta t  \geq C_{e}$, using  \textbf{Remark 4.1} and \textbf{Assumption \ref{ass4}}, we get
\begin{equation}
\begin{aligned}
&\frac{\lambda\varepsilon}{2} \left\|\nabla e_{\phi}^{m}\right\|^{2}+\frac{1}{2}\left\|e_{b}^{m}\right\|^{2}+\frac{1}{2}\left\| e_{u}^{m}\right\|^{2}+\frac{\Delta t^{2}}{2} \left\|\nabla e_{p}^{m}\right\|^{2}+  \frac{\lambda M\Delta t}{2}\left\|\nabla e_{w}^{m}\right\|^{2}+\lambda \left\|e_{N}^{m}\right\|^{2}\\
& + \Delta t\sum_{k=0}^{m}\left( \frac{\lambda\varepsilon}{2\Delta t} \left\|\delta(\nabla e_{\phi}^{k})\right\|^{2}+ \frac{1}{2\sigma\mu}\left\|\nabla\times e_{b}^{k}\right\|^{2}+ \frac{1}{2\Delta t}\left\|\delta e_{b}^{k}\right\|^{2}  + \frac{1}{2\Delta t}\left\|\tilde{e}_{u}^{k}- e_{u}^{k-1}\right\|^{2} + \frac{\nu}{2} \left\|\nabla\tilde{e}_{u}^{k}\right\|^{2}+\frac{\lambda}{\Delta t}\left\|\delta e_{N}^{k} \right\|^{2} \right)\\
&\leq C_{ieq}\leq \frac{C_{ieq}}{C_{e}^{2}} C_{e}^{2} \leq\frac{C_{ieq}}{C_{e}^{2}}  \Delta t^{2}\leq C\Delta t^{2}.
\end{aligned}
\end{equation}

Thus, the proof is completed by considering the above two cases.
\end{proof}

{\Rem (\textbf{The Scheme III})
Obviously,  $N^{k}$ must be updated at each step of the \textbf{Scheme II}, as the energy conservation property of the algorithm may not correspond to the original energy's conservation. The improved IEQ (IIEQ) algorithm can avoid this problem, as demonstrated in \cite{CHEN2022114405}. The details are given by 

\textbf{Step IIEQ-1}. Compute $\phi^{k}$ and $w^{k}$ from 
\begin{subequations}\label{4-2iieq}
\begin{align}
&\frac{\phi^{k}- \phi^{k-1}}{\Delta t}+\nabla\cdot \left(\phi^{k-1} \textbf{v}^{k-1}\right)-\Delta t\lambda\nabla\cdot \left(\left(\phi^{k-1}\right)^{2}\nabla w^{k}\right)=M\Delta w^{k},\\
&w^{k}=-\varepsilon \Delta\phi^{k}+ \frac{1}{\varepsilon } \phi^{k-1} \tilde{N}^{k}+ S^{\prime}  \left(\phi^{k}-\phi^{k-1}\right),\\
&\tilde{N}^{k}-N \left(\phi^{k-1}\right)  =2 \phi^{k-1} \left( \phi^{k}-\phi^{k-1}\right), \label{4-28c}\\
&N \left(\phi^{k-1}\right)=\left(\phi^{k-1}\right)^{2}-1,\label{4-28d}
\end{align}
\end{subequations}
along with the boundary conditions $\frac{\partial\phi^{k}}{\partial \textbf{n}}|_{\partial\Omega}$=0 and  $\frac{\partial w^{k}}{\partial \textbf{n}}|_{\partial\Omega}$=0.  The term $S^{\prime}\left(\phi^{k}-\phi^{k-1}\right)$  is a first-order stabilized term, where  $S^{\prime}$ represents a positive stabilization parameter. It is noteworthy that we can calculate  $\tilde{N}^{k}$ using the equations (\ref{4-28c}) and (\ref{4-28d}), thereby allowing us to avoid iterating $N^{k}$ within the \textbf{Scheme II}.
}

\begin{The} Combining the condition $S^{\prime}\geq  \frac{ 1}{2\varepsilon}  \max\left\{\left|\phi^{k-1}\right|^{2}-2 \left|\phi^{k-1}\right|^{2}-1   \right\}  $, and setting the source term  $\textbf{f}$=$\textbf{0}$, the \textbf{Scheme III} (\textbf{Step IIEQ-1}, \textbf{Step 2}-\textbf{Step 4}) is unconditionally energy stable in the sense that
\begin{equation*}
\begin{aligned}
&E^{k}_{iieq}- E^{k-1}_{iieq} \leq 0, \\
&E^{k}_{iieq}=\frac{\lambda\varepsilon}{2}\left\|\nabla\phi^{k}\right\|^{2}+\frac{1}{2\mu}\left\|\textbf{b}^{k}\right\|^{2}+ \frac{1}{2}\left\|\textbf{v}^{k}\right\|^{2}+\frac{ \lambda}{4\varepsilon}\left \|N\left(\phi ^{k}\right)\right\|^{2} +\frac{\Delta t^{2}}{2}\left\|\nabla p^{k}\right\|^{2}.
\end{aligned}
\end{equation*}

\end{The}

\begin{proof}
The proof closely resembles the one in \cite{CHEN2022114405}, as well as \textbf{Theorem \ref{theorem4-1}}, thus we omit the details.
\end{proof}

\section{Numerical examples}

In this section, we present several numerical examples to validate our theoretical results and
illustrate the performance of the proposed scheme. Here, we employ the inf-sup stable pair $P_{1}-P_{1}$ to discretize the phase-field variable and chemical potential, the inf-sup stable MINI element pair $(P_{1}^{b}, P_{1})$ for  the velocity field and pressure field, and the $P_{1}$ finite element for the magnetic field. The optimal convergence results are expected as reported in \cite{2022Highly}:
\begin{equation}\label{6GU-a}
\begin{aligned}
&\|\phi(t_{n})-\phi_{h}^{n}\|\lesssim \Delta t+h^{2}\lesssim h^{2}, \, \|\nabla\phi(t_{n})-\nabla\phi_{h}^{n}\|\lesssim \Delta t+h\lesssim h,\\
&\|\textbf{v}(t_{n})-\textbf{v}_{h}^{n}\|\lesssim \Delta t+h^{2}\lesssim h^{2}, \, \|\nabla\textbf{v}(t_{n})-\nabla\textbf{v}_{h}^{n}\|+\|p(t_{n})-p_{h}^{n}\|\lesssim \Delta t+h\lesssim h,\\
&\|\textbf{b}(t_{n})-\textbf{b}_{h}^{n}\|\lesssim \Delta t+h^{2}\lesssim h^{2},\, \|\nabla\textbf{b}(t_{n})-\nabla\textbf{b}_{h}^{n}\| \lesssim \Delta t+h\lesssim h.
\end{aligned}
\end{equation}

\subsection{A smooth solution}
This example is  intended to validate the convergence orders of three schemes: \textbf{Scheme I}, \textbf{Scheme II}, and \textbf{Scheme III}. The considered domain is $\Omega$=$[0, 1]^{2}$, with all parameters set as follows: $\nu$=$\mu$=$\lambda$=$\sigma$=$M$=$\varepsilon$=$S$=1, and $S^{\prime}$=0.
We choose the right-hand sides, initial conditions, and boundary conditions such that the two-phase MHD system admits the following exact solution 
\begin{eqnarray*}
\left\{
\begin{array}{ll}
\phi= \rm sin^2(\pi x) sin^2(\pi y)sin(t),\\
\textbf{v}=\left(\rm x^2(x - 1)^2y(y - 1)(2y -1),\ -\rm y^2(y - 1)^2x(x - 1)(2x - 1) \right)\rm cos(t),\\
p= \rm (2x-1)(2y -1)cos(t) ,\\
\textbf{b}=\left(\rm sin(\pi x)cos(\pi y),  -\rm sin(\pi y) cos(\pi x)\right)\rm cos(t). \\
\end{array}
\right.
\end{eqnarray*}

For simplicity, we verify both the time and space convergence orders at the end time $T$=1 with  $\Delta t$ =$O(h^{2})$.  The convergence results of the \textbf{Scheme I} are shown in \textbf{Table 1}, while the numerical results of the \textbf{Scheme II} and the \textbf{Scheme III} are presented in \textbf{Tables 2-3}. From  \textbf{Tables 1-3}, we  observe that  $\|\phi(t_{n})-\phi_{h}^{n}\|\lesssim  h^{2}$, $\|\nabla\phi(t_{n})-\nabla\phi_{h}^{n}\| \lesssim h$, $\|\textbf{v}(t_{n})-\textbf{v}_{h}^{n}\| \lesssim h^{2}$, $\|\textbf{b}(t_{n})-\textbf{b}_{h}^{n}\| \lesssim h^{2}$,  $\|\nabla\textbf{b}(t_{n})-\nabla\textbf{b}_{h}^{n}\| \lesssim h$,  and  $\|\nabla\textbf{v}(t_{n})-\nabla\textbf{v}_{h}^{n}\|\lesssim h^{2}$, and $\|p(t_{n})-p_{h}^{n}\|\lesssim h^{2}$.

\begin{table}[hpt]\label{1-table1}
\tabcolsep 1.2mm {\footnotesize\textbf{Table 1:} Convergence results for  Scheme I.}
\begin{center}
\begin{tabular}{ccccccccccccccccccccccccccccccc}
\hline
h & $\|\phi(t_{n})-\phi_{h}^{n}\|$ & rate & $\|\nabla (\phi(t_{n})-\phi_{h}^{n})\|$ & rate & $\|\textbf{v}(t_{n})-\textbf{v}_{h}^{n}\|$ & rate & $\|\nabla (\textbf{v}(t_{n})-\textbf{v}_{h}^{n})\|$ & rate &\\ \hline
 1/8 & 0.0972 &1.41 & 0.6314  & 1.18 &0.0095 & 2.80 &0.2266 & 2.35& \\
 1/16 &0.0274 & 1.83 & 0.2649  & 1.25  &0.0030 & 1.65 & 0.0484 & 2.23& \\
 1/32 &0.0071 & 1.96 & 0.1224 & 1.11  & 0.0008  & 1.84  & 0.0128 & 1.92 & \\
 1/64 & 0.0018  &1.99 & 0.0597  & 1.03  & 0.0002  & 1.96 &0.0033 & 1.96&\\
\hline
h &$\|\textbf{b}(t_{n})-\textbf{b}_{h}^{n}\|$ & rate & $\|\nabla (\textbf{b}(t_{n})-\textbf{b}_{h}^{n})\|$ & rate & $\|p(t_{n})-p_{h}^{n}\|$ & rate & \\ \hline
  1/8 & 0.0168 &  1.80  &  0.3315 & 0.97   & 0.9722& 0.53 & \\
  1/16 & 0.0045  & 1.91  &  0.1664 &  0.99  & 0.2995& 1.70& \\
  1/32 & 0.0011   & 2.02 & 0.0833   & 1.00  & 0.0784 & 1.93 & \\
  1/64 & 0.0003  & 2.00  &  0.0417 &  1.00 & 0.0198 & 1.98 &\\
\hline
\end{tabular}
\end{center}
\end{table}

\begin{table}[hpt]\label{2-table1}
\tabcolsep 1.2mm {\footnotesize\textbf{Table 2:} Convergence results for Scheme II.}
\begin{center}
\begin{tabular}{ccccccccccccccccccccccccccccccc}
\hline
h & $\|\phi(t_{n})-\phi_{h}^{n}\|$ & rate & $\|\nabla (\phi(t_{n})-\phi_{h}^{n})\|$ & rate & $\|\textbf{v}(t_{n})-\textbf{v}_{h}^{n}\|$ & rate & $\|\nabla (\textbf{v}(t_{n})-\textbf{v}_{h}^{n})\|$ & rate &\\ \hline
 1/8 & 0.3083  & 1.41  &  0.3830  &  1.18 & 4.5393 & 2.80  & 14.6861 &  2.35 & \\
 1/16 & 0.0870  & 1.82  &  0.1608   & 1.25    &1.4461  & 1.65   & 3.1365   & 2.23  & \\
 1/32 & 0.0227 & 1.94  &  0.0743 & 1.11  &  0.4032   &  1.84  &  0.8268 &  1.92  & \\
 1/64 &  0.0060   & 1.93 &  0.0363   &  1.03  &  0.1040  & 1.95  &0.2138  & 1.95 &\\
\hline
h &$\|\textbf{b}(t_{n})-\textbf{b}_{h}^{n}\|$ & rate & $\|\nabla (\textbf{b}(t_{n})-\textbf{b}_{h}^{n})\|$ & rate & $\|p(t_{n})-p_{h}^{n}\|$ & rate & \\ \hline
  1/8 &  0.0439  &  1.80  &  0.1953 &  0.97    &5.4031 &  0.53  & \\
  1/16 &  0.0114   &  1.94  & 0.0980    &  0.99   & 1.6730  & 1.70 & \\
  1/32 &  0.0029  &  1.98  &  0.0491   &  1.00  & 0.4469  & 1.90   & \\
  1/64 &  0.0007  &  2.00 & 0.0245   &  1.00  &  0.1220 &1.87   &\\
\hline
\end{tabular}
\end{center}
\end{table}

\begin{table}[hpt]\label{3-table1}
\tabcolsep 1.2mm {\footnotesize\textbf{Table 3:} Convergence results for  Scheme III.}
\begin{center}
\begin{tabular}{ccccccccccccccccccccccccccccccc}
\hline
h & $\|\phi(t_{n})-\phi_{h}^{n}\|$ & rate & $\|\nabla (\phi(t_{n})-\phi_{h}^{n})\|$ & rate  & $\|\textbf{v}(t_{n})-\textbf{v}_{h}^{n}\|$ & rate & $\|\nabla (\textbf{v}(t_{n})-\textbf{v}_{h}^{n})\|$ & rate &\\ \hline
 1/8 &  0.3082   & 1.41   &  0.3828    & 1.18   &  4.5418 &  2.8009  &  14.6802  &  2.35  & \\
 1/16 & 0.0868  & 1.83  &    0.1606    &  1.25    &  1.4482 &  1.6490    & 3.1386    &  2.23  & \\
 1/32 & 0.0224   &  1.96   &  0.0742  &  1.11  &   0.4035   &   1.8434   &  0.8269  &  1.92   & \\
 1/64 &  0.0056     &  1.99    &   0.0362   &  1.03  &  0.1038   &  1.9596  & 0.2129  & 1.96  &\\
\hline
h &$\|\textbf{b}(t_{n})-\textbf{b}_{h}^{n}\|$ & rate & $\|\nabla (\textbf{b}(t_{n})-\textbf{b}_{h}^{n})\|$ & rate & $\|p(t_{n})-p_{h}^{n}\|$ & rate & \\ \hline
  1/8 &   0.0439  &  1.80   &  0.1953 &    0.97    &  5.39381 &   0.53   & \\
  1/16 &   0.0114  &  1.94     &  0.0980    &  0.99     &   1.66086  & 1.70  & \\
  1/32 &  0.0029   &   1.98    &  0.0491   &   1.00  &  0.434808  & 1.93    & \\
  1/64 &   0.0007    & 2.00   &  0.0245  &     1.00 &  0.109956  & 1.98   &\\
\hline
\end{tabular}
\end{center}
\end{table}

\subsection{Spinodal decomposition}
In this example, we simulate the benchmark problem of
spinodal decomposition for phase separation \cite{SU2023107126} using the \textbf{Scheme I} and \textbf{Scheme III}. We set the computation domain  $\Omega$=$[0, 1]^{2}$. The parameters are $\nu$=$\mu$=$\sigma$=$M$=1, $\varepsilon$=$\lambda$=0.01, and $S$=$S^{\prime}$=$1/\varepsilon$. The initial values are given as  
\begin{equation}\label{5-2-1}
\begin{aligned}
&\textbf{v}^{0}=\textbf{0}, \quad p^{0}=0,\quad  \textbf{b}^{0}=\textbf{0},\quad \phi^{0}=\bar{\phi}+0.001 \text{rand}(r),
\end{aligned}
\end{equation}
where   $\bar{\phi}$=-0.05, and $\rm rand (r)$ is  a uniformly distributed random function in $[-1, 1]$ with zero mean.  We impose periodic boundary conditions on the phase field.  Choosing a space size of $h=\frac{1 }{64}$ and time sizes $\Delta t$=1, 0.1, 0.01,  0.001, 0.0001,  we display the curves of random energy and the mass of the phase field in Figure 1. Figure 1 (a) and (b) show the energy computed by \textbf{Scheme I} and  \textbf{Scheme III}, respectively, while (c) and (d) present the mass of the phase field computed by the  two schemes.
From  Figure 1 (a) and (b), we observe that all energy curves are dissipating, which is in accordance with the energy law.  The mass of the phase field is also conserved at different time scales, as illustrated in Figure 1 (c) and (d).

\begin{figure}[htbp]
	\centering
\subfigure[  ]{
		\begin{minipage}[t]{0.25\linewidth}
			\centering
			\includegraphics[width=\textwidth]{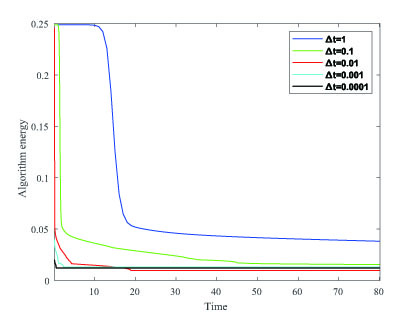}
		\end{minipage}
	}%
	\subfigure[  ]{
		\begin{minipage}[t]{0.25\linewidth}
			\centering
			\includegraphics[width=\textwidth]{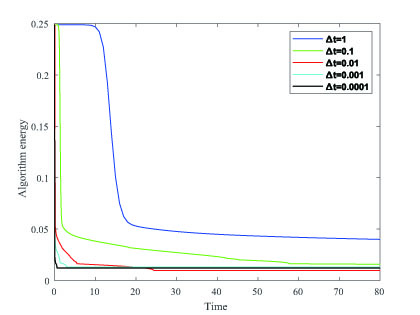}
		\end{minipage}
	}%
	\subfigure[  ]{
		\begin{minipage}[t]{0.25\linewidth}
			\centering
			\includegraphics[width=\textwidth]{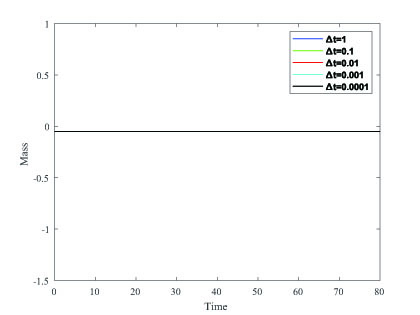}
		\end{minipage}
	}%
	\subfigure[  ]{
		\begin{minipage}[t]{0.25\linewidth}
			\centering
			\includegraphics[width=\textwidth]{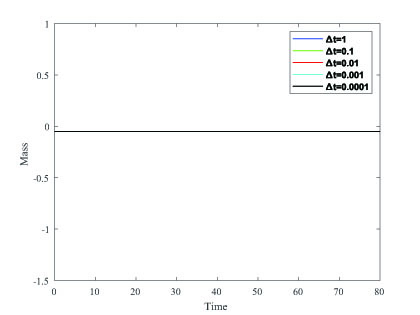}
		\end{minipage}
	}%
	\centering
	\caption{The algorithm energy of  Scheme I (a), Scheme III (b); The mass of phase field for Scheme I  (c) and Scheme III (d).}
	\label{fig3}
\end{figure}

Besides, we set $h=\frac{1}{ 128}$ and $\Delta t$=0.0001 to monitor the evolution of the phase field over time. The results are presented in Figures 2-3. We find that the snapshots of the numerical phase field  gradually coarsen, and the  evolutionary effects of the \textbf{Scheme I} and \textbf{Scheme III} are almost the same at the same time. 
 
\begin{figure}[htbp]
	\centering
	\subfigure[t=0.0001 ]{
		\begin{minipage}[t]{0.25\linewidth}
			\centering
			\includegraphics[width=\textwidth]{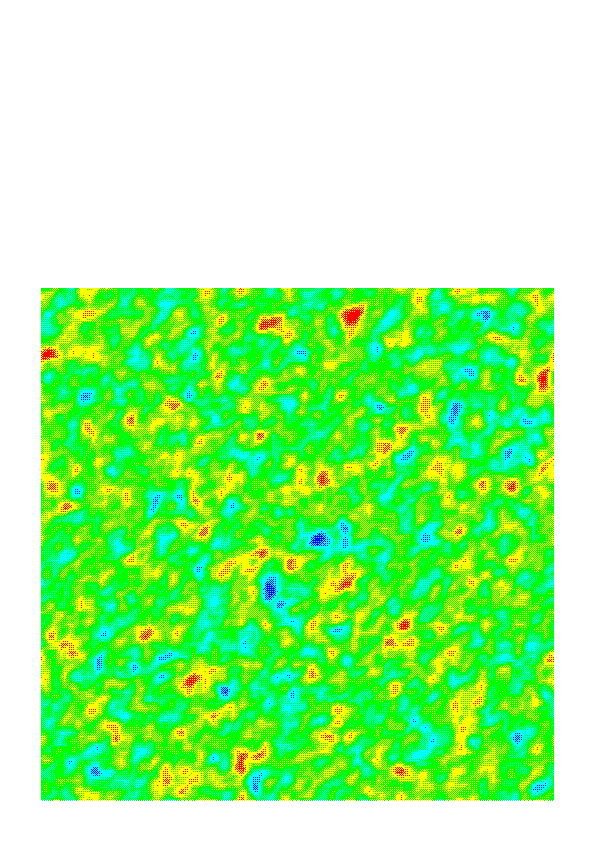}
		\end{minipage}%
	}%
	\subfigure[t=0.003 ]{
		\begin{minipage}[t]{0.25\linewidth}
			\centering
			\includegraphics[width=\textwidth]{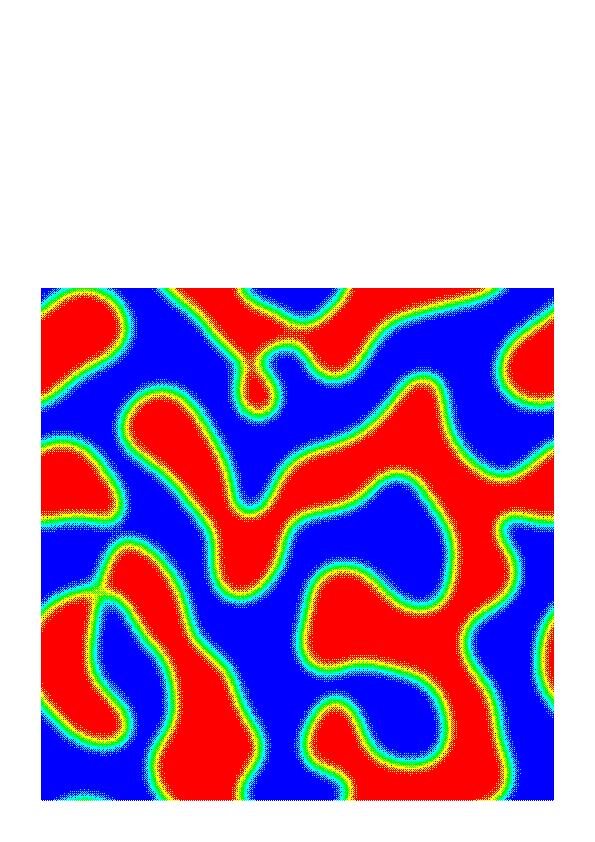}
		\end{minipage}
	}%
	\subfigure[t=0.05 ]{
		\begin{minipage}[t]{0.25\linewidth}
			\centering
			\includegraphics[width=\textwidth]{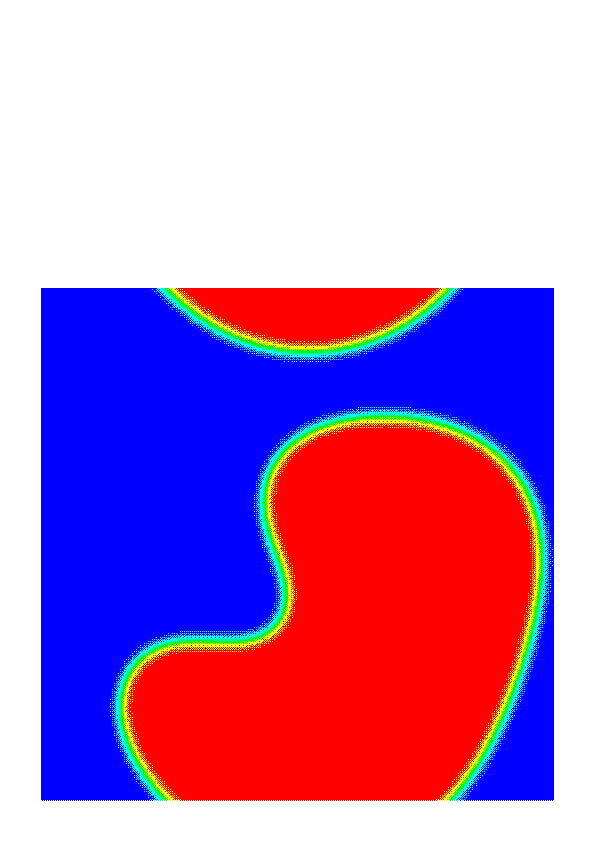}
		\end{minipage}
	}%
	\subfigure[t=0.5 ]{
		\begin{minipage}[t]{0.25\linewidth}
			\centering
			\includegraphics[width=\textwidth]{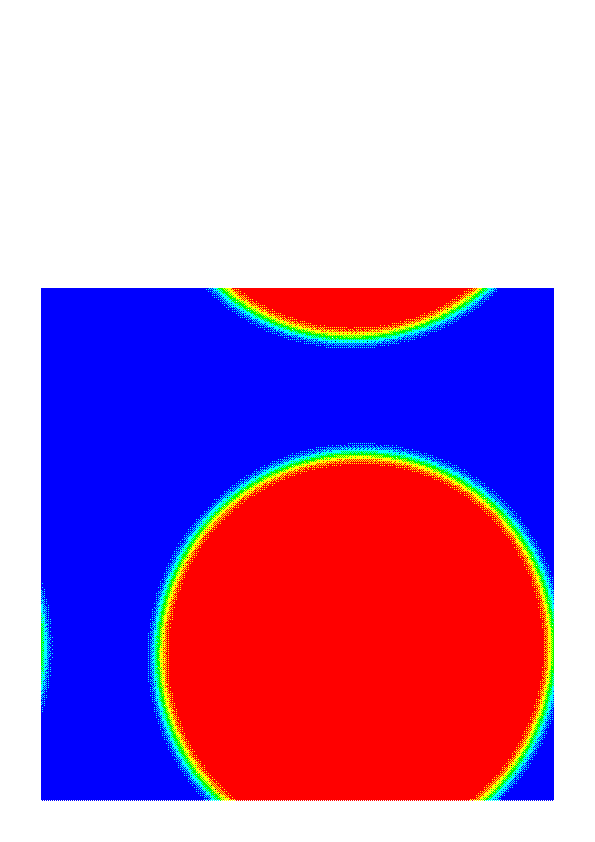}
		\end{minipage}
	}%
	\centering
	\caption{The dynamics of spinodal decomposition examples for  Scheme I with $\varepsilon$=0.01, $\lambda$=0.01$, \bar{\phi}$=-0.05, $\mu$=1.}
	\label{fig6}
\end{figure}

\begin{figure}[htbp]
	\centering
	\subfigure[t=0.0001 ]{
		\begin{minipage}[t]{0.25\linewidth}
			\centering
			\includegraphics[width=\textwidth]{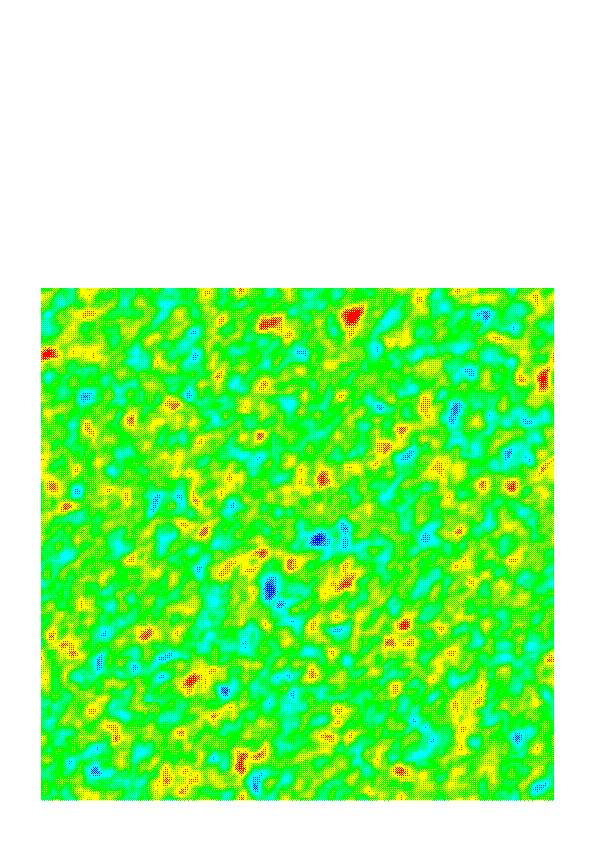}
		\end{minipage}%
	}%
	\subfigure[t=0.003 ]{
		\begin{minipage}[t]{0.25\linewidth}
			\centering
			\includegraphics[width=\textwidth]{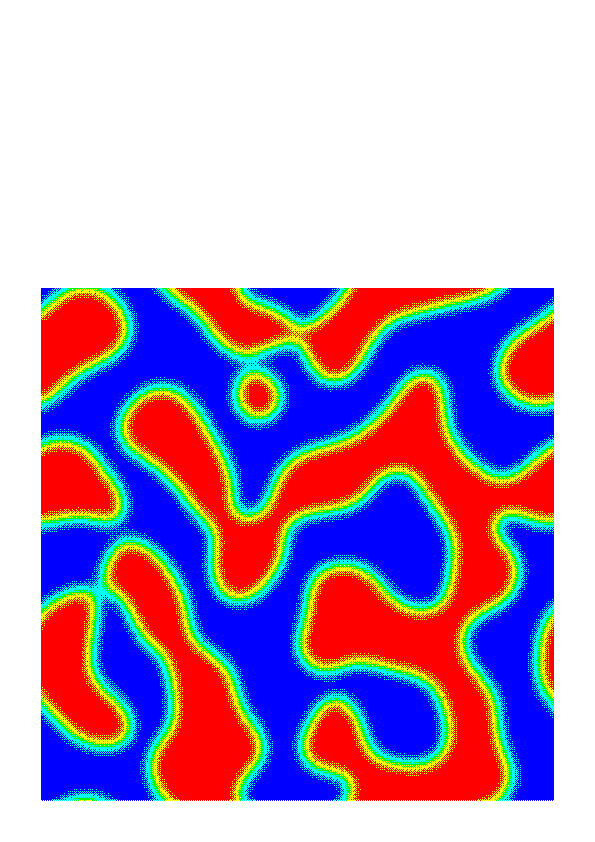}
		\end{minipage}
	}%
	\subfigure[t=0.05 ]{
		\begin{minipage}[t]{0.25\linewidth}
			\centering
			\includegraphics[width=\textwidth]{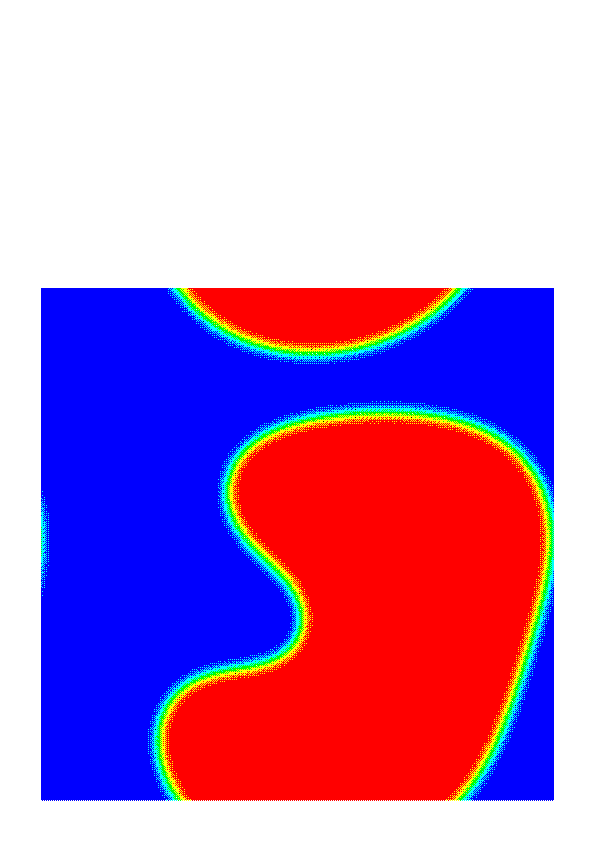}
		\end{minipage}
	}%
	\subfigure[t=0.5 ]{
		\begin{minipage}[t]{0.25\linewidth}
			\centering
			\includegraphics[width=\textwidth]{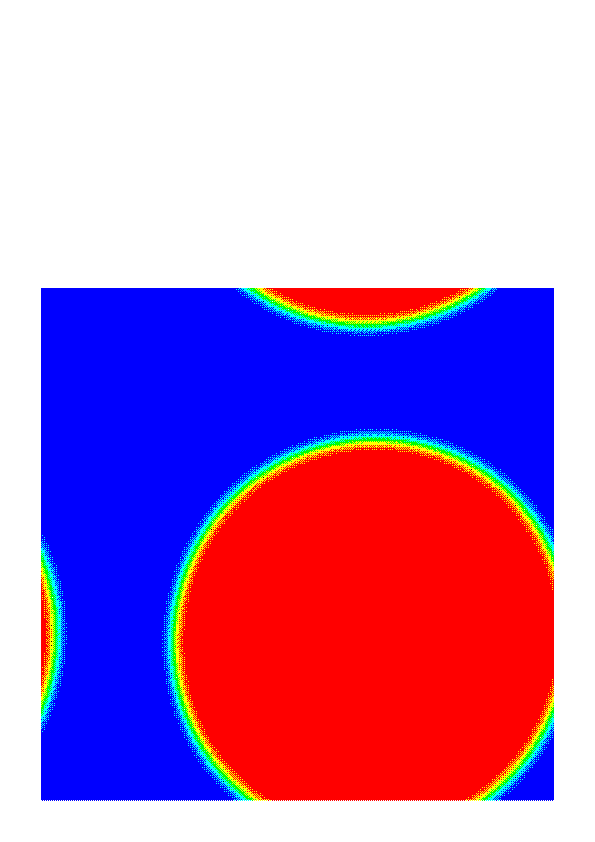}
		\end{minipage}
	}%
	\centering
	\caption{The dynamics of spinodal decomposition examples for Scheme III with   $\varepsilon$=0.01, $\lambda$=0.01$, \bar{\phi}$=-0.05, $\mu$=1.}
	\label{fig7}
\end{figure}

\subsection{Boussinesq approximation}

We investigate the Boussinesq approximation \cite{SU2023107126, 2007AnZHANG}, where the two fluids have a small density ratio.  Due to the similarity of the \textbf{Schemes I-III}, we only provide the evolutionary results of \textbf{Scheme I} in this part. We consider the domain  $\Omega$=$[0, 1]\times[0, 1.5]$,  the space scale size $h=\frac{1}{200}$, and the time scale size $\Delta t$=0.001. The momentum equations are reconstructed as  
\begin{equation*}
\begin{aligned}
&\rho_{0}(\textbf{v}_{t}+(\textbf{v}\cdot \nabla) \textbf{v})- \nu\Delta \textbf{v}+ \nabla p-\frac{1}{\mu}\nabla\times\textbf{b}\times\textbf{b}+ \lambda\phi\nabla w=-(1+\phi)\textbf{g}(\rho_{1}-\rho_{0})-(1-\phi)\textbf{g}(\rho_{2}-\rho_{0}),
\end{aligned}
\end{equation*}
where $\rho_{1}$=1 and $\rho_{2}$=9 represent the densities of two immiscible, incompressible fluids.   The ``background'' density is defined by $\rho_{0}$=$(\rho_{1}+\rho_{2})/2$, and the gravitational constant vector $\textbf{g}$=$[0, 10]^{T}$. The  parameters are given as
\begin{equation*}
\nu_{1}=\nu_{2}=1, \quad \sigma_{1}=300,\quad \sigma_{2}=400,\quad M=10^{-4},\quad \varepsilon=0.01, \quad \lambda=5,\quad S=1,
\end{equation*}
where $\nu$=$(\nu_{1}, \nu_{2})$ and $\sigma$=($\sigma_{1}, \sigma_{2}$). The boundary conditions are set  as
\begin{equation*}
\left\{
\begin{aligned}
&\frac{\partial\phi}{\partial \textbf{n}}|_{\partial\Omega}=0, \frac{\partial w}{\partial \textbf{n}}|_{\partial\Omega}=0,\\
& \textbf{v}|_{y=0,\, 1.5}=\textbf{0},\\
& v_{1} =0 \,  \text{otherwise},\\
&\textbf{n}\times \textbf{b}|_{\partial\Omega}=\textbf{n}\times (0, 1)^{T}|_{\partial\Omega}.
\end{aligned}
\right.
\end{equation*}

\begin{figure}[htbp]
	\centering
	\subfigure[t=0.01 ]{
		\begin{minipage}[t]{0.2\linewidth}
			\centering
			\includegraphics[width=\textwidth]{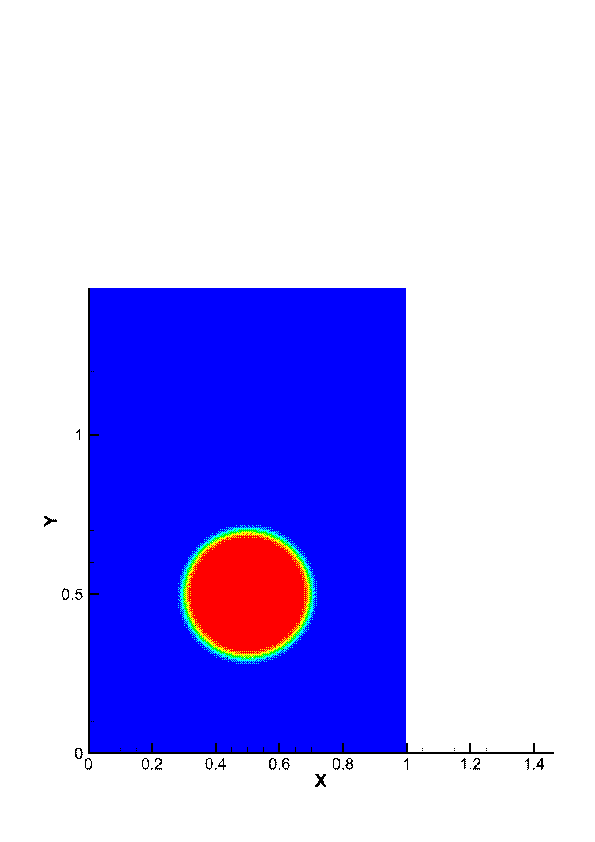}
		\end{minipage}%
	}%
	\subfigure[t=0.5 ]{
		\begin{minipage}[t]{0.2\linewidth}
			\centering
			\includegraphics[width=\textwidth]{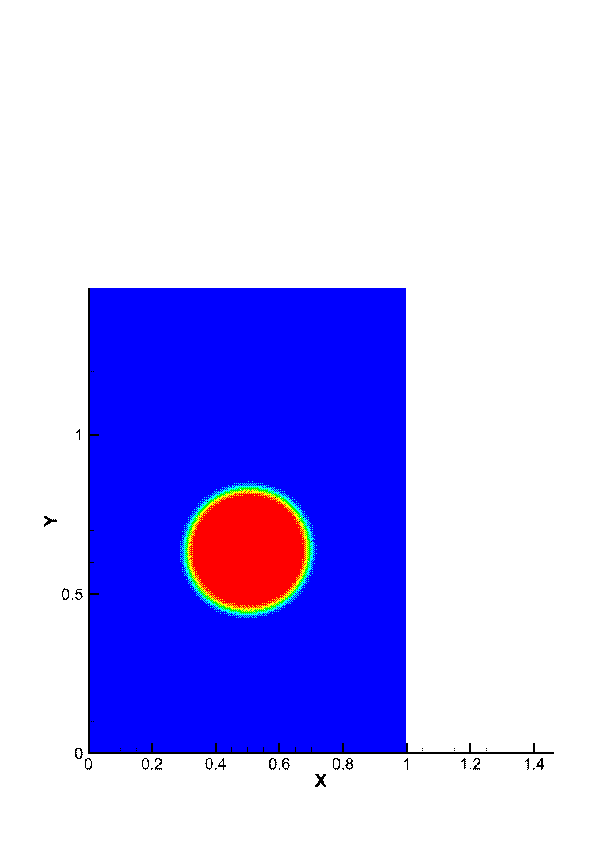}
		\end{minipage}
	}%
	\subfigure[t=1 ]{
		\begin{minipage}[t]{0.2\linewidth}
			\centering
			\includegraphics[width=\textwidth]{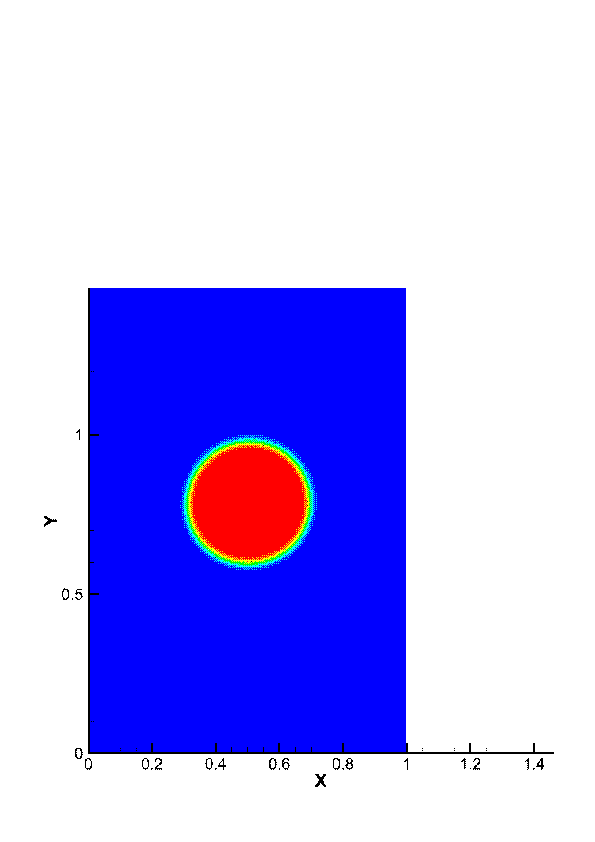}
		\end{minipage}
	}%
	\subfigure[t=2 ]{
		\begin{minipage}[t]{0.2\linewidth}
			\centering
			\includegraphics[width=\textwidth]{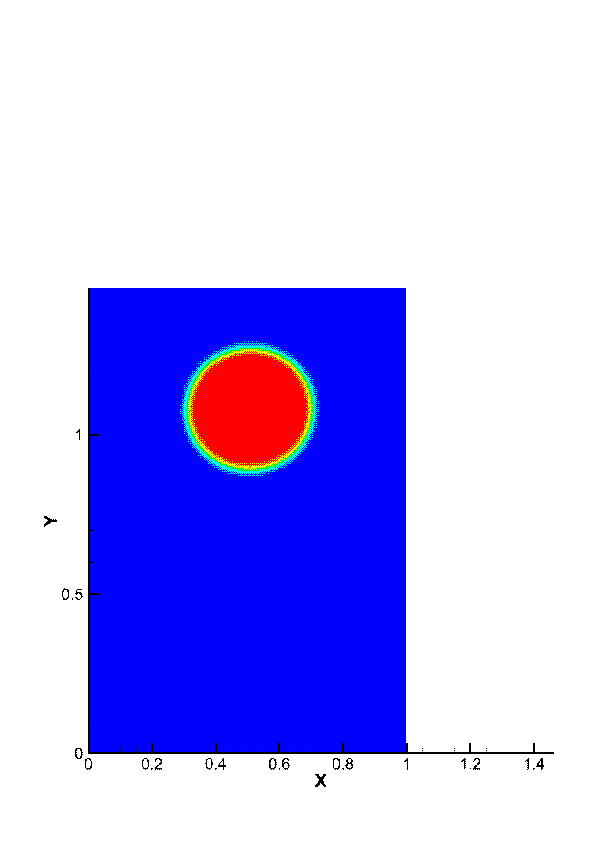}
		\end{minipage}
	}%
   \subfigure[t=3 ]{
		\begin{minipage}[t]{0.2\linewidth}
			\centering
			\includegraphics[width=\textwidth]{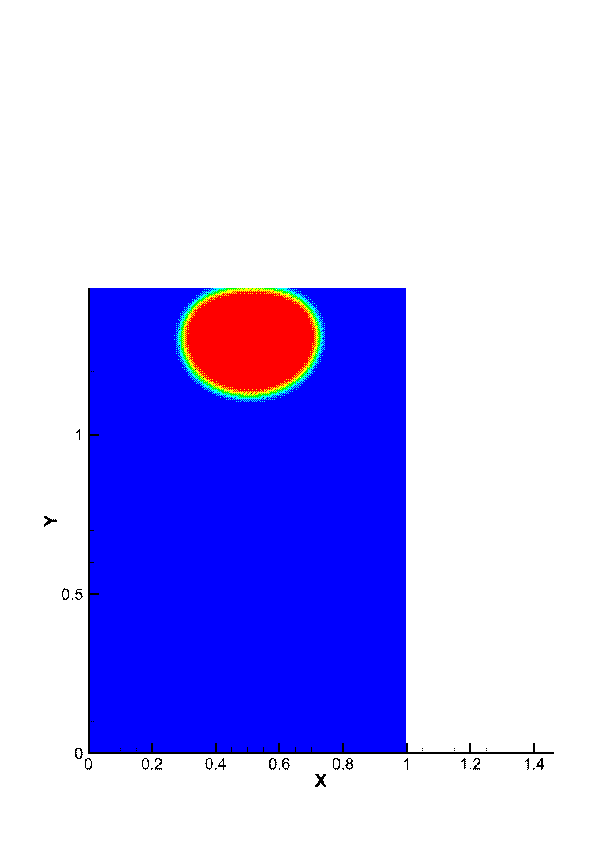}
		\end{minipage}
	}%
	\centering
	\caption{Snapshots of phase field without the Lorentz force term.}
	\label{fig13}
\end{figure}

\begin{figure}[htbp]
	\centering
	\subfigure[t=0.01 ]{
		\begin{minipage}[t]{0.2\linewidth}
			\centering
			\includegraphics[width=\textwidth]{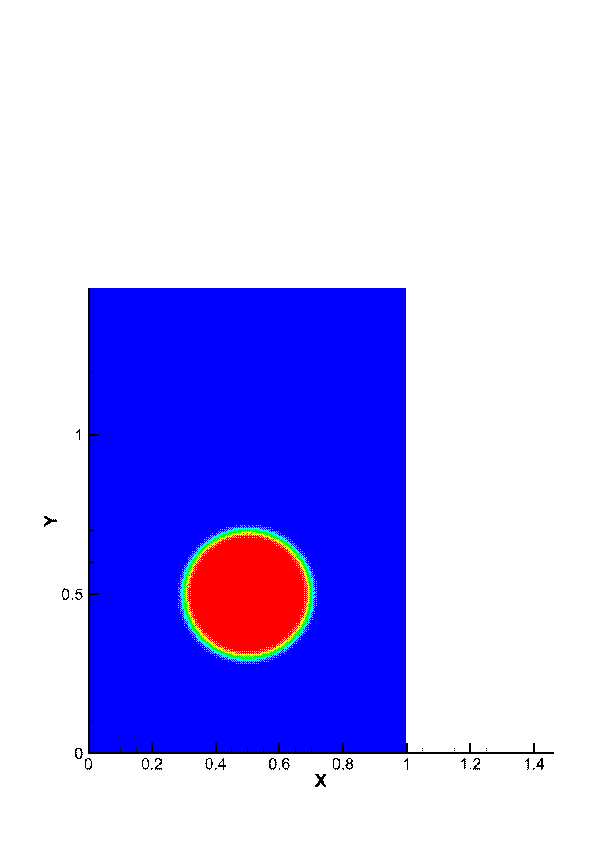}
		\end{minipage}%
	}%
	\subfigure[t=0.5 ]{
		\begin{minipage}[t]{0.2\linewidth}
			\centering
			\includegraphics[width=\textwidth]{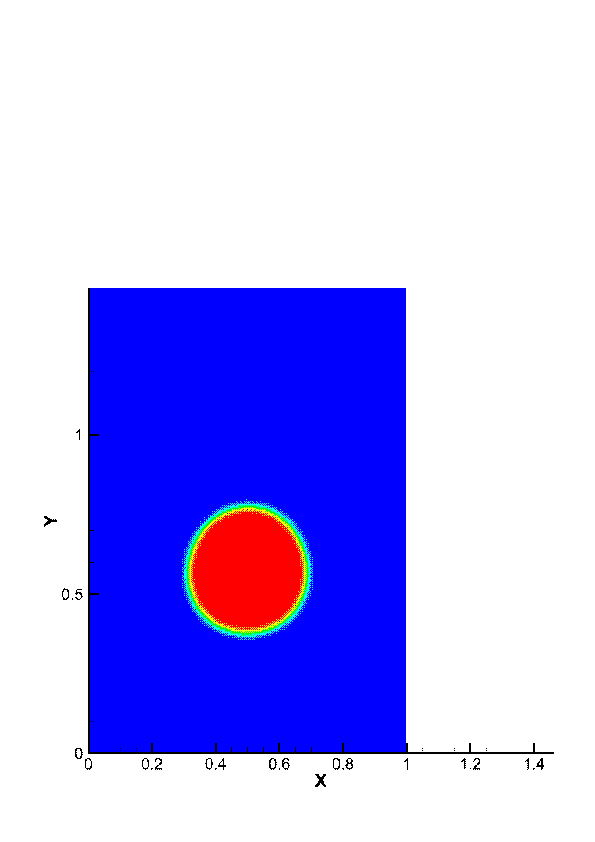}
		\end{minipage}
	}%
	\subfigure[t=1 ]{
		\begin{minipage}[t]{0.2\linewidth}
			\centering
			\includegraphics[width=\textwidth]{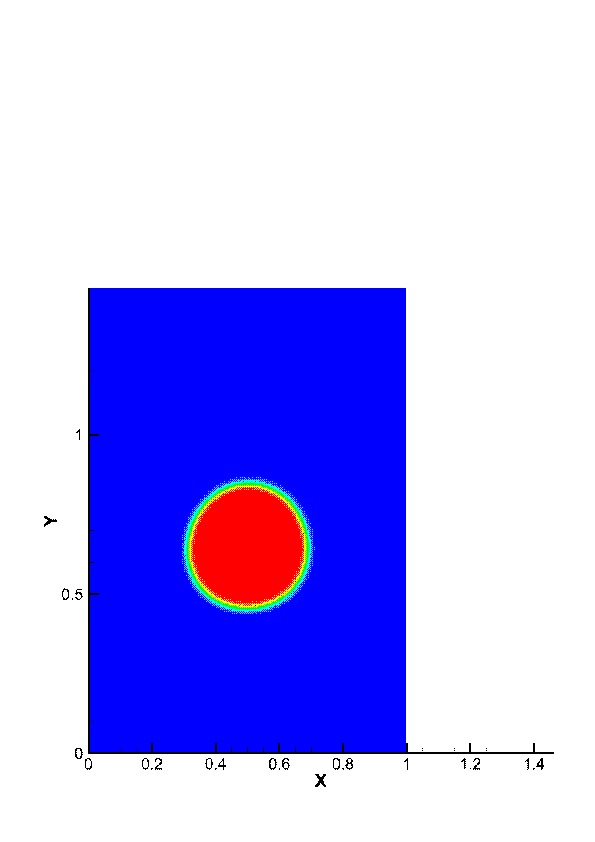}
		\end{minipage}
	}%
	\subfigure[t=2 ]{
		\begin{minipage}[t]{0.2\linewidth}
			\centering
			\includegraphics[width=\textwidth]{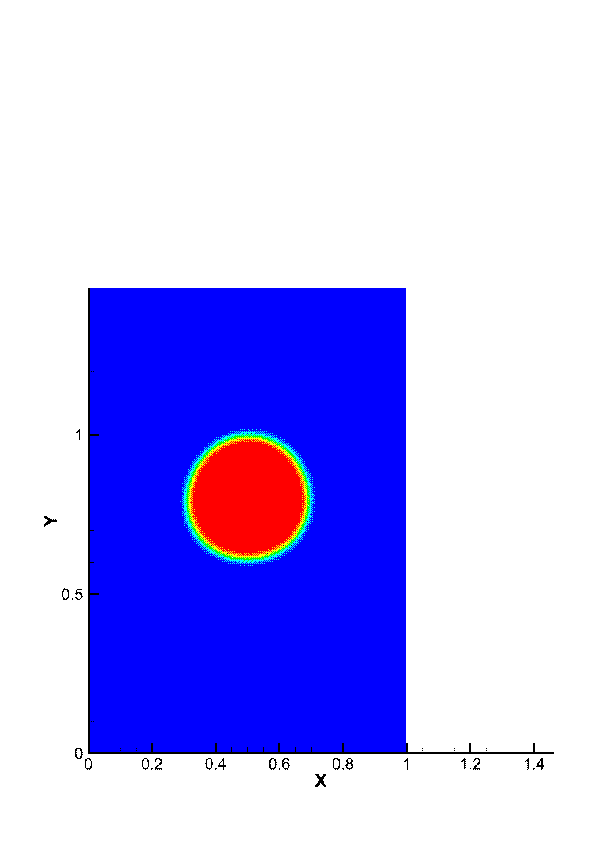}
		\end{minipage}
	}%
   \subfigure[t=3 ]{
		\begin{minipage}[t]{0.2\linewidth}
			\centering
			\includegraphics[width=\textwidth]{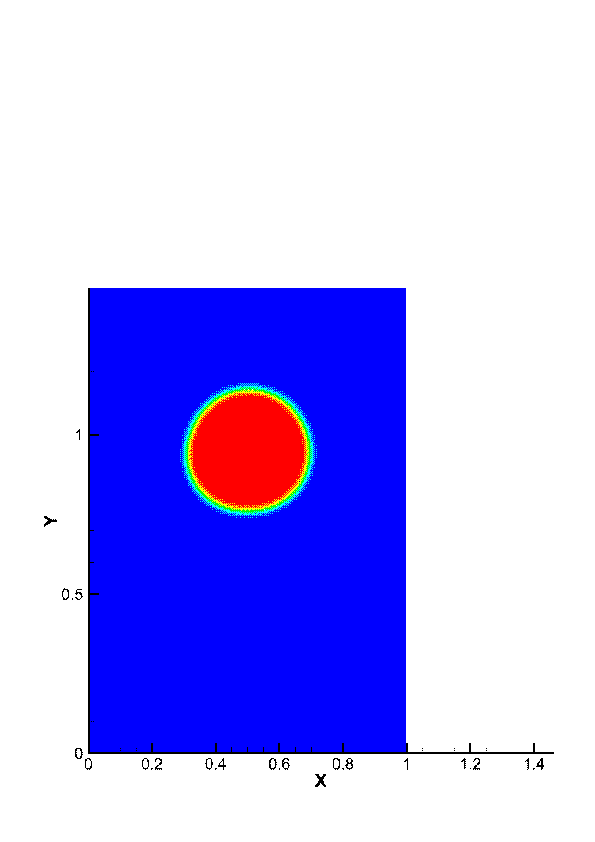}
		\end{minipage}
	}%
	\centering
	\caption{Snapshots of phase field with  $\mu$=0.001.}
	\label{fig13}
\end{figure}
Here, we explore  the influence of the Lorentz force on evolution. The numerical results without  Lorentz force and with Lorentz force (i.e. $\mu$=0.001) are shown in Figure 4 and Figure 5, respectively. As the magnetic permeability decreases (Lorentz force increases), we observe that the rising speed of bubbles slows down, indicating that the Lorentz force suppresses their buoyancy.

\section{Conclusions} 
In this paper, we propose two linear, fully-decoupled, and unconditionally energy-stable semi-discrete
schemes for two-phase MHD model. The two methods consist of the semi-implicit stabilization method and the invariant energy IEQ method, which are both applied to the phase field system. The pressure correction method is designed for the saddle point system, and appropriate implicit-explicit treatments are employed for the nonlinear coupled terms.  For both schemes, we strictly establish unconditional energy-stable and  error estimates
without any time step and mesh size constraint in 2D/3D cases. Specifically, to obtain the
error estimates of the fully-decoupled semi-discrete schemes, we derive the bound for $\left\|\phi^{k}\right\|_{L^{\infty}}$ and $\left\|\textbf{b}^{k}\right\|_{L^{\infty}}$. Based on the results of $\left\|\phi^{k}\right\|_{L^{\infty}}$ and $\left\|\textbf{b}^{k}\right\|_{L^{\infty}}$, we obtain the error estimates under the condition $\Delta t\leq C$, and the convergence results are derived through the stability results for $\Delta t\geq C$.  Furthermore, several numerical experiments are conducted to evaluate the stability and accuracy of the proposed schemes.

\section*{Statements and Declarations}
Competing Interests: the authors have no relevant financial or non-financial interests to disclose.

Data Availability: all data generated or analyzed during this study are included in this article.

This work is partly supported by  the NSF of China (No. 12126361) support.
\newpage
\bibliographystyle{Ieeetr}

\bibliography{CHMHD-reference}

\begin{thebibliography}{10}

\bibitem{2000Liquid}
N.~Morley, S.~Smolentsev, L.~Barleon, I.~Kirillov, and M.~Takahashi, ``Liquid
  magnetohydrodynamics-recent progress and future directions for fusion,'' {\em
  FUSION ENG DES}, vol.~51, pp.~701--713, 2000.

\bibitem{2016Numerical}
A.~Hadidi and D.~Jalali-Vahid, ``Numerical simulation of dielectric bubbles
  coalescence under the effects of uniform magnetic field,'' {\em THEOR COMP
  FLUID DYN}, vol.~30, no.~3, pp.~165--184, 2016.

\bibitem{2003ALC}
C.~Liu and J.~Shen, ``A phase field model for the mixture of two incompressible
  fluids and its approximation by a {F}ourier-spectral method,'' {\em PHYSICA
  D}, vol.~179, no.~3, pp.~211--228, 2003.

\bibitem{2010Level}
H.~Ki, ``Level set method for two-phase incompressible flows under magnetic
  fields,'' {\em COMPUT PHYS COMMUN}, vol.~181, no.~6, pp.~999--1007, 2010.

\bibitem{2012Effect}
M.~R. Ansari, A.~Hadidi, and M.~E. Nimvari, ``Effect of a uniform magnetic
  field on dielectric two-phase bubbly flows using the level set method,'' {\em
  J MAGN MAGN MATER}, vol.~324, no.~23, pp.~4094--4101, 2012.

\bibitem{1964EFFECT}
R.~J. Thome, ``Effect of a transverse magnetic field on vertical two-phase flow
  through a rectangular channel,'' {\em Argonne {N}ational {L}aboratory
  {R}eport}, vol.~ANL-6854, 1964.

\bibitem{2022Highly}
H.~Su and G.~Zhang, ``Highly efficient and energy stable schemes for the
  2{D}/3{D} diffuse interface model of two-phase magnetohydrodynamics,'' {\em J
  SCI COMPUT}, vol.~90, no.~63, 2022.

\bibitem{2021SharpZX}
X.~Zhang, ``Sharp-interface limits of the diffuse interface model for two-phase
  inductionless magnetohydrodynamic fluids,'' {\em arXiv.2106.10433}, 2021.

\bibitem{2019AYang}
J.~Yang, S.~Mao, X.~He, X.~Yang, and Y.~He, ``A diffuse interface model and
  semi-implicit energy stable finite element method for two-phase
  magnetohydrodynamic flows,'' {\em COMPUT METHOD APPL M}, vol.~356,
  pp.~435--464, 2019.

\bibitem{zhang2023gauge}
J.~Zhang, H.~Su, and X.~Feng, ``Gauge-{U}zawa-based, highly efficient decoupled
  schemes for the diffuse interface model of two-phase magnetohydrodynamic,''
  {\em COMMUN NONLINEAR SCI}, vol.~126, p.~107477, 2023.

\bibitem{SU2023107126}
H.~Su and G.~Zhang, ``Energy stable schemes with second order temporal accuracy
  and decoupled structure for diffuse interface model of two-phase
  magnetohydrodynamics,'' {\em COMMUN NONLINEAR SCI}, vol.~119, p.~107126,
  2023.

\bibitem{AAMM-13-761}
J.~Zhao, R.~Chen, and H.~Su, ``An energy-stable finite element method for
  incompressible {M}agnetohydrodynamic-{C}ahn-{H}illiard coupled model,'' {\em
  ADV APPL MATH MECH}, vol.~13, no.~4, pp.~761--790, 2021.

\bibitem{2022Unconditional}
C.~Chen and T.~Zhang, ``Unconditional stability and optimal error estimates of
  first order semi-implicit stabilized finite element method for two phase
  magnetohydrodynamic diffuse interface model,'' {\em APPL MATH COMPUT},
  no.~429, p.~127238, 2022.

\bibitem{2025Error}
X.~Shen and Y.~Cai, ``Error estimates of time discretizations for a
  {C}ahn-{H}illiard phase-field model for the two-phase magnetohydrodynamic
  flows,'' {\em APPL NUMER MATH}, vol.~207, pp.~585--607, 2025.

\bibitem{2024A}
D.~Wang, Y.~Guo, F.~Liu, H.~Jia, and C.~Zhang, ``A fully decoupled linearized
  and second-order accurate numerical scheme for two-phase magnetohydrodynamic
  flows,'' {\em INT J NUMER METH FL}, vol.~96, no.~4, pp.~482--509, 2024.

\bibitem{2023ErrorQIU}
H.~Qiu, ``Error analysis of fully discrete scheme for the
  {C}ahn-{H}illiard-{M}agneto-{H}ydrodynamics problem,'' {\em J SCI COMPUT},
  vol.~95, no.~16, 2023.

\bibitem{0Convergence}
C.~Wang, J.~Wang, S.~Wise, Z.~Xia, and L.~Xu, ``Convergence analysis of a
  temporally second-order accurate finite element scheme for the
  {C}ahn-{H}illiard-{M}agnetohydrodynamics system of equations,'' {\em J COMPUT
  APPL MATH}, vol.~436, p.~115409, 2024.

\bibitem{2025Unconditionally}
J.~Yang, S.~Mao, and X.~He, ``Unconditionally optimal convergent
  zero-energy-contribution scheme for two phase mhd model,'' {\em J SCI
  COMPUT}, vol.~102, no.~55, pp.~1--72, 2025.

\bibitem{2019FullyZGD}
G.~Zhang, X.~He, and X.~Yang, ``Fully decoupled, linear and unconditionally
  energy stable time discretization scheme for solving the magneto-hydrodynamic
  equations,'' {\em J COMPUT APPL MATH}, vol.~369, p.~112636, 2019.

\bibitem{2020Convergence}
X.~Yang and G.~Zhang, ``Convergence analysis for the invariant energy
  quadratization ({IEQ}) schemes for solving the {C}ahn-{H}illiard and
  {A}llen-{C}ahn equations with general nonlinear potential,'' {\em J SCI
  COMPUT}, vol.~82, no.~55, pp.~1--28, 2020.

\bibitem{1958Free}
J.~W. Cahn and J.~E. Hilliard, ``Free energy of a nonuniform system. {I}.
  interfacial free energy,'' {\em J CHEM PHYS}, vol.~28, no.~2, p.~258, 1958.

\bibitem{2010Energy}
J.~Shen and X.~Yang, ``Energy stable schemes for {C}ahn-{H}illiard phase-field
  model of two-phase incompressible flows,'' {\em CHINESE ANN MATH B}, vol.~31,
  no.~5, pp.~743--758, 2010.

\bibitem{2019The}
A.~Miranville, {\em The {C}ahn-{H}illiard Equation: Recent Advances and
  Applications}.
\newblock 2019.

\bibitem{1975Sobolev}
R.~A. Adams, ``Sobolev spaces academic press,'' 1975.

\bibitem{2019Stability}
F.~Cheng and J.~Shen, ``Stability and convergence analysis of rotational
  velocity correction methods for the {N}avier-{S}tokes equations,'' {\em ADV
  COMPUT MATH}, vol.~45, no.~5, pp.~3123--3136, 2019.

\bibitem{2021EfficientEF}
X.~Yang, ``Efficient and energy stable scheme for the hydrodynamically coupled
  three components {C}ahn-{H}illiard phase-field model using the
  stabilized-invariant energy quadratization ({S-IEQ}) approach,'' {\em J
  COMPUT PHYS}, no.~438, p.~110342, 2021.

\bibitem{1990Finite}
J.~G. Heywood and R.~Rannacher, ``Finite-element approximation of the
  nonstationary {N}avier-{S}tokes problem part {IV}: Error analysis for
  second-order time discretization,'' {\em SIAM J NUMER ANAL}, vol.~27, no.~2,
  pp.~353--384, 1990.

\bibitem{2020ErrorXU}
Z.~Xu, X.~Yang, and H.~Zhang, ``Error analysis of a decoupled, linear
  stabilization scheme for the {C}ahn-{H}illiard model of two-phase
  incompressible flows,'' {\em J SCI COMPUT}, vol.~83, no.~3, pp.~1--27, 2020.

\bibitem{2019Fast}
C.~Chen and X.~Yang, ``Fast, provably unconditionally energy stable, and
  second-order accurate algorithms for the anisotropic {C}ahn-{H}illiard
  model,'' {\em COMPUT METHOD APPL M}, vol.~351, pp.~35--59, 2019.

\bibitem{2019EfficientLIU}
Z.~Liu and X.~Li, ``Efficient modified techniques of invariant energy
  quadratization approach for gradient flows,'' {\em APPL MATH LETT}, vol.~98,
  pp.~206--214, 2019.

\bibitem{CHEN2022114405}
R.~Chen and S.~Gu, ``On novel linear schemes for the {C}ahn-{H}illiard equation
  based on an improved invariant energy quadratization approach,'' {\em J
  COMPUT APPL MATH}, vol.~414, p.~114405, 2022.

\bibitem{2007AnZHANG}
Z.~Zhang and H.~Tang, ``An adaptive phase field method for the mixture of two
  incompressible fluids,'' {\em COMPUT FLUIDS}, vol.~36, no.~8, pp.~1307--1318,
  2007.

\end{thebibliography}
\end{document}